\documentclass[reqno]{amsart}

\usepackage[foot]{amsaddr}

\usepackage{graphicx}
\usepackage{comment}
\usepackage[a4paper,hmargin=2cm,vmargin=3cm]{geometry}
\usepackage[T1]{fontenc}
\usepackage[utf8]{inputenc}
\usepackage{amsmath}
\usepackage{amssymb}
\usepackage{mathtools}
\usepackage{bm}
\usepackage{bbm}
\usepackage{enumitem}
\usepackage{hyperref}
\usepackage[normalem]{ulem}
\usepackage{bbm}
\usepackage{epstopdf}
\usepackage{stmaryrd}
\usepackage{esint}
\usepackage{mathrsfs} 
\usepackage[english]{babel}
\usepackage{amssymb}
\usepackage{color}
\usepackage{amsfonts}
\usepackage{amsthm}
\usepackage{subfigure}
\usepackage{afterpage}

\usepackage{placeins} 

\makeatletter

\usepackage{color}
\usepackage[dvipsnames]{xcolor}
\usepackage{pgfplots,pgfmath}
\pgfplotsset{compat=1.18}
\usepgfplotslibrary{colorbrewer}

\usepackage{caption}
\usepackage{wrapfig}
\usepackage{tikz}
\usetikzlibrary{decorations.markings,fpu}

\usepackage{pgfplotstable}

\definecolor{notefontcolor}{rgb}{0.800781, 0.800781, 0.800781}
\definecolor{grey30}{rgb}{0.7,0.7,0.7}

\renewcommand{\eqref}[1]{\textup{\tagform@{\ref{#1}}}}

\numberwithin{equation}{section}

\theoremstyle{plain}
\newtheorem{theorem}{Theorem}[section]
\newtheorem{lemma}[theorem]{Lemma}
\newtheorem{proposition}[theorem]{Proposition}
\newtheorem{corollary}[theorem]{Corollary}

\theoremstyle{remark}
\newtheorem{remark}[theorem]{Remark}

\renewcommand{\Im}{\textup{Im}}
\renewcommand{\Re}{\textup{Re}}

\DeclareMathOperator{\MP}{MP}

\DeclareMathOperator{\argsup}{argsup}

\newcommand{\musc}[1][]{ {{\mu_{{\rm sc}\ifx#1\empty\else,#1\fi}}} }
\newcommand{\muMP}{ {{\mu_{\MP,\a}}} }

\newcommand{\e}{\e}

\newcommand{\R}{{\mathbb{R}}}

\newcommand{\dom}{{\rm dom}}
\newcommand{\sign}{{\rm sign}}

\newcommand{\muplusnu}{{\mu\boxplus\nu}}

\renewcommand {\(}{\left(}

\renewcommand {\)}{\right)}

\def\be{\begin{equation}}
\def\ee{\end{equation}}
\def\bea{\begin{eqnarray}}
\def\eea{\end{eqnarray}}

\def\been#1{\begin{equation}#1\end{equation}}

\DeclareMathSymbol{\leqlant}{\mathalpha}{AMSa}{"36} 
\DeclareMathSymbol{\geqslant}{\mathalpha}{AMSa}{"3E} 
\DeclareMathSymbol{\eset}{\mathalpha}{AMSb}{"3F}     
\renewcommand{\leq}{\:\leqlant\:}                   
\renewcommand{\geq}{\:\geqslant\:}                   

\def\nn{\nonumber}
\def\a{\beta}
\def\e{\varepsilon}
\def\d{\delta}

\def\b{\beta}

\def\l{\lambda}

\def\t{\tau}

\def\R{\mathbb{R}}

\DeclareMathOperator{\supp}{supp}

\definecolor{light}{gray}{.9}

\definecolor{mypink1}{rgb}{0.858, 0.188, 0.478}
\definecolor{g3col1}{rgb}{0.0, 0.8, 0.0}

\tikzset{every node/.style={font=\footnotesize}}
\pgfplotsset{
    alignedplot/.style={
        width=3.2cm,
        height=3.2cm,
        grid=none,
        scale only axis,
        tick label style={font=\footnotesize},
        clip=true
    },
    fourcolwidth/.style={
        width=3.1cm,
        height=3.1cm,
        clip=true
    },
}

\title[logarithmic potential of free convolutions]{Variational formula for the logarithmic potential of free additive convolutions}

\author{David Belius$^{1,*}$} 
\email{david.belius@cantab.net}

\author{Francesco Concetti$^{1,*}$}
\email{concio2@gmail.com}

\thanks{$^*$ D. B. and F. C. are supported by the SNSF grants 176918 and 206148. G.G. is partially supported by the NSERC Discovery Grant RGPIN-2025-04930.}
\address{$^1$Faculty of Mathematics and Computer Science, UniDistance Suisse, 3900 Brig, Switzerland}

\author{Giuseppe Genovese$^{2,*}$}
\email{giuseppe.genovese@math.ubc.ca}

\address{$^2$Dep. of Mathematics, University of British Columbia, 1984 Mathematics Road V6T 1Z2 Vancouver BC,  Canada}

\begin{document}

\begin{abstract}
    We establish a general variational formula for the logarithmic potential of the free additive convolution of two compactly supported probability measure on $\R$. The formula is given in terms of the $R$-transform of the first measure, and the logarithmic potential of second measure. The result applies in particular to the additive convolution with the semicircle or Marchenko-Pastur laws, for which the formula simplifies. The logarithmic potential of additive convolutions appears for instance in estimates of the determinant of sums of independent random matrices.
    {\bf MSC:} 46L54, 31A10, 60B20.
\end{abstract}
\maketitle

\section{Introduction}\label{sect:intro}

In this work we derive a formula for the logarithmic potential associated to the free additive convolution of two probability measures on $\R$ with compact support. The logarithmic potential of a probability measure $\mu$ is
\be\label{eq:log-P}
    U_\mu(z):=\int \log|z-\l|\mu(d\l),\quad z\in\mathbb{C}.
\ee
This natural quantity in planar potential theory appears in a number of different contexts.
Our interest in $U_\mu$
is due to its appearance in formulas for the exponential-scale asymptotics of determinants of the sums of independent random matrices. Such determinant asymptotics are crucial for studying high-dimensional random landscapes using the Kac-Rice formula, for instance in the context of spin glasses \cite{noi, ben, mon} or machine learning theory \cite{AMS,BKMN,MBAB,TM}. In such applications, one of the measures involved will often be the semicircle law or the Marchenko-Pastur law, for which our formula specializes to \eqref{eq:special_case_sc}-\eqref{eq:special_case_MP} below.



The Stieltjes transform of a probability measure $\mu$ on $\mathbb{R}$ is 
\be\label{eq:defSti}
    G_\mu(z) := \int\frac{\mu(d\l)}{\l-z},
    \quad
    z \in \mathbb{C}.
\ee
This is a holomorphic function in $\mathbb{C}$  with the support of $\mu$ removed.
The $R$-transform is defined by the identity
\newcommand{\inv}{ { {\left\langle -1\right\rangle} } }
\be\label{eq:R-transf}
    R_\mu(t)
    :=
    G_\mu^{\inv}(-t) - t^{-1}
\ee
of formal power series, where the superscript $\left\langle -1\right\rangle$ denotes the inverse.
Given two probability measures $\mu,\nu$, there is a unique probability measure called the free additive convolution of $\mu,\nu$, denoted by $\mu\boxplus\nu$, such that
\be\label{eq:sumR}
    R_{\mu\boxplus\nu}(t) = R_\mu(t) + R_\nu(t),
\ee
(see \cite[Theorem 18, Chapter 2]{speicher}). If $\mu$ and $\nu$ have compact support, then so does $\mu\boxplus\nu$ (see 
\eqref{eq:add_conv_supp_ineq}).
Let $\supp_{\pm}\rho$ denote the inf. resp. sup. of the support of a probability $\rho$ on $\mathbb{R}$. For any $\mu,\nu$ with compact support let
\begin{equation}\label{eq:zstar_gstar_def}
    z_{\mu \boxplus \nu}^* = \supp_{-}\mu\boxplus\nu,
    \quad\quad\quad\text{and}\quad\quad\quad
    g^*_{\mu \boxplus \nu}
    =
    G_{\mu\boxplus\nu}(z_{\mu \boxplus \nu}^*),
\end{equation}
where the latter denotes $\lim_{z \uparrow z_{\mu \boxplus \nu}^*}G_{\mu\boxplus\nu}(z) \in (0,\infty]$.

\subsection{Result for semicircle and Marchenko-Pastur laws}

When $\mu=\musc[\beta]$ is the semicircle law supported on $[-2\beta,2\beta]$ (see \eqref{eq:sc_def}), and $\nu$ is any compactly supported probability on $\mathbb{R}$,
our variational formula for the logarithmic potential $U_{\mu\boxplus\nu}$ reads
\begin{equation}\label{eq:special_case_sc}
    U_{\mu\boxplus\nu}(z)
    =
    \inf_{
        g \in (0,g^*_{\mu \boxplus \nu})
        :
        \beta^2 g < \supp_{-}\nu - z
    }
    \left\{
        \frac{\beta^2 g^2}{2}
        +
        \int\log( \lambda - z - \beta^2 g)\nu(d\lambda)
    \right\}
    \, \qquad \forall z<z_{\mu \boxplus \nu}^*.
\end{equation}

When $\mu=\muMP$ is the Marchenko-Pastur law with parameter $\beta>0$ (see \eqref{eq:MP}), our variational formula reads
\begin{equation}\label{eq:special_case_MP}
    U_{\mu\boxplus\nu}(z)
    =
    \inf_{
        g \in (0,g_{\mu \boxplus \nu}^*)
     : z-\supp_{-}\nu < \frac{\beta}{1+g}
    }
    \left\{ 
        \beta \log(1+g)-\frac{\beta g}{1+g}
        +
        \int\log\left(\lambda-z+\frac{\beta}{1+g}\right)\nu(d\lambda)
    \right\}
    \,\,\,\,\,
    \forall z<z_{\mu \boxplus \nu}^*.
\end{equation}

Furthermore, letting $E_{\musc,\nu,z}(g)$ resp. $E_{\muMP,\nu,z}(g)$ denote the expression in the brackets of \eqref{eq:special_case_sc} resp. 
\eqref{eq:special_case_MP}, it holds for both $\mu = {\musc[\beta]}$ and $\mu = \muMP$ that:
\begin{equation}\label{eq:special_case_sc_MP_E_crit_points}
    \begin{array}{cl}
        z>z_{\mu\boxplus\nu}^*\implies & E_{\mu,\nu,z}\text{ has no crit. points.}\\
        z=z_{\mu\boxplus\nu}^*\implies & E_{\mu,\nu,z}\text{ has a unique crit. point at } g = g^*_{\mu\boxplus\nu},\text{ which is an inflection point.}\\
        z<z_{\mu\boxplus\nu}^*\implies & E_{\mu,\nu,z}\text{ has a unique loc. min. at }g=G_{\mu\boxplus\nu}(z)\text{, and at most one additional}\\
         & \text{crit. point, which if present is a loc. max. in }(G_{\mu\boxplus\nu}(z),\infty) \text{ (see Figure \ref{fig:intro_E_mu_nu_z_F_mu_nu_shapes1} a,b).}
    \end{array}
\end{equation}
The implications in \eqref{eq:special_case_sc_MP_E_crit_points} can be used to ``read off'' the quantities 
$z_{\mu \boxplus \nu}^*,g^*_{\mu \boxplus \nu}$, which appear in the formulas \eqref{eq:special_case_sc}-\eqref{eq:special_case_MP}, from $E_{\mu,\nu,z}$. They can alternatively be computed via the formulas
\begin{equation}\label{eq:end_point_etc_special_cases_sc_MP}
    z_{\mu \boxplus \nu}^*
    =
    \sup_{h\in(-\infty,\supp_{-}\nu)}F_{\mu,\nu}(h)
    ,\quad\quad\text{and}\quad\quad
    g_{\mu \boxplus \nu}^*
    =
    G_{\nu}\left(
        \underset{h\in(-\infty,\supp_{-}\nu)}{\argsup}F_{\mu,\nu}(h)
    \right),
\end{equation}
where 
\begin{equation}\label{eq:F_mu_nu_special_cases_sc_MP}
    F_{\musc[\beta],\nu}(h) = h -  \beta^2 G_{\nu}(h),
    \quad\quad\text{and}\quad\quad
    F_{\muMP,\nu}(h) = h + \frac{\beta}{1+G_{\nu}(h)}.
\end{equation}
Furthermore, the number, nature and position of the critical points of $E_{\mu,\nu,z}$ are related to solutions of $F_{\mu,\nu}(h) = z$ (see Proposition \ref{prop:Emunu_props} below).

Our variational formula \eqref{eq:special_case_sc} can be viewed as an extension and generalization of \cite[Lemma 4.3]{mon}. The latter is essentially an upper bound on $U_{\musc[\beta],\nu}(z)$, for a special $\nu$ arising in the computation of the expected number of critical points of a certain energy surface in the Sherrington-Kirkpatrick model of spin glasses (i.e. the model's ``annealed complexity''). Though the argument of \cite{mon} was the starting point of our investigation of the topic, the eventual proof of our  main result (i.e. Theorem \ref{TH:main} below) takes quite a different route from \cite{mon}. We use \eqref{eq:special_case_sc} and \eqref{eq:special_case_sc_MP_E_crit_points}
in our work \cite{noi} 
about the random determinant appearing in a computation of the expected number of critical points similar to that of \cite{mon}.

\subsection{General variational formula}

The formulas \eqref{eq:special_case_sc}-\eqref{eq:special_case_MP} are special cases of our general result, which holds for any compactly supported $\mu,\nu$, and which we now formulate. Allowing a slight abuse of notation, the formula \eqref{eq:R-transf} also defines $R_\mu$ as a real analytic function. Indeed, $G_\mu$ restricted to $ [\supp_{-}\mu, \supp_{+}\mu]^c \subset \mathbb{R}$ is real analytic, strictly monotone and therefore invertible. Letting the inverse on the r.h.s. of \eqref{eq:R-transf} denote the inverse of this restriction, we obtain a real analytic function $R_\mu(t)$ with domain
\begin{equation}\label{eq:def_Dhatmu_intro}
    \hat{D}_\mu 
    :=
    (-G_\mu(\supp_-\mu),-G_\mu(\supp_+\mu))\,.
\end{equation}    
Our general variational formula for the logarithmic potential of $\mu \boxplus \nu$  involves the function
\begin{equation}\label{eq:E_def}
    E_{\mu,\nu,z}(g)
    :=
    \int_{0}^{g}s R_\mu'(-s) ds
    +
    \int \log(\lambda-z+R_\mu(-g)) \nu(d\lambda),
\end{equation}
with domain
\begin{equation}\label{eq:E_func_dom_Ehat}
    \hat{\mathcal{E}}_{\mu,\nu,z} := \{ g\in -\hat{D}_\mu: z-R_\mu(-g)< \supp_{-}\nu\}.
\end{equation}
Note that the second term on the r.h.s. of \eqref{eq:E_def} is $U_\nu(z - R_\mu(-g))$. For many probability measures $\mu$, the real analytic function $R_\mu$ can be further analytically extended, and if so there is a unique extension to a maximal real interval $D_\mu $ 
around zero (see Lemma \ref{lem:real_R_trans_extend}; for instance for the semicircle law $R_{\musc[\beta]}(t) = t \beta^2$ and $D_{\musc[\beta]} = \mathbb{R}$, while $\hat{D}_{\musc[\beta]} = (-\beta,\beta)$).
Using this extension the function \eqref{eq:E_def} can be extended to the domain
\begin{equation}\label{eq:E_func_dom_no_hat}
    \mathcal{E}_{\mu,\nu,z} := \{g \in -D_\mu: z - R_\mu(-g) < \supp_{-} \nu \}.
\end{equation}
Our main result is the following.
\begin{theorem}\label{TH:main}
    Assume that $\mu,\nu$ are compactly supported probability measures on $\mathbb{R}$, s.t. $\mu$ is non-degenerate, and recall the definitions \eqref{eq:zstar_gstar_def} and \eqref{eq:E_def}-\eqref{eq:E_func_dom_no_hat}.
    For all real $z<z^*_{\mu \boxplus \nu}$
    \be\label{eq:main}
        U_{\mu\boxplus\nu}(z)
        =
        \inf_{g \in \hat{\mathcal{E}}_{\mu,\nu,z} \cap (0, g^*_{\mu \boxplus \nu}) }
        E_{\mu,\nu,z}(g)
        =
        \inf_{g \in \mathcal{E}_{\mu,\nu,z} \cap (0,g_{\mu \boxplus \nu}^*) }
        E_{\mu,\nu,z}(g).
    \ee
    Furthermore, the unique minimizer of both infima is $g=G_{\mu\boxplus\nu}(z)$.
\end{theorem}

\begin{remark}
    a) Theorem \ref{thm:add_conv_supp_end_point} below provides formulas for the quantities $z_{\mu\boxplus\nu}^*$ and  $g_{\mu\boxplus\nu}^*$ that appear in the statement of Theorem \ref{TH:main}.

    b) In \eqref{eq:main} two variants of the variational formula are given; the first one restricts the minimization to the smaller set $\hat{\mathcal{E}}_{\mu,\nu,z}$, and the second restricts it to the larger set $\mathcal{E}_{\mu,\nu,z}$.
    
    To determine the smaller set $\hat{\mathcal{E}}_{\mu,\nu,z}$ and use the first variant in \eqref{eq:main}, one needs to take into account the values of $G_\mu(\supp_{\pm}\mu)$, which define the domain $\hat{D}_\mu$ and thus also $\hat{\mathcal{E}}_{\mu,\nu,z}$ (see \eqref{eq:E_func_dom_Ehat}).
    
    To determine the larger set $\mathcal{E}_{\mu,\nu,z}$ and use the second variant in \eqref{eq:main}, one needs to take into account the domain $D_\mu$ to which $R_\mu$ can be analytically extended (see \eqref{eq:E_func_dom_no_hat}). When a probability distribution is defined in terms of its $R$-transform, then typically the extended domain $D_\mu$ is obtained somewhat more directly - e.g. for the semicircle law $R_{\musc[\beta]}(t) = t\beta^2$ and $D_{\musc[\beta]}=\mathbb{R}$, while $\hat{D}_{\musc[\beta]} = (-\beta,\beta)$, as already mentioned, and for the Marchenko-Pastur law $R_\muMP(t) = \beta (1-t)^{-1}$ and $D_\muMP=(-\infty,1)$, while $\hat{D}_\muMP$ is the more complicated expression in \eqref{eq:R_trans_MP_special_case}. In such cases the second variant in \eqref{eq:main} takes a simpler form.
\end{remark}

\subsubsection{Overview of the proof of Theorem \ref{TH:main}}

The main step in the proof of Theorem \ref{TH:main} is proving the identity
\begin{equation}\label{eq:intro_inf_E_is_E_of_G+}
    \inf_{g \in \mathcal{E}_{\mu,\nu,z} \cap (0,g_{\mu \boxplus \nu}^*) } E_{\mu,\nu,z}(g)
    =
    E_{\mu,\nu,z}(G_{\mu\boxplus\nu}(z))
    \quad\forall\, z < z_{\mu\boxplus\nu}^*. \tag{\it I}
\end{equation}
The proof of \eqref{eq:intro_inf_E_is_E_of_G+} is based on the formula
\begin{equation}\label{eq:intro_E_deriv}
        E_{\mu,\nu,z}'(g)
        =
        R_\mu'(-g)
        \left(
            g
            -
            \int\frac{1}{\lambda - z + R_\mu(-g)} \nu(d\lambda)
        \right),
\end{equation}
which is derived straight-forwardly from \eqref{eq:E_def} (see Lemma \ref{lem:E_mu_nu_z_deriv}). Consider the fixed-point equation
\begin{equation}\label{eq:intro_fp}
    \int\frac{1}{\lambda-z+R_\mu(-g)}\nu(d\lambda) = g,
\end{equation}
cf. the second factor on the r.h.s. of \eqref{eq:intro_E_deriv}. The fixed-point equation \eqref{eq:intro_fp} has sometimes been used to define the additive convolution, at least when $\mu$ is the semicircle law (see \cite[(1.6)]{pastur}). A special case also appears in \cite[(4.4)]{mon}. Indeed, one can show that for $z < z_{\mu \boxplus \nu}^*$, the number $g = G_{\mu \boxplus \nu}(z)$ is the unique solution of \eqref{eq:intro_fp} in $(0,g^*_{\mu \boxplus \nu})$. We will show this by using the definitions \eqref{eq:defSti}-\eqref{eq:sumR} to derive the equivalences
\begin{equation}\label{eq:intro_fp_iff}
    \begin{array}{ccccl}
        \eqref{eq:intro_fp} & \overset{\eqref{eq:defSti}}{\iff} & G_{\nu}(z - R_\mu(-g)) = g & \iff & z - R_\mu(-g) = G_{\nu}^{\inv}(g)\\
        & & & \overset{\eqref{eq:R-transf}}{\iff} & z - R_\mu(-g) = R_\nu(-g) - g^{-1}\\
         &  &  & \overset{\eqref{eq:sumR}}{\iff} & z = R_{\mu \boxplus \nu}(-g) - g^{-1}\\
         &  &  & \overset{\eqref{eq:R-transf}}{\iff} & z = G_{\mu\boxplus\nu}^{\inv}(g),
    \end{array}
\end{equation}
valid for $g \in (0,g^*_{\mu \boxplus\nu})$ and $z < \supp_- \nu$ (see Lemma \ref{lem:fp_diff_sign}). The inverse $G_{\mu\boxplus\nu}^{\inv}$ is monotone since $G_{\mu\boxplus\nu}$ is, so the equation $z = G_{\mu\boxplus\nu}^{\inv}(g)$ has the unique solution $g = G_{\mu\boxplus\nu}(z)$ in the relevant part of the domain of $G_{\mu\boxplus\nu}$, which is $(0,g^*_{\mu\boxplus\nu})$. This is then also the unique solution of the fixed-point equation \eqref{eq:intro_fp} in that range. Using this in \eqref{eq:intro_E_deriv}, we will conclude that $g=G_{\mu \boxplus \nu}(z)$ is a critical point of $E_{\mu,\nu,z}$ when $z < z^*_{\mu \boxplus \nu}$. Furthermore, one can verify that
\begin{equation}\label{eq:intro_R_deriv_pos}
    R_\mu'(-g) \text{ is positive on }(0,g^*_{\mu\boxplus\nu})
\end{equation}
(Lemma \ref{lemma:Rmu_increasing} and \eqref{eq:D_sandwich}), which will allow us to conclude that $g=G_{\mu \boxplus \nu}(z)$
is the unique critical point of $E_{\mu,\nu,z}$ in $(0,g^*_{\mu\boxplus\nu})$. Lastly, \eqref{eq:intro_fp_iff} can be extended to the identity
\begin{equation}\label{eq:intro_fp_sign}
    \sign\left(
        g - \int\frac{1}{\lambda-z+R_\mu(-g)}\nu(d\lambda)
    \right)
    =
    \sign( G_{\mu\boxplus\nu}^{\inv}(g) - z ),
\end{equation}
also valid for $g \in (0,g^*_{\mu \boxplus\nu}), z < \supp_- \nu$. Combining \eqref{eq:intro_E_deriv}, \eqref{eq:intro_R_deriv_pos}-\eqref{eq:intro_fp_sign} with the aforementioned monotonicity of $G_{\mu\boxplus\nu}^{\inv}$ will allow us to argue that $g=G_{\mu \boxplus \nu}(z)$ is a local minimum of $E_{\mu,\nu,z}$. Since it is also the unique critical point in $(0,g_{\mu \boxplus \nu}^*)$, we will be able to deduce \eqref{eq:intro_inf_E_is_E_of_G+}.

Lastly, it is straight-forward to prove the identity
\begin{equation}\label{eq:intro_E_of_G+}
    E_{\mu,\nu,z}(G_{\mu \boxplus \nu}(z))
    =
    U_{\mu \boxplus \nu}(z)
    \quad
    \forall\, z < z_{\mu,\nu}^*, \tag{{\it II}}
\end{equation}
essentially by verifying that both sides of \eqref{eq:intro_E_of_G+} have the same
derivative in $z$, using the chain rule and that $g=G_{\mu \boxplus \nu}(z)$ is a critical point of $E_{\mu,\nu,z}$ (see Corollary \ref{cor:U_+(z)_and_E_mu_nu(z,G_+(z))_ident}).

We will then conclude the proof of \eqref{eq:main} by simply combining \eqref{eq:intro_inf_E_is_E_of_G+} and \eqref{eq:intro_E_of_G+}.

The sketch above, and indeed the proofs of all our main results (excluding Theorem \ref{thm:stieltjes_global_invert}), are based purely on the properties of the Stieltjes and $R$-transforms as functions on the real line, rather than as holomorphic functions in a complex domain. Only when connecting the $R$-transform as a real function to the standard theory of $R$-transforms in Section \ref{sect:R_transform_theory} do we take the latter point of view

\subsection{Spurious $g \notin [0, g_{\mu\boxplus\nu}^*]$ with $E_{\mu,\nu,z}(z) < U_{\mu \boxplus \nu}(z)$}

The fixed-point equation \eqref{eq:intro_fp} can have solutions $g$ s.t. $g>g^*_{\mu\boxplus\nu}$, and it is also possible that $R_\mu'(-g) = 0$ for $g$ outside $(0,g^*_{\mu\boxplus\nu})$. Such $g$ give rise to ``spurious'' critical points of $E_{\mu,\nu,z}$, which can be local minima, local maxima or inflection points, and for which the value of $E_{\mu,\nu,z}(g)$ may be lower than
$U_{\mu \boxplus \nu}(z) = E_{\mu,\nu,z}(G_{\mu\boxplus\nu}(z))$ (see Figure \ref{fig:intro_E_mu_nu_z_F_mu_nu_shapes1} c). In addition, the global minimum of the function $E_{\mu,\nu,z}$ may lie at the right endpoint of its support, either as a finite number, or because $E_{\mu,\nu,z}$ diverges to $-\infty$ approaching that endpoint (see Figure \ref{fig:intro_E_mu_nu_z_F_mu_nu_shapes1} b). The possible presence of these ``spurious'' $g$ with values of $E_{\mu,\nu,z}(g)$ below $U_{\mu \boxplus \nu}(z)$ makes it necessary to exclude $g \notin [0, g^*_{\mu\boxplus\nu}]$ in the infima in \eqref{eq:main}. It is also possible for $E_{\mu,\nu,z}$ to have critical points when $z > z^*_{\mu \boxplus \nu}$ (see Figure \ref{fig:intro_E_mu_nu_z_F_mu_nu_shapes1} d). In particular, \eqref{eq:special_case_sc_MP_E_crit_points} does not hold in general, so it cannot always be used to characterize $z_{\mu \boxplus \nu}^*,g_{\mu \boxplus \nu}^*$.

\subsection{General characterization of $z_{\mu \boxplus \nu}^*$ and $g_{\mu \boxplus \nu}^*$}

To apply Theorem \ref{TH:main} one needs to know the values of $z_{\mu \boxplus \nu}^*,g_{\mu \boxplus \nu}^*$ from \eqref{eq:zstar_gstar_def},
which are properties of $\mu \boxplus \nu$ that one typically does not have direct access to in situations where the variational formula \eqref{eq:main} may be of use. One can characterize $z_{\mu,\nu}^*,g^*_{\mu \boxplus \nu}$ in terms of the fixed point equation \eqref{eq:intro_fp}, or in terms of $R_\mu$ and $R_\nu$. The most ``user-friendly'' characterization is arguably the following one in terms of $R_\mu$ and $G_\nu$. Note that one {\it would} typically have some control over $R_\mu$ and $G_\nu$ in situations where one has such control over $E_{\mu,\nu,z}$ from \eqref{eq:E_def}, which are likely to be the situations in which \eqref{eq:main} is of practical use. 
Let
\begin{equation}\label{eq:F_func_def}
    F_{\mu,\nu}(h)
    =
    R_\mu(-G_{\nu}(h)) + h
    \quad \text{ for }\quad
    \dom(F_{\mu,\nu}) =
    \left\{
        h\in(-\infty,\supp_{-}\nu):
        G_{\nu}(h)
        \in
        - D_\mu
    \right\},
\end{equation}
\begin{equation}\label{eq:mathcal_H_mu_nu_def}
    \mathcal{H}_{\mu,\nu} = \left\{ h\in \dom(F_{\mu,\nu}):F_{\mu,\nu}'(t)>0\text{ for all }t\in(-\infty,h)\right\},
    \quad\quad
     h_{\mu,\nu}^* := \sup \mathcal{H}_{\mu,\nu}.
\end{equation}
The definition \eqref{eq:mathcal_H_mu_nu_def} expresses as a formula that $h_{\mu,\nu}^*$ is the leftmost critical point of $F_{\mu,\nu}$ if such a point exists, and otherwise $h_{\mu,\nu}^*$ is the right endpoint of $\dom(F_{\mu,\nu})$ (see Proposition \ref{prop:Fmunu_props} \ref{item:leftmost_crit_point[prop:Fmunu_props]},\ref{item:if_no_crit_point[prop:Fmunu_props]} below).
\begin{theorem}\label{thm:add_conv_supp_end_point}
    Let $\mu,\nu$ be compactly supported probability measures on $\mathbb{R}$, s.t. $\mu$ is non-degenerate, and recall \eqref{eq:F_func_def}-\eqref{eq:mathcal_H_mu_nu_def} and the notation \eqref{eq:zstar_gstar_def}. Then

    \begin{enumerate}[label=\alph*)]
        \item \label{thm:add_conv_supp_end_point-item:dom_F_interval} $\dom(F_{\mu,\nu})$ is an interval of the form
        $(-\infty,h_{\max})$. 
        
        \item \label{thm:add_conv_supp_end_point-item:F_inc} $F_{\mu,\nu}$ is increasing for small enough $h$.
    
        \item \label{thm:add_conv_supp_end_point-item:g_star}
            \begin{equation}\label{eq:gstar[thm:add_conv_supp_end_point]} 
                g^*_{\mu \boxplus \nu} =
                G_\nu( h_{\mu,\nu}^* ).
            \end{equation}
    
        \item
        \label{thm:add_conv_supp_end_point-item:z_star}
            \begin{equation}\label{eq:z_star[thm:add_conv_supp_end_point]}
                z_{\mu \boxplus \nu}^*
                =
                F_{\mu,\nu}(h_{\mu,\nu}^*)
                =
                \sup_{h \in \mathcal{H}_{\mu,\nu}}
                F_{\mu,\nu}(h),
            \end{equation}
            (where if $h_{\mu,\nu}^* \notin \dom(F_{\mu,\nu})$, then $F_{\mu,\nu}(h_{\mu,\nu}^*)$ denotes the limit $\lim_{h \uparrow h_{\mu,\nu}^*} F_{\mu,\nu}(h)$, which exists).
        \item \label{thm:add_conv_supp_end_point-item:G_nu(h_star)_in_closure}
            $G_\nu(h_{\mu,\nu}^*)$ lies in the closure of $-\hat{D}_\mu$.
        
        \item \label{thm:add_conv_supp_end_point-item:D_hat}
            All the above remain true if $D_\mu$ in \eqref{eq:F_func_def} is replaced by the possibly smaller interval $\hat{D}_\mu$ from \eqref{eq:def_Dhatmu_intro}.
    \end{enumerate}
\end{theorem}

\subsubsection{Overview of proof of Theorem \ref{thm:add_conv_supp_end_point}}

The proof of Theorem \ref{thm:add_conv_supp_end_point} is based on the change of variables $h = G_\nu^{\inv}(g)$,
which is useful since similarly to how \eqref{eq:intro_fp_iff} and \eqref{eq:intro_fp_sign} are derived, one can use the definitions \eqref{eq:R-transf}-\eqref{eq:sumR} to show that
\begin{equation}\label{eq:intro_F_G+_inv}
    F_{\mu,\nu}(G_{\nu}^{\inv}(g))
    =
    G_{\mu\boxplus\nu}^{\inv}(g)
    \quad\text{ for }\quad
    g \in (0,g^*_{\mu \boxplus \nu})
\end{equation}
(Lemma \ref{lem:fp_diff_sign}).
Thus $g \to F_{\mu,\nu}(G_{\nu}^{\inv}(g))$ is a real analytic
extension of $G_{\mu\boxplus\nu}^{\inv}$. The quantities $z_{\mu\boxplus\nu}^*,g_{\mu\boxplus\nu}^*$
are encoded in the ``shape'' of $G_{\mu\boxplus\nu}^{\inv}$: for
instance if
\begin{equation}\label{eq:G+_inv_shape_example_case}
    G_{\mu\boxplus\nu}(z_{\mu\boxplus\nu}^*)
    < 
    \infty
    = G_{\mu\boxplus\nu}'(z_{\mu\boxplus\nu}^*),
\end{equation}
then $g_{\mu\boxplus\nu}^*$ is the leftmost critical point of
$G_{\mu\boxplus\nu}^{\inv}$ and $z^*_{\mu \boxplus \nu} = G_{\mu\boxplus\nu}^{\inv}(g_{\mu\boxplus\nu}^*)$ (see Figure \ref{fig:intro_E_mu_nu_z_F_mu_nu_shapes1} b,c,d) - this is a consequence of $G_{\mu\boxplus\nu}^{\inv}$'s aforementioned monotonicity, and the basic fact
that if $(x,f(x))$ is the graph of a function $f$, then its ``mirror'' $(f(x),x)$ is the graph of the inverse of $f$. Using the reverse change of
variables $G_\nu(g) = h$, the aforementioned properties of $G_{\mu\boxplus\nu}^{\inv}$ translate to $h_{\mu,\nu}^* = G_{\nu}^{\inv}(g_{\mu\boxplus\nu}^*)$
being the leftmost critical point of $F_{\mu,\nu}$, and $F_{\mu,\nu}(h_{\mu,\nu}^*)=z_{\mu\boxplus\nu}^*$.
This is roughly speaking how \eqref{eq:gstar[thm:add_conv_supp_end_point]}-\eqref{eq:z_star[thm:add_conv_supp_end_point]} are proved when \eqref{eq:G+_inv_shape_example_case} holds
- mostly similar arguments apply in the other possible situations.

\subsection{Additional properties of $F_{\mu,\nu}$}

Usually $h_{\mu,\nu}^*$ is a local maximum of $F_{\mu,\nu}$ (see e.g. Figure \ref{fig:intro_E_mu_nu_z_F_mu_nu_shapes1} b,c,d), Proposition \ref{prop:Fmunu_props} \ref{item:BES_conds[prop:Fmunu_props]} and Remark \ref{rem:BES} below), and e.g. when $\mu = {\musc[\beta]}$ or $\mu = \muMP$ it is the global maximum (recall \eqref{eq:end_point_etc_special_cases_sc_MP}). It furthermore appears common that it is the unique critical point of $F_{\mu,\nu}$.
However, there are examples of $\mu,\nu$ for which $F_{\mu,\nu}$ has no critical points, or several critical points. Figure \ref{fig:intro_E_mu_nu_z_F_mu_nu_shapes1} contains some examples of the possible shapes of the function $F_{\mu,\nu}$.
The next proposition collects some facts about $F_{\mu,\nu}$, $h_{\mu,\nu}^*$ and \eqref{eq:gstar[thm:add_conv_supp_end_point]}-\eqref{eq:z_star[thm:add_conv_supp_end_point]}, including the aforementioned geometric description of $h_{\mu,\nu}^*$. It complements Theorem \ref{thm:add_conv_supp_end_point}. For any $z \in \mathbb{R}$, let
\begin{equation}\label{eq:h_mu_nu^<(z)_def}
        h_{\mu,\nu}^<(z)
        \text{ denote leftmost solution of }
        F_{\mu,\nu}(h) = z
        \text{ in }\dom(F_{\mu,\nu})
        \text{, if a solution exists,}
\end{equation}
and otherwise $h_{\mu,\nu}^<(z) = \sup \dom(F_{\mu,\nu})$.

\begin{proposition}\label{prop:Fmunu_props}
    If $\mu,\nu$ are compactly supported probability measures on $\mathbb{R}$, s.t. $\mu$ is non-degenerate, then the following holds.
    \begin{enumerate}[label=\alph*)]
        \item \label{item:leftmost_crit_point[prop:Fmunu_props]}
        If $F_{\mu,\nu}$ has a critical point, then $h_{\mu,\nu}^*$ is the leftmost critical point. 
        
        \item \label{item:if_no_crit_point[prop:Fmunu_props]} If $F_{\mu,\nu}$ has no critical points, then $h_{\mu,\nu}^* = \sup\dom( F_{\mu,\nu})$.
        
        \item \label{item:if_no_crit_point2[prop:Fmunu_props]} If $F_{\mu,\nu}$ has no critical points, then furthermore (i) $h_{\mu,\nu}^* = \supp_- \nu$
        and $g^*_{\mu\boxplus\nu} = G_\nu( \supp_- \nu )$, or (ii) $G_\mu( \supp_- \mu ) = G_\nu(h_{\mu,\nu}^*) = g^*_{\mu\boxplus\nu}$.

        \item \label{item:BES_conds[prop:Fmunu_props]} If there are constants $c_1,c_2>0$ s.t. $c_1^{-1}\varepsilon^{3/2}\le(\mu\boxplus\nu)(z_{\mu\boxplus\nu}^*,z_{\mu\boxplus\nu}^*+\varepsilon)\le c_1\varepsilon^{3/2}$ for all $\varepsilon \in (0,c_2)$, then $h_{\mu,\nu}^*$ is a local maximum of $F_{\mu,\nu}$ (see Figure \ref{fig:intro_E_mu_nu_z_F_mu_nu_shapes1} b,c,d).

        \item \label{item:F'_monotone[prop:Fmunu_props]}
        If $F_{\mu,\nu}$ is strictly concave, then $h_{\mu,\nu}^*$ is the global maximum of $F_{\mu,\nu}$, and \eqref{eq:end_point_etc_special_cases_sc_MP} holds for the given $\mu,\nu$
        (see Figure \ref{fig:intro_E_mu_nu_z_F_mu_nu_shapes1} a,b).

        \item \label{item:F(h)=z_sol[prop:Fmunu_props]}
        The equation $F_{\mu,\nu}(h) = z$ has a solution in $\dom(F_{\mu,\nu}) \cap (-\infty,h_{\mu,\nu}^*)$ iff $z < z^*_{\mu \boxplus \nu}$. If $z < z^*_{\mu \boxplus \nu}$ then $h_{\mu,\nu}^<(z)$ is the unique solution in that set, and $G_{\mu \boxplus \nu}(z) = G_\nu(h_{\mu,\nu}^<(z))$.
    \end{enumerate}
\end{proposition}

\begin{remark}\label{rem:BES}
    The condition in Proposition \ref{prop:Fmunu_props} \ref{item:BES_conds[prop:Fmunu_props]} is satisfied whenever $\mu \boxplus \nu$ has a density $f$ which has the power law behaviour $f(x) \propto (z^*_{\mu\boxplus\nu} - x)^{1/2}$ at its left edge, i.e. if it has the same left-edge behaviour as the 
    semicircle law. This is quite common. For instance, by \cite[Theorem 2.2]{BES}, it holds whenever $\mu,\nu$ have densities that are positive on an interval and have power law behaviour at the edges with any exponent in $[-1,1]$.
\end{remark}

\subsection{Additional properties of $E_{\mu,\nu,z}$}
 The shape of $F_{\mu,\nu}$ furthermore determines the shape of $E_{\mu,\nu,z}$ through its derivative \eqref{eq:intro_E_deriv}, whose sign is related to $F_{\mu,\nu}$ via \eqref{eq:intro_fp_sign} and \eqref{eq:intro_F_G+_inv}. In particular the critical points of $E_{\mu,\nu,z}$ are related to the solutions of $F_{\mu,\nu}(h) = z$, and the sign of $F_{\mu,\nu}(h) - z$ as a function of $h$ determines their nature as local minima, local maxima or inflection points. The next proposition collects some resulting properties of $E_{\mu,\nu,z}$, its critical points, and their relation to $F_{\mu,\nu}$. It complements Theorem \ref{TH:main}.
\begin{proposition}\label{prop:Emunu_props}
    If $\mu,\nu$ are compactly supported probability measures on $\mathbb{R}$, s.t. $\mu$ is non-degenerate, then the following holds.
    \begin{enumerate}[label=\alph*)]
        \item \label{prop:Emunu_props_E_func_crit_point}
        A number $g \in \mathcal{E}_{\mu,\nu,z}$ is a critical point of $E_{\mu,\nu,z}$ iff $R_\mu'(-g)=0$, or $g=G_{\nu}(h)$ for an $h<\supp_{-}\nu$ s.t. $F_{\mu,\nu}(h)=z$.        

        \item \label{prop:Emunu_props_loc_min_of_E_func}
        $E_{\mu,\nu,z}$ has a critical point in $(-\infty,g^*_{\mu \boxplus \nu}) \cap \mathcal{E}_{\mu,\nu,z}$ iff $z<\supp_{-} \mu\boxplus\nu$. If so, it has a unique critical point in that set, which is a non-degenerate local minimum and equals $G_{\mu \boxplus \nu}(z)$ (cf. Proposition \ref{prop:Fmunu_props} \ref{item:F(h)=z_sol[prop:Fmunu_props]}.
    \end{enumerate}
\end{proposition}
\begin{remark}\label{rem:intro_spurios_crit_point_charaterization}
    Proposition \ref{prop:Emunu_props} characterizes precisely any ``spurious'' critical points of $E_{\mu,\nu,z}$, i.e. critical points other than $g = G_{\mu \boxplus \nu}(z)$ when $z < z^*_{\mu \boxplus \nu}$.
\end{remark}

Under certain circumstances, one can further extend the range of the second infimum in \eqref{eq:main}, and more precisely characterize the shape of $E_{\mu,\nu,z}$ (as \eqref{eq:special_case_sc_MP_E_crit_points} does when $\mu$ is the semicircle or Marchenko-Pastur law). To this end let
\begin{equation}\label{eq:h_mu_nu^>(z)_def}
    \begin{array}{c}
        h_{\mu,\nu}^>(z)\text{ denote the second solution of }F_{\mu,\nu}(h) = z\text{ from the left,}\\
        \text{if such a solution exists, and otherwise }h_{\mu,\nu}^>(z)=\sup\dom(F_{\mu,\nu}).
    \end{array}
\end{equation}
Let furthermore
\begin{equation}\label{eq:D_mu_plus_def}
    D_\mu^+ := \{ t \in D_\mu : R_\mu'(t)>0\}
\end{equation}
for any compactly supported probability $\mu$.

\begin{proposition}\label{prop:Emunu_inf_weaker_cond}
    If $\mu,\nu$ are compactly supported probability measures on $\mathbb{R}$, s.t. $\mu$ is non-degenerate, then the following holds.
    \begin{enumerate}[label=\alph*)]
        \item \label{item:F(h)=z_sec_sol[prop:Emunu_inf_weaker_cond]}
        Let $z < z^*_{\mu \boxplus \nu}$ and assume that $h^>_{\mu,\nu}(z)$ is a solution of $F_{\mu,\nu}(z)=h$. Let
        $g = G_\nu(h^>_{\mu,\nu}(z))$. If $(0,g) \subset - D_\mu^+$, then $g$ is a local maximum or an inflection point of $E_{\mu,\nu,z}$.
        

        
        


        \item \label{item:var_form_larger_range[prop:Emunu_inf_weaker_cond]}
        The rightmost variational formula in \eqref{eq:main} remains true if $(0,g_{\mu \boxplus \nu}^*)$ is replaced by $(u, v)$, for any $u \in [-\infty,0)$ and $v \in (g_{\mu \boxplus \nu}^*, G_{\nu}(h_{\mu,\nu}^>(z)))$ s.t. $(u,v) \subset -D_\mu^+$.

        \item \label{item:F_mon_at_most_loc_min_and_loc_max[prop:Emunu_inf_weaker_cond]} If $F_{\mu,\nu}$ is strictly concave and $\mathcal{E}_{\mu,\nu,z} \subset -D^+_\mu$, then \eqref{eq:special_case_sc_MP_E_crit_points} holds for the given $\mu,\nu$.       
    \end{enumerate}
\end{proposition}
We will see that $\hat{D}_\mu \subset D_\mu^+$ for any $\mu$ (Lemma \ref{lemma:Rmu_increasing}). The above proposition is probably most useful when $D_\mu^{+} = D_\mu$, which holds for instance for the semicircle and Marchenko-Pastur laws.

\subsection{Additional result on invertibility of Stieltjes transform}
Lastly, we include a result on the invertibility of $G_\mu$ as a  holomorphic function on  $\mathbb{C}\setminus\supp \mu$. Such results are important for the theory of $R$-transforms. Our result significantly strengthens \cite[Theorem 17, Chapter 3]{speicher}, which proves that $G_\mu$ restricted to $\{z \in \mathbb{C}:|z|>4r\}$ is invertible, where $r = |\supp_+\mu| - |\supp_-\mu|$. Our stronger result Theorem \ref{thm:stieltjes_global_invert} implies that $G_\mu$ restricted to the larger set $\{ z \in \mathbb{C} : \Re(z) \notin [\supp_-\mu,\supp_+\mu]\}$ is invertible. Thus it establishes invertibility all the way up to the boundary of the support of $\mu$. A previous version of this manuscript used Theorem \ref{thm:stieltjes_global_invert} in the proofs of Section \ref{sect:R_transform_theory}, but the current one presents a simpler argument for which it is  enough to use \cite[Theorem 17, Chapter 3]{speicher}. However, we believe that the stronger invertibility result is of independent interest, and may be useful for future work. It is therefore stated and proved in Section \ref{sec:stronger_stieltjes_invert}.

\subsection{Examples}
Here we deduce the special cases \eqref{eq:special_case_sc} (semicircle law) and \eqref{eq:special_case_MP} (Marchenko-Pastur law) from Theorem \ref{TH:main}, and the corresponding 
\eqref{eq:special_case_sc_MP_E_crit_points}-\eqref{eq:F_mu_nu_special_cases_sc_MP} from Theorem \ref{thm:add_conv_supp_end_point}, Proposition \ref{prop:Fmunu_props} and Proposition \ref{prop:Emunu_inf_weaker_cond}.

\subsubsection{Semicircle law}

For any $\beta>0$, the semicircle law supported on $[-2\beta,2\beta]$ is
\be\label{eq:sc_def}
    \musc[\beta](d\l) = \frac{1}{2\pi\beta^2}\sqrt{ 4\beta^2 - \l^2} 1_{[-2\beta,2\beta]}d\l.
\ee
Its Stieltjes transform satisfies
\be\label{eq:G_sc}
    G_{\musc[\beta]}(z)
    =
    \frac{ - z - \sqrt{z^2 - 4\beta^2}}{2\beta ^2}
    \quad\text{for}\quad
    z \in (-\infty,-2\beta).
\ee
Its $R$-transform, as a real analytic function, is given by
\begin{equation}\label{eq:R_trans_sc_special_case}
    R_{\musc[\beta]}(t) = t \beta^2 \text{ with domain } \hat{D}_{\musc[\beta]} = (-\beta,\beta) \text{ resp. } D_{\musc[\beta]} = \mathbb{R},
\end{equation}
before resp. after analytic extension.


%

\begin{proof}[Proof of \eqref{eq:special_case_sc}]
    By \eqref{eq:E_func_dom_no_hat} and \eqref{eq:R_trans_sc_special_case}
    \begin{equation}
        \mathcal{E}_{\musc[\beta],\nu,z}
        =
        \{g \in \mathbb{R} : \beta^2 g < \supp_{-} \nu - z \}.
    \end{equation}
    By \eqref{eq:E_def} and \eqref{eq:R_trans_sc_special_case}
    \begin{equation}
        E_{\musc[\beta],\nu,z}(g)
        =
        \frac{\beta^2 g^2}{2} + \int\log(\lambda - z - \beta^2 g)\nu(d\lambda).
    \end{equation}
    Thus \eqref{eq:special_case_sc} follows from Theorem \ref{TH:main}.
\end{proof}

\begin{proof}[Proof of \eqref{eq:F_mu_nu_special_cases_sc_MP} for $\mu = {\musc[\beta]}$]
    The identity for $F_{\musc[\beta],\nu}(h)$ in \eqref{eq:F_mu_nu_special_cases_sc_MP} follows from \eqref{eq:F_func_def} and \eqref{eq:R_trans_sc_special_case}.
\end{proof}

\begin{lemma}\label{lem:F_musc_nu_concave}
    Let $\beta>0$ and assume $\mu = \musc[\beta]$. For any  compactly supported $\nu$, the function $F_{\musc[\beta],\nu}$ from \eqref{eq:F_mu_nu_special_cases_sc_MP} is strictly concave.
\end{lemma}

\begin{proof}
    The derivative is $F_{\musc,\nu}'(h) = 1 - \beta^2 G_\nu'(h)$. 
    Since $G_\nu'(h) = \int (z - h)^{-2}\nu(dz)$, the function $G_\mu'$ is strictly monotone increasing on $(-\infty,\supp_-\nu)$, so $F_{\musc,\nu}'$ is strictly monotone decreasing.
\end{proof}

\begin{proof}[Proof of \eqref{eq:special_case_sc_MP_E_crit_points} for $\mu = {\musc[\beta]}$]
    Combine Proposition \ref{prop:Emunu_inf_weaker_cond} \ref{item:F_mon_at_most_loc_min_and_loc_max[prop:Emunu_inf_weaker_cond]}, Lemma \ref{lem:F_musc_nu_concave} and \eqref{eq:R_trans_sc_special_case}.
\end{proof}

\begin{proof}[Proof of \eqref{eq:end_point_etc_special_cases_sc_MP} for $\mu = {\musc[\beta]}$]
    Combine Proposition \ref{prop:Fmunu_props} \ref{item:F'_monotone[prop:Fmunu_props]} and Lemma \ref{lem:F_musc_nu_concave}.
\end{proof}

\subsubsection{Marchenko-Pastur law}

Recall the Marchenko-Pastur law with parameter $\beta>0$
\be\label{eq:MP}
    \muMP
    :=
    \max(1-\a,0)\d_0
    +
    \frac{\sqrt{(\l_+-\l)(\l-\l_-)}}{2\pi\l} d\l
    \,\,,\quad \l_\pm:=(1 \pm \sqrt\a)^2.
\ee
Its Stieltjes transform satisfies
\be\label{eq:G_sc}
    G_\muMP(z) = \frac{\sqrt{(\beta -1 - z)^2 - 4z} - (\beta - z - 1)}{-2z}
    \quad\text{for}\quad
    z \in (-\infty,0).
\ee
Its $R$-transform is given by
\begin{equation}\label{eq:R_trans_MP_special_case}
    R_\muMP(t) = \frac{\beta}{1-t}
    \text{ with domain }
    \hat{D}_\muMP
        =
        \left(
            t_{\min}(\beta)
            ,
            \frac{1}{1+\sqrt{\beta}}
        \right)
    \text{ resp. }
    D_\muMP = (-\infty,1),
\end{equation}
before resp. after analytic extension, where $t_{\min}(\beta) = -\frac{1}{\sqrt{\beta}-1}$ for $\beta>1$ and $t_{\min}(\beta)=-\infty$ for $\beta\in(0,1]$.


\begin{proof}[Proof of \eqref{eq:special_case_MP}]
    By \eqref{eq:E_func_dom_no_hat} and \eqref{eq:R_trans_MP_special_case}
    \begin{equation}
        \mathcal{E}_{\muMP,\nu,z}
        =
        \left\{
            g \in (-\infty,1) : z-\supp_{-}\nu<\frac{\beta}{1+g}
        \right\}.
    \end{equation}
    By \eqref{eq:E_def} and \eqref{eq:R_trans_MP_special_case}
    \begin{equation}
        E_{\muMP,\nu,z}(g)
        =
        \beta\left(
            \log(1+g)-\frac{g}{1+g}
        \right)
        +
        \int\log\left(\lambda-z+\frac{\beta}{1+g}\right)\nu(d\lambda).
    \end{equation}
    Thus \eqref{eq:special_case_MP} follows from Theorem \ref{TH:main}.
\end{proof}

\begin{proof}[Proof of \eqref{eq:F_mu_nu_special_cases_sc_MP} for $\mu = \muMP$]
    The identity for $F_{\muMP,\nu}(h)$ in \eqref{eq:F_mu_nu_special_cases_sc_MP} follows from \eqref{eq:F_func_def} and \eqref{eq:R_trans_MP_special_case}.
\end{proof}

\begin{lemma}\label{lem:F_muMP_nu_concave}
    Let $\beta>0$ and assume $\mu=\mu_{\MP,\b}$. For any  compactly supported $\nu$, the function $F_{\muMP,\nu}$ from \eqref{eq:F_mu_nu_special_cases_sc_MP} is strictly concave.
\end{lemma}
\begin{proof}
    The derivative is $F_{\muMP,\nu}'(h) = 1 - \beta\frac{G_{\nu}'(h)}{(1+G_{\nu}(h))^2}$. 
    We now argue that $h\to\frac{G_{\nu}'(h)}{(1+G_{\nu}(h))^2}$ is strictly monotone increasing, which implies the claim.

    The derivative of $\log \frac{G_{\nu}'(h)}{(1+G_{\nu}(h))^2}$ is $\frac{G_{\nu}''(h)}{G_{\nu}'(h)}-2\frac{G_{\nu}'(h)}{1+G_{\nu}(h)}$. The condition $\frac{G_{\nu}''(h)}{G_{\nu}'(h)}-2\frac{G_{\nu}'(h)}{1+G_{\nu}(h)}\ge0$ is equivalent to $G_{\nu}''(h)(1+G_{\nu}(h))\ge2G_{\nu}'(h)^2$, which in turn is equivalent to
    $\mathbb{E}[X^{3}]\left(1+\mathbb{E}[X]\right)\ge\mathbb{E}[X^2]^2$ where $X=(\Lambda-z)^{-1}$ and the r.v. $\Lambda$ has law $\nu$. But any non-negative r.v. $X$ satisfies $\mathbb{E}[X^{3}]\mathbb{E}[X]\ge\mathbb{E}[X^2]^2$ by Hölder's inequality. Thus $h\to\frac{G_{\nu}'(h)}{(1+G_{\nu}(h))^2}$ is indeed strictly monotone increasing.
\end{proof}

\begin{proof}[Proof of \eqref{eq:special_case_sc_MP_E_crit_points} for $\mu = \muMP$]
    Combine Proposition \ref{prop:Emunu_inf_weaker_cond} \ref{item:F_mon_at_most_loc_min_and_loc_max[prop:Emunu_inf_weaker_cond]}, Lemma \ref{lem:F_muMP_nu_concave} and \eqref{eq:R_trans_MP_special_case}.
\end{proof}

\begin{proof}[Proof of \eqref{eq:end_point_etc_special_cases_sc_MP} for $\mu = \muMP$]
    Combine Proposition \ref{prop:Fmunu_props} \ref{item:F'_monotone[prop:Fmunu_props]} and Lemma \ref{lem:F_muMP_nu_concave}.
\end{proof}


\subsection{Organization of the paper}

In Section \ref{sect:R-trans} the $R$-transform as a real analytic function is properly introduced (omitting some proofs), and certain preliminary properties of this real $R$-transform and the Stieltjes transform are proved. Then the proof of Theorem \ref{thm:add_conv_supp_end_point} is given in Section \ref{sect:end_point_proof},
and the proof of Theorem \ref{TH:main}  in Section \ref{sect:Proof}. Section \ref{sect:R_transform_theory} relates the $R$-transform as a real analytic function to the more standard theory of $R$-transforms presented in \cite{speicher} (proving the results only stated in Section \ref{sect:R-trans} - it e.g. derives the identity \eqref{eq:sumR} for this notion of $R$-transform).
Finally, Section \ref{sec:stronger_stieltjes_invert} contains the statement and proof of our aforementioned strong result on the invertibility of $G_\mu$ as a function in the complex plane.

\newpage


\iftrue

\begin{figure}[t]
    \caption{
        Plots of $E_{\mu,\nu,z}, G^{\inv}_{\mu\boxplus\nu},G^{[-1]}_{\mu\boxplus\nu}$ and $F_{\mu,\nu}$ for four combinations of $\mu$ and $\nu$. See Remark \ref{rem:fig_caption_overflow} for additional discussion.
        $$
            \begin{array}{llll}
                a) & R_\mu(t) = \frac{1}{2(1-t)} & \nu = \frac{1}{2}\delta_{-1} + \frac{1}{2}\delta_1 & z=-1.5\\
                b) & R_\mu(t) = t            
                &\nu=\frac{1}{2}\delta_{-1}+\frac{1}{2}\delta_1
                &
                z=-3\\
                c)  &R_\mu(t)=t+\frac{t^2}{6}
                &
                \nu=\frac{1}{2}\delta_{-0.17}+\frac{1}{2}\delta_{-1.17}
                &
                z=-2.77\\
                d) & R_\mu(t)=t+\frac{t^2}{4.5}+\frac{t^3}{80}                    
                &\nu=\frac{1}{2}\delta_{-0.17}+\frac{1}{2}\delta_{-1.17}
                & z=-2.53725>z^*_{\mu \boxplus \nu} = -2.63725
            \end{array}
        $$
    }    
    
    \label{fig:intro_E_mu_nu_z_F_mu_nu_shapes1}    
    \centering
    \iftrue

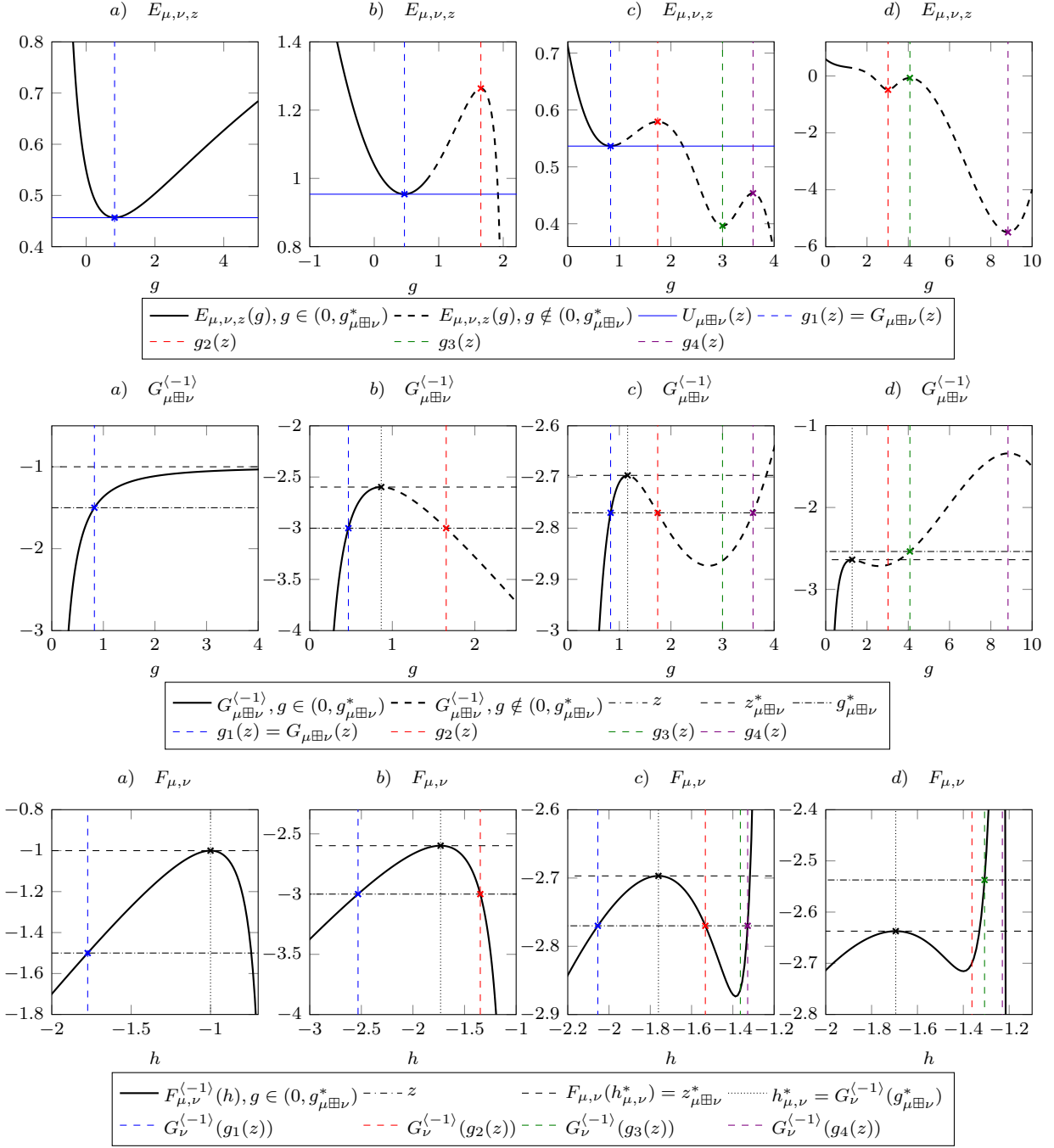
    \definecolor{g1col}{rgb}{0,0,1} 
    \definecolor{g2col}{rgb}{1,0,0} 
    \definecolor{g3col}{rgb}{0,0.5,0} 
    \definecolor{g4col}{rgb}{0.5,0,0.5} 

    \def\xshiftmusc{4cm}
    \def\xshiftmuMP{0cm}
    \def\xshiftEmanycritpt{8cm}

    \begin{minipage}{\textwidth}
        \begin{tikzpicture}
            \begin{scope}[xshift=\xshiftEmanycritpt]
                \begin{axis}[ legend style={at={(0.5,-0.2)}, anchor=north, legend columns=1, legend cell align=left},
                    alignedplot,
                    xlabel={$g$},
                    xmin=0, xmax=4,
                    ymin=0.36, ymax=0.72,
                    title={$c)\quad E_{\mu,\nu,z}$},
                    clip=true]
                    
                    \addplot[domain=-1:1.03371, samples=300,thick, black] 
                        ({x}, {x^2/2 - x^3/9 + (1/2)*ln(1.6 - x + x^2/6) + (1/2)*ln(2.6 - x + x^2/6)});
                    \addplot[domain=1.03371:7.5, samples=300,dashed, thick, black, dash phase=3pt] 
                        ({x}, {x^2/2 - x^3/9 + (1/2)*ln(1.6 - x + x^2/6) + (1/2)*ln(2.6 - x + x^2/6)});

                    \pgfmathdeclarefunction{EfuncC}{1}{%
                        \pgfmathparse{((#1)^2)/2 - ((#1)^3)/9 + (1/2)*ln(1.6 - #1 + ((#1)^2)/6) + (1/2)*ln(2.6 - #1 + ((#1)^2)/6)}%
                    }
                    \pgfmathsetmacro{\EyValA}{EfuncC(0.830841)}
                    \addplot[g1col] coordinates {(-5,\EyValA) (5, \EyValA)};
                    
                    \addplot[dashed, g1col] coordinates {(0.830841, -5) (0.830841, 3.2)};               
                    \addplot[only marks,mark=x,mark size=1.8pt,g1col,thick] coordinates {(0.830841, EfuncC(0.830841)};
                    
                    \addplot[dashed, g2col] coordinates {(1.7432, -5) (1.7432, 3.2)};
                    \addplot[only marks,mark=x,mark size=1.8pt,g2col,thick] coordinates {(1.7432, EfuncC(1.7432)};                    
                    
                    \addplot[dashed, g3col] coordinates {(3.0, -5) (3.0, 3.2)};
                    \addplot[only marks,mark=x,mark size=1.8pt,g3col,thick] coordinates {(3.0, EfuncC(3.0)};                    
                    
                    \addplot[dashed, g4col] coordinates {(3.5912, -5) (3.5912, 3.2)};
                    \addplot[only marks,mark=x,mark size=1.8pt,g4col,thick] coordinates {(3.5912, EfuncC(3.5912)};                        
                \end{axis}
            \end{scope}
            \begin{scope}[xshift=\xshiftmusc]
                \begin{axis}[ legend style={at={(0.5,-0.2)}, anchor=north, legend columns=1, legend cell align=left},
                    alignedplot,
                    xlabel={$g$},
                    xmin=-1, xmax=2.2,
                    ymin=0.8, ymax=1.4,
                    title={$b)\quad E_{\mu,\nu,z}$},
                    clip=true
                ]                    
                    \addplot[domain=-1:0.866025, samples=300,thick, black] 
                        ({x}, {x^2/2  + 0.5*ln(-1 + 3 -x) + 0.5*ln(1 + 3 -x)});
                    \addplot[domain=0.866025:2, samples=300, dashed, thick, black,dash phase=3pt] 
                        ({x}, {x^2/2  + 0.5*ln(-1 + 3 -x) + 0.5*ln(1 + 3 -x)});

                    \pgfmathdeclarefunction{EfuncB}{1}{%
                        \pgfmathparse{((#1)^2)/2 + 0.5*ln(2 - #1) + 0.5*ln(4 - #1)}%
                    }

                    \pgfmathsetmacro{\EyValB}{EfuncB(0.467911)}
                    \addplot[g1col] coordinates {(-2,\EyValB) (5, \EyValB)};
                    
                    \addplot[dashed, g1col] coordinates {(0.467911, 0) (0.467911, 1.4)};
                    \addplot[only marks,mark=x,mark size=1.8pt,g1col,thick] coordinates {(0.467911, EfuncB(0.467911)};
                    
                    \addplot[dashed, g2col] coordinates {(1.6527, 0) (1.6527, 1.4)};
                    \addplot[only marks,mark=x,mark size=1.8pt,g2col,thick] coordinates {(1.6527, EfuncB(1.6527)}; 
                \end{axis}
            \end{scope}
            \begin{scope}[xshift=\xshiftmuMP]
                \begin{axis}[ legend style={at={(0.5,-0.2)}, anchor=north, legend columns=1, legend cell align=left},
                    alignedplot,
                    xlabel={$g$},
                    xmin=-1, xmax=5,
                    ymin=0.4, ymax=0.8,
                    title={$a)\quad E_{\mu,\nu,z}$},
                    clip=true
                ]
                    \addplot[domain=-0.9:5, samples=300,thick, black] 
                        ({x}, {(1/2)*(-1 + 1/(1 + x) + ln(1 + x)) + (1/2)*ln(0.5 + 1/(2*(1 + x))) + (1/2)*ln(2.5 + 1/(2*(1 + x)))});

                    \pgfmathdeclarefunction{EfuncA}{1}{%
                        \pgfmathparse{(1/2)*(-1 + 1/(1 + #1) + ln(1 + #1)) + (1/2)*ln(0.5 + 1/(2*(1 + #1))) + (1/2)*ln(2.5 + 1/(2*(1 + #1)))}%
                    }

                    \pgfmathsetmacro{\EyValB}{EfuncA(0.826462)}
                    \addplot[g1col] coordinates {(-2,\EyValB) (5, \EyValB)};
                
                    \addplot[dashed, g1col] coordinates {(0.826462, 0) (0.826462, 2)};
                    \addplot[only marks,mark=x,mark size=1.8pt,g1col,thick] coordinates {(0.826462, EfuncA(0.826462)};
                \end{axis}
            \end{scope}
            \begin{scope}[xshift=12cm]
                \begin{axis}[ legend style={at={(0.5,-0.2)}, anchor=north, legend columns=1, legend cell align=left},
                        alignedplot,
                        xlabel={$g$},
                        xmin=0, xmax=10,
                        ymin=-6, ymax=1.2,
                        title={$d)\quad E_{\mu,\nu,z}$},
                        clip=true
                    ]
                    \addplot[domain=-1:1.28098, samples=300,thick, black] 
                        ({x}, {x^2/2 -(2/3)* (x^3/4.5)+(3/4)* (x^4/80) + 0.5*ln(-0.17 + 2.53725 + (-x + x^2/4.5-x^3/80)) + 0.5*ln(-1.17 + 2.53725 +(-x + x^2/4.5-x^3/80))});
                    \addplot[domain=1.28098:10, samples=300,dashed, thick, black, dash phase=3pt] 
                        ({x}, {x^2/2 -(2/3)* (x^3/4.5)+(3/4)* (x^4/80) + 0.5*ln(-0.17 + 2.53725 + (-x + x^2/4.5-x^3/80)) + 0.5*ln(-1.17 + 2.53725 +(-x + x^2/4.5-x^3/80))});

                    \pgfmathdeclarefunction{EfuncD}{1}{%
                        \pgfmathparse{((#1)^2)/2 -(2/3)*((#1)^3)/4.5)+(3/4)*((#1)^4)/80) + 0.5*ln(-0.17 + 2.53725 + (-#1 + ((#1)^2)/4.5-((#1)^3)/80)) + 0.5*ln(-1.17 + 2.53725 +(-#1 + ((#1)^2)/4.5-((#1)^3)/80))}%
                    }
                    
                    \addplot[dashed,g2col] coordinates {(3.01905, -6) (3.01905, 1.2)};
                    \addplot[only marks,mark=x,mark size=1.8pt,g2col,thick] coordinates {(3.01905, EfuncD(3.01905)};
                    
                    \addplot[dashed, g3col] coordinates {(4.07941, -6) (4.07941, 1.2)};
                     \addplot[only marks,mark=x,mark size=1.8pt,g3col,thick] coordinates {(4.07941, EfuncD(4.07941)};
                     
                    \addplot[dashed,g4col] coordinates {(8.8328, -6) (8.8328, 1.2)};
                    \addplot[only marks,mark=x,mark size=1.8pt,g4col,thick] coordinates {(8.8328, EfuncD(8.8328)};
                \end{axis}
            \end{scope} 
            \begin{scope}[xshift=14cm,yshift=-0.8cm]
                                    \begin{axis}[
                                        hide axis,
                                        scale only axis,
                                        height=0cm,
                                        width=0cm,
                                        xmin=0, xmax=1,
                                        ymin=0, ymax=1,
                                        legend columns=4,
                                         legend style={cells={anchor=west}}
                                    ]
                    \addlegendimage{black,thick}
                    \addlegendentry{$E_{\mu,\nu,z}(g),g \in (0,g^*_{\mu\boxplus \nu})$}
                    \addlegendimage{black,dashed,thick}
                    \addlegendentry{$E_{\mu,\nu,z}(g),g \notin (0,g^*_{\mu\boxplus \nu})$}
                    \addlegendimage{g1col}
                    \addlegendentry{$U_{\mu \boxplus \nu}(z)$}
                    \addlegendimage{g1col,dashed}
                    \addlegendentry{$g_1(z)=G_{\mu\boxplus \nu}(z)$}
                    \addlegendimage{g2col,dashed}
                    \addlegendentry{$g_2(z)$}
                     \addlegendimage{dashed, g3col}
                    \addlegendentry{$g_3(z)$}
                    \addlegendimage{dashed,g4col}
                    \addlegendentry{$g_4(z)$}
                \end{axis}
            \end{scope}
            \begin{scope}[yshift=-6cm,xshift=\xshiftEmanycritpt]
                \begin{axis}[ legend style={at={(0.5,-0.2)}, anchor=north, legend columns=1, legend cell align=left},
                    alignedplot,
                    xlabel={$g$},
                    xmin=0, xmax=4,
                    ymin=-3, ymax=-2.6,
                    title={$c)\quad G_{\mu\boxplus\nu}^{\inv}$},
                    clip=true
                ]
                    \addplot[domain=0.1:1.17371, samples=300,thick, black] 
                        ({x}, {-x + x^2/6 + (-1 - 1.3399999999999999*x -    sqrt(-4*(0.6699999999999999 + 0.1989*x)*x + (1 + 1.3399999999999999*x)^2))/(2*x)});
                    \addplot[domain=1.03371:7.5, samples=300,dashed, thick, black] 
                        ({x},{-x + x^2/6 + (-1 - 1.3399999999999999*x -    sqrt(-4*(0.6699999999999999 + 0.1989*x)*x + (1 + 1.3399999999999999*x)^2))/(2*x)});

                    \pgfmathdeclarefunction{GinvC}{1}{%
                        \pgfmathparse{-#1 + (#1)^2/6 + (-1 - 1.3399999999999999*#1 - sqrt(-4*(0.6699999999999999 + 0.1989*#1)*#1 + (1 + 1.3399999999999999*#1)^2))/(2*#1)}%
                    }
                    
                    \addplot[dashed,black] coordinates {(-3, -2.69691) (7.5, -2.69691)};
                    \addplot[densely dotted,black] coordinates {(1.16103, -5) (1.16103, 0)};
                    \addplot[only marks,mark=x,mark size=1.8pt,black,thick] coordinates {(1.16103, -2.69691)};
                    %
                    
                    \addplot[densely dashdotted,black] coordinates {(-3, -2.77) (7.5, -2.77)};
                    
                    \addplot[dashed, g1col] coordinates {(0.830841, -5) (0.830841, 3.2)};
                    \addplot[only marks,mark=x,mark size=1.8pt,g1col,thick] coordinates {(0.830841, -2.77)};
                    
                    \addplot[dashed, g2col] coordinates {(1.7432, -5) (1.7432, 3.2)};
                    \addplot[only marks,mark=x,mark size=1.8pt,g2col,thick] coordinates {(1.7432, -2.77)};
                    
                    \addplot[dashed, g3col] coordinates {(3.0, -5) (3.0, 3.2)};
                    
                    \addplot[dashed, g4col] coordinates {(3.5912, -5) (3.5912, 3.2)};
                    \addplot[only marks,mark=x,mark size=1.8pt,g4col,thick] coordinates {(3.5912, -2.77)};
                \end{axis}
            \end{scope}
            \begin{scope}[xshift=\xshiftmusc,yshift=-6cm]
                \begin{axis}[ legend style={at={(0.5,-0.2)}, anchor=north, legend columns=1, legend cell align=left},
                    alignedplot,
                    xlabel={$g$},
                    xmin=0, xmax=2.5,
                    ymin=-4, ymax=-2,
                    title={$b)\quad G_{\mu\boxplus\nu}^{\inv}$},
                    clip=true
                ]
                    \addplot[domain=0.1:0.866025, samples=300,thick, black] 
                        ({x}, {-x + (-1 - sqrt(1 + 4*x^2))/(2*x)});
                    \addplot[domain=0.866025:4, samples=300, dashed, thick, black,dash phase=3pt]  ({x}, {-x + (-1 - sqrt(1 + 4*x^2))/(2*x)});

                    \addplot[dashed,black] coordinates {(-3, -2.59808) (7.5, -2.59808)};
                    \addplot[densely dotted,black] coordinates {(0.866025, -5) (0.866025, 0)};
                    \addplot[only marks,mark=x,mark size=1.8pt,black,thick] coordinates {(0.866025, -2.59808)};
            
                    \addplot[densely dashdotted,black] coordinates {(0, -3) (7.5, -3)};
                    
                    \addplot[dashed, g1col] coordinates {(0.467911, -6) (0.467911, -1.5)};
                    \addplot[only marks,mark=x,mark size=1.8pt,g1col,thick] coordinates {(0.467911, -3)};
                    
                    \addplot[dashed, g2col] coordinates {(1.6527, -6) (1.6527, -1.5)};                    
                    \addplot[only marks,mark=x,mark size=1.8pt,g2col,thick] coordinates {(1.6527, -3)};

                \end{axis}
            \end{scope}
            \begin{scope}[xshift=\xshiftmuMP,yshift=-6cm]
                \begin{axis}[ legend style={at={(0.5,-0.2)}, anchor=north, legend columns=1, legend cell align=left},
                    alignedplot,
                    xlabel={$g$},
                    xmin=0, xmax=4,
                    ymin=-3, ymax=-0.5,
                    title={$a)\quad G_{\mu\boxplus\nu}^{\inv}$},
                    clip=true
                ]
                    \addplot[domain=0.1:5, samples=300,thick, black] 
                        ({x}, {1/(2*(1 + x)) + (-1 - sqrt(1 + 4*x^2))/(2*x)});

                    \addplot[dashed,black] coordinates {(-3,-1) (7.5, -1)};
                
                    \addplot[densely dashdotted,black] coordinates {(0, -1.5) (4, -1.5)};
                    
                    \addplot[dashed, g1col] coordinates {(0.826462, -3) (0.826462, 2)};
                    \addplot[only marks,mark=x,mark size=1.8pt,g1col,thick] coordinates {(0.826462, -1.5)};
                \end{axis}
            \end{scope}
            \begin{scope}[xshift=12cm,yshift=-6cm]
                \begin{axis}[ legend style={at={(0.5,-0.2)}, anchor=north, legend columns=1, legend cell align=left},
                    alignedplot,
                    xlabel={$g$},
                    xmin=0, xmax=10,
                    ymin=-3.5, ymax=-1,
                    title={$d)\quad G_{\mu\boxplus\nu}^{\inv}$},
                    clip=true
                ]
                
                    \addplot[domain=0:1.28098, samples=300,thick, black] 
                        ({x}, {-x + 0.2222222222222222*x^2 - x^3/80 + 
                       (-1 - 1.3399999999999999*x - 
                        sqrt(-4*(0.6699999999999999 + 0.1989*x)*
                           x + (1 + 1.3399999999999999*x)^2))/(2*x)});
                    \addplot[domain=1.28098:10, samples=300,dashed, thick, black, dash phase=3pt] 
                        ({x}, {-x + 0.2222222222222222*x^2 - x^3/80 + 
                       (-1 - 1.3399999999999999*x - 
                        sqrt(-4*(0.6699999999999999 + 0.1989*x)*
                           x + (1 + 1.3399999999999999*x)^2))/(2*x)});

                    \addplot[dashed,black] coordinates {(-3, -2.63725) (20, -2.63725)};
                    \addplot[densely dotted,black] coordinates {(1.28098, -5) (1.28098, 0)};
                    \addplot[only marks,mark=x,mark size=1.8pt,black,thick] coordinates {(1.28098, -2.63725)};
            
                    \addplot[densely dashdotted,black] coordinates {(-1,-2.53725) (10, -2.53725)};
                    
                    \addplot[dashed, g2col] coordinates {(3.01905, -6) (3.01905, -0.8)}; 

                    \addplot[dashed, g3col] coordinates {(4.07941, -6) (4.07941, 1.2)};
                     \addplot[only marks,mark=x,mark size=1.8pt,g3col,thick] coordinates {(4.07941, -2.53725)};
                     
                    \addplot[dashed,g4col] coordinates {(8.8328, -6) (8.8328, -0.8)};

                \end{axis}
            \end{scope}
            \begin{scope}[xshift=13cm,yshift=-6.8cm]
                \begin{axis}[
                    hide axis,
                    scale only axis,
                    height=0cm,
                    width=0cm,
                    xmin=0, xmax=1,
                    ymin=0, ymax=1,
                    legend columns=5,
                    legend style={cells={anchor=west}}
                ]
                    \addlegendimage{black,thick}
                    \addlegendentry{$G^{\inv}_{\mu \boxplus \nu},g \in (0,g^*_{\mu\boxplus \nu})$}
                    \addlegendimage{black,dashed,thick}
                    \addlegendentry{$G^{\inv}_{\mu \boxplus \nu},g \notin (0,g^*_{\mu\boxplus \nu})$}
                    \addlegendimage{black,dashdotted}
                    \addlegendentry{$z$}
                    \addlegendimage{dashed,black}
                    \addlegendentry{$z^*_{\mu \boxplus \nu}$}
                    \addlegendimage{densely dashdotted,black}
                    \addlegendentry{$g^*_{\mu \boxplus \nu}$}
                    \addlegendimage{g1col,dashed}
                    \addlegendentry{$g_1(z)=G_{\mu\boxplus \nu}(z)$}
                    \addlegendimage{g2col,dashed}
                    \addlegendentry{$g_2(z)$}
                     \addlegendimage{dashed, g3col}
                    \addlegendentry{$g_3(z)$}
                    \addlegendimage{dashed,g4col}
                    \addlegendentry{$g_4(z)$}
                \end{axis}
            \end{scope}
            \begin{scope}[yshift=-12cm,xshift=\xshiftEmanycritpt]
                \begin{axis}[ legend style={at={(0.5,-0.2)}, anchor=north, legend columns=1, legend cell align=left},
                    alignedplot,
                    xlabel={$h$},
                    xmin=-2.2, xmax=-1.2,
                    ymin=-2.9, ymax=-2.6,
                    title={$c)\quad F_{\mu,\nu}$},
                    clip=true
                ]
                    \addplot[domain=-4.5:-1.2, samples=300,thick, black]
                        ({x}, {(1/6)*(-(1/(2*(-1.17 - x))) - 1/(2*(-0.17 - x)))^2 -  1/(2*(-1.17 - x)) - 1/(2*(-0.17 - x)) + x});
                
                    \addplot[densely dashdotted,black] coordinates {(-4.5, -2.77) (-1, -2.77)};
                    
                    \addplot[densely dotted,black] coordinates {(-1.76055, -5) (-1.76055, 2)};
                    
                    \addplot[dashed, black] coordinates {(-5,-2.69691) (-1,-2.69691)};

                    \addplot[only marks,mark=x,mark size=1.8pt,black,thick] coordinates {(-1.76055, -2.69691)};
                    
                    \addplot[dashed, g1col] coordinates {(-2.05421, -5) (-2.05421, 2)};
                    \addplot[only marks,mark=x,mark size=1.8pt,g1col,thick] coordinates {(-2.05421, -2.77)};
                    
                    \addplot[dashed, g2col] coordinates {(-1.53326, -5) (-1.53326, -2)};
                    \addplot[only marks,mark=x,mark size=1.8pt,g2col,thick] coordinates {(-1.53326, -2.77)};
                    
                    \addplot[dashed, g3col] coordinates {(-1.36371, -5) (-1.36371, -2)};
                    
                    \addplot[dashed, g4col] coordinates {(-1.32825, -5) (-1.32825, -2)};
                    \addplot[only marks,mark=x,mark size=1.8pt,g4col,thick] coordinates {(-1.32825, -2.77)};
                \end{axis}
            \end{scope}
            \begin{scope}[yshift=-12cm,xshift=\xshiftmusc]
                \begin{axis}[ legend style={at={(0.5,-0.2)}, anchor=north, legend columns=1, legend cell align=left},
                    alignedplot,
                    xlabel={$h$},
                    xmin=-3, xmax=-1,
                    ymin=-4, ymax=-2.3,
                    title={$b)\quad F_{\mu,\nu}$},
                    clip=true
                ]
                    \addplot[domain=-4:-1.1, samples=300,thick, black] 
                        ({x}, {(1/2)*(-(1/(-1 - x)) - 1/(1 - x)) + x});
                    
                    \addplot[densely dashdotted,black] coordinates {(-4, -3) (-0.7, -3)};
                    
                    \addplot[densely dotted, black] coordinates {(-1.73205, -6) (-1.73205, -1.5)};

                    \addplot[dashed, black] coordinates {(-4, -2.59808) (-0.7, -2.59808)};

                    \addplot[only marks,mark=x,mark size=1.8pt,black,thick] coordinates {(-1.73205, -2.59808)};

                    \addplot[dashed, g1col] coordinates {(-2.53209, -4) (-2.53209, -1.8)};
                    \addplot[only marks,mark=x,mark size=1.8pt,g1col,thick] coordinates {(-2.53209, -3)};
                    
                    \addplot[dashed, g2col] coordinates {(-1.3473, -4) (-1.3473, -1.8)};
                    \addplot[only marks,mark=x,mark size=1.8pt,g2col,thick] coordinates {(-1.3473, -3)};
                    
                \end{axis}
            \end{scope}
            \begin{scope}[yshift=-12cm,xshift=\xshiftmuMP]
                \begin{axis}[ legend style={at={(0.5,-0.2)}, anchor=north, legend columns=1, legend cell align=left},
                    alignedplot,
                    xlabel={$h$},
                    xmin=-2, xmax=-0.7,
                    ymin=-1.8, ymax=-0.8,
                    title={$a)\quad F_{\mu,\nu}$},
                    clip=true
                ]
                    \addplot[domain=-4:-0.3, samples=300,thick, black] 
                        ({x}, {1/(2*(1 + (1/2)*(1/(-1 - x) + 1/(1 - x)))) + x});
                
                    \addplot[densely dashdotted,black] coordinates {(-2, -1.5) (-0.7, -1.5)};
                    
                    \addplot[densely dotted, black] coordinates {(-1, -6) (-1, 0)};

                    \addplot[only marks,mark=x,mark size=1.8pt,black,thick] coordinates {(-1, -1)};                    
                    
                    \addplot[dashed, black] coordinates {(-2, -1) (-0.7, -1)};

                    \addplot[dashed, g1col] coordinates {(-1.77375, -6) (-1.77375, -0.6)};
                    \addplot[only marks,mark=x,mark size=1.8pt,g1col,thick] coordinates {(-1.77375, -1.5)};
                \end{axis}
            \end{scope}
            \begin{scope}[yshift=-12cm,xshift=12cm]
                \begin{axis}[
                        legend style={at={(0.5,-0.2)}, anchor=north, legend columns=1, legend cell align=left},
                        alignedplot,
                        xlabel={$h$},
                        xmin=-2, xmax=-1.1,
                        ymin=-2.8, ymax=-2.4,
                        title={$d)\quad F_{\mu,\nu}$},
                        clip=true
                    ]
                    \addplot[domain=-4.5:-1.2, samples=300,thick, black]
                        ({x}, {0.2222222222222222*(-(1/(2*(-1.17 - x))) - 1/(2*(-0.17 - x)))^2 + 
                       (1/80)*(-(1/(2*(-1.17 - x))) - 1/(2*(-0.17 - x)))^3 - 
                     1/(2*(-1.17 - x)) - 1/(2*(-0.17 - x)) + x}); 

                    \addplot[densely dashdotted,black] coordinates {(-4.5, -2.53725) (-1, -2.53725)};
                    
                    \addplot[densely dotted, black] coordinates {(-1.69464, -5) (-1.69464, 2)};

                    \addplot[dashed, black] coordinates {(-4.5, -2.63725) (-1, -2.63725)};

                   \addplot[only marks,mark=x,mark size=1.8pt,black,thick] coordinates {(-1.69464,-2.63725)};                    
                    
                    \addplot[dashed, g2col] coordinates {(-1.36233, -5) (-1.36233, -1)};

                    \addplot[dashed,g3col] coordinates {(-1.30737, -5) (-1.30737, -1)};
                    \addplot[only marks,mark=x,mark size=1.8pt,g3col,thick] coordinates {(-1.30737, -2.53725)};

                    \addplot[dashed,g4col] coordinates {(-1.2298, -5) (-1.2298, -1)};
                \end{axis}
            \end{scope}
            \begin{scope}[xshift=14cm,yshift=-12.9cm]
                \begin{axis}[
                        hide axis,
                        scale only axis,
                        height=0cm,
                        width=0cm,
                        xmin=0, xmax=1,
                        ymin=0, ymax=1,
                        legend columns=4,
                        legend style={cells={anchor=west}}
                    ]
                    \addlegendimage{black,thick}
                    \addlegendentry{$F^{\inv}_{\mu , \nu}(h),g \in (0,g^*_{\mu\boxplus \nu})$}
                    \addlegendimage{black,dashdotted}
                    \addlegendentry{$z$}
                    \addlegendimage{black,dashed}
                    \addlegendentry{$F_{\mu,\nu}(h^*_{\mu,\nu}) = z^*_{\mu\boxplus\nu}$}
                    \addlegendimage{black,densely dotted}
                    \addlegendentry{$h^*_{\mu,\nu} = G_\nu^{\inv}(g^*_{\mu \boxplus \nu})$}                    
                    \addlegendimage{g1col,dashed}
                    \addlegendentry{$G_\nu^{\inv}(g_1(z))$}
                    \addlegendimage{g2col,dashed}
                    \addlegendentry{$G_\nu^{\inv}(g_2(z))$}
                    \addlegendimage{dashed, g3col}
                    \addlegendentry{$G_\nu^{\inv}(g_3(z))$}
                    \addlegendimage{dashed,g4col}
                    \addlegendentry{$G_\nu^{\inv}(g_4(z))$}
                \end{axis}
            \end{scope}
        \end{tikzpicture}
    \end{minipage}
    \fi
\end{figure}

\begin{remark}[Additional discussion of Figure \ref{fig:intro_E_mu_nu_z_F_mu_nu_shapes1}]\suppressfloats
    \label{rem:fig_caption_overflow}
    The function $G^{[-1]}_{\mu \boxplus \nu}$ is the maximal real analytic extension of $G^{\inv}_{\mu \boxplus \nu}$ (see \eqref{eq:Grho_inv_ext_def}).

    \begin{enumerate}
        \item[a)] Note that this $\mu$ is the semicircle law $\mu = \musc[1]$.
        \item[b)] Note that this $\mu$ is the Marchenko-Pastur law $\mu_{\MP,1/2}$. The top plot shows that $g = G_{\mu \boxplus \nu}(z)$ is not the global minimum of $E_{\mu,\nu,z}$ in general.    
        \item[c)] The top plot shows that it is also not the lowest of all the local minima in general.    
        \item[d)] The top plot shows that $E_{\mu,\nu,z}$ can have local minima to the right of $g^*_{\mu \boxplus \nu}$, even when $z>z^*_{\mu \boxplus \nu}$. Note that $U_{\mu \boxplus \nu}(z)$ is not well-defined for this $z$.
    \end{enumerate}

    In each case, the $g_i(z)$ denote critical points of $E_{\mu,\nu,z}$, ordered right to left. In a,b,c) we have $g_1(z)=G_{\mu \boxplus \nu}(z)$, while in d) the quantity $G_{\mu \boxplus \nu}(z)$ is not well-defined and the critical points are denoted by $g_2(z),g_3(z),g_4(z)$. In the top row, the positions of the critical points are indicated by the coloured vertical lines. In the second row they are in the same position, and it can be observed that some of them correspond to solutions of $G^{[-1]}_{\mu \boxplus \nu}(g) = z$ - the ones that do not are instead solutions of $R_\mu'(-g) = 0$ (not shown). In the bottom row, the coloured lines indicate the position of $G_\nu^{\inv}(g_i(z))$, and it can be observed that the $g_i(z)$ that are solutions of $G^{[-1]}_{\mu \boxplus \nu}(g) = z$ also correspond to solutions of $F_{\mu,\nu}(h) = z$, and vice-versa.
\end{remark}

\fi



\section{Preliminaries on the  (real) $R$-transform and free additive convolution}\label{sect:R-trans}

In this section we define the $R$-transform as a real analytic function, which we call the ``real $R$-transform'', and give its relation to the free additive convolution. For this purpose we state the crucial Lemma \ref{lem:real_R_trans_extend} and Lemma \ref{lem:real_R_trans_add_cond}, but postpone their proofs until Section \ref{sect:R_transform_theory}. We also prove some important facts about the real $R$-transform, and about the Stieltjes transform as a real analytic function.

\subsection{Stieltjes transform restricted to the reals}
Recall that for any probability measure $\mu$ on $\mathbb{R}$, $\supp_{\pm}\mu$ is the sup resp. inf of the support of $\mu$. We write
\begin{equation}
    \overline{\supp}\,\mu
    =
    [ \supp_{-}\mu, \supp_{+}\mu ] \subset \mathbb{R}
\end{equation}
so that 
\begin{equation}\label{eq:supp_overline_complement}
    (\overline{\supp}\,\mu)^c
    =
    (-\infty,\supp_{-}\mu) \cup (\supp_{+}\mu,\infty).
\end{equation}
For any compactly supported $\mu$, it can be seen from its definition \eqref{eq:defSti} that the Stieltjes
transform $G_\mu$ restricted to $(\overline{\supp}\,\mu)^c$
is real-valued and real analytic, and
\begin{equation}\label{eq:G_mu_strict_inc}
    G_\mu
    \text{ is strictly increasing on } (-\infty,\supp_{-}\mu) \text{ and on } (\supp_{+}\mu,\infty).
\end{equation}
Furthermore $G_\mu$ is positive on $(-\infty,\supp_{-}\mu)$, and
negative on $(\supp_{+}\mu,\infty)$. Thus we can define
\begin{equation}\label{eq:G_mu_at_endpoints}
    G_\mu(\supp_{-}\mu) := \lim_{z\uparrow\supp_{-}\mu} G_\mu(z) \in (0,\infty]
    \quad\text{and}\quad
    G_\mu(\supp_{+}\mu) := \lim_{z\downarrow\supp_{+}\mu} G_\mu(z) \in [-\infty,0).
\end{equation}
In general, for any $D\subset\mathbb{R}$ and function 
\begin{equation}
    \begin{array}{c}\label{eq:f_well_defined_at_endpoints}
        f:D\to\mathbb{R}\text{ and }x\in{\rm cl}(D)\subset[-\infty,\infty]\text{, we say that}\text{ ``\ensuremath{f(x)} is well-defined'' and write \ensuremath{f(x):=\ensuremath{{\displaystyle {\lim_{y\to x}}f(y)}}}}\\
        \text{if that limit over \ensuremath{y\in D} exists as an extended real number in \ensuremath{[-\infty,\infty]}}.
    \end{array}
\end{equation}
Here the closure is taken in the topology of the extended real numbers, so e.g. ${\rm cl}([0,\infty))=[0,\infty)\cup\{\infty\}$.
In this sense a continuous function $f$ on $D\subset{\mathbb{R}}$ is well-defined for all $x\in D$, and e.g. if $f:\mathbb{R} \to \mathbb{R},f(t)=t$ then $f(\infty)$ is well-defined and equals $\infty$. Furthermore $G_\mu(\supp_{\pm}\mu)$ are well-defined by \eqref{eq:G_mu_at_endpoints}. Note however that strictly speaking we do not include $\supp_{\pm}\mu$ in the domain of $G_\mu$ (if we did then it would not be real analytic at $\supp_{\pm}\mu$, even when $G_\mu(\supp_{\pm}\mu)$ are finite - see Lemma \ref{lem:Gnotanal} below).

From \eqref{eq:defSti} it follows that
\begin{equation}\label{eq:G_mu_limit_z_to_inf}
    G_\mu(-\infty)
    =
    \lim_{z \downarrow -\infty} G_\mu(z)
    =
    G_\mu(\infty)
    =
    \lim_{z \uparrow \infty} G_\mu(z)
    =
    0.
\end{equation}
The above implies that
\begin{equation}\label{eq:G_mu_maps_bijectively}
    G_\mu\text{ maps }(-\infty,\supp_{-}\mu)\text{ to }(0,G_\mu(\supp_{-}\mu))\text{ bijectively},
\end{equation}
and $(\supp_{+}\mu,\infty)$ to $(G_\mu(\supp_{+}\mu),0)$ bijectively.
Thus $G_\mu|_{(\overline{\supp}\:\mu)^c}$ has a well-defined
inverse, which is real analytic and which we denote by
\begin{equation}
    G_\mu^{\inv}: (G_\mu(\supp_{+}\mu),0) \cup (0,G_\mu(\supp_{-}\mu)) \to 
    (\overline{\supp}\,\mu)^c.
\end{equation}
Recalling from \eqref{eq:def_Dhatmu_intro} the notation
\begin{equation}\label{eq:D_hat_def}
    \hat{D}_\mu
    :=
    (
        -G_\mu(\supp_{-}\mu)
        ,
        -G_\mu(\supp_{+}\mu)
    ),
\end{equation}
we see that $\dom(G_\mu^{\inv}) = -\hat{D}_\mu \setminus \{0\}$. By \eqref{eq:G_mu_maps_bijectively}
\begin{equation}\label{eq:Ginv_maps_bijecitvely}
    G_\mu^{\inv}
    \text{ maps }
    (0,G_\mu(\supp_{-}\mu))
    \text{\,to }
    (-\infty,\supp_{-}\mu)
    \text{ bijectively}
\end{equation}
and by \eqref{eq:G_mu_strict_inc}
\begin{equation}\label{eq:Ginv_increasing}
    G_\mu^{\inv}\text{ is strictly increasing on $(G_\mu(\supp_{+}\mu),0)$ and on $(G_\mu(\supp_{+}\mu),0)$}.
\end{equation}
This implies that
\begin{equation}\label{eq:Ginv_limit}
    G_\mu^{\inv}( G_\mu(\supp_{-}\mu) )
    \overset{\eqref{eq:f_well_defined_at_endpoints}}{=}
    \lim_{ g \uparrow G_\mu(\supp_{-}\mu) }
    G_\mu^{\inv}(g)
    =
    \supp_{-}\mu,
\end{equation}
(including if $G_\mu(\supp_{-}\mu) = \infty$).

\subsection{Definition of real $R$-transform}
We can thus define the function
\begin{equation}\label{eq:def_real_R_trans_hat}
    \hat{R}_\mu(u)
    =
    G_\mu^{\inv}(-u)
    -
    u^{-1},
    \quad\quad\quad\quad
    \hat{R}_\mu:\hat{D}_\mu\setminus\{0\}\to\mathbb{R},
\end{equation}
Note that $\hat{R}_\mu(g)$ is well-defined for $g = -G_\mu(\supp_{\pm}\mu)$ in the sense of \eqref{eq:f_well_defined_at_endpoints} (since $G_\mu^{\inv}(g)$ is - see \eqref{eq:Ginv_limit}), but again we do not claim that we obtain an analytic function if we include $g$ in the domain of $\hat{R}_\mu$ (sometimes we do, sometimes we don't). However, $\hat{R}_\mu$ always has a unique real analytic extension $R_\mu$ to a maximal interval of $\mathbb{R}$ (possibly equalling $\hat{D}_\mu$, if the extension is trivial), as stated by the following lemma. It is this uniquely defined extension $R_\mu$ which we call the \emph{real $R$-transform} of the probability measure $\mu$.

\begin{lemma}[Real $R$-transform]\label{lem:real_R_trans_extend}
    Let $\mu$ be a probability measure on $\mathbb{R}$ with compact support. There exists 
    \begin{equation}\label{[lem:real_R_trans_extend]eq:D_mu}
        \text{a unique open interval }D_\mu\subseteq\mathbb{R} \text{ s.t. }\hat{D}_\mu \subseteq D_\mu\text{, and}
    \end{equation}
    \begin{equation}
        \text{a unique real analytic function } R_\mu:D_\mu\to\mathbb{R},
    \end{equation}
    such that for any open interval $D\subset\mathbb{R}$ containing $\hat{D}_\mu$, and any real analytic extension $f:D\to\mathbb{R}$ of $\hat{R}_\mu$ from $\hat{D}_\mu \setminus \{0\}$ to $D$, in fact $D \subset D_\mu$, and $f=R_\mu|_D$.
\end{lemma}
Note that necessarily $0 \in D_\mu$ (recall \eqref{eq:D_hat_def}). Lemma \ref{lem:real_R_trans_extend} is proved in Section \ref{sect:R_transform_theory}.
\subsection{Free additive convolution in terms of real $R$-transform}
The next lemma, also proved in Section \ref{sect:R_transform_theory}, relates the real $R$-transform to the free additive convolution.
\begin{lemma}[Free additive convolution in terms of real $R$-transform]\label{lem:real_R_trans_add_cond}
    For any compactly supported probability measures $\mu$ and $\nu$ on $\mathbb{R}$, there exists a unique compactly supported probability measure $\rho$ s.t.
    \begin{equation}\label{eq:add_conv_R_trans_real}
        D_\mu\cap D_{\nu} \subset D_\rho
        \quad\quad\text{ and }\quad\quad
        R_\rho(u)=R_\mu(u)+R_\nu(u)\quad\text{for all}\quad u\in D_\mu\cap D_{\nu}.
    \end{equation}
\end{lemma}
We denote the uniquely defined $\rho$ from Lemma \ref{lem:real_R_trans_add_cond} by $\mu\boxplus\nu$, and call it the \emph{free additive convolution of $\mu$ and $\nu$}.
In Section \ref{sect:R_transform_theory}, Lemma \ref{lem:real_R_trans_extend} and Lemma \ref{lem:real_R_trans_add_cond} are derived from the standard theory of $R$-transforms.

By \cite[Lemma 6.1]{GM} it holds for any compactly supported $\mu$ and $\nu$ that
\begin{equation}\label{eq:add_conv_G_at_endpoint_ineq}
     G_{\mu\boxplus\nu}(\supp_{-}\mu\boxplus\nu)
     \le
     \min( G_\mu(\supp_{-}\mu), G_{\nu}(\supp_{-}\nu) .
\end{equation}
Recalling \eqref{eq:D_hat_def}, this implies that $\hat{D}_{\mu \boxplus \nu} \subset \hat{D}_\mu \cap \hat{D}_\nu$, which combined with Lemma \ref{lem:real_R_trans_extend} and Lemma \ref{lem:real_R_trans_add_cond} 
yields
\begin{equation}\label{eq:D_sandwich}
    \hat{D}_{\mu \boxplus \nu}
    \subset
    \hat{D}_\mu\cap \hat{D}_{\nu}
    \subset
    D_\mu \cap D_{\nu}
    \subset
    D_{\mu \boxplus \nu}.
\end{equation}

\subsection{Further basic properties of the Stieltjes transform and real $R$-transform}

Next we proceed to prove some useful facts about the Stieltjes transform restricted to the reals, and about the real $R$-transform.

The following generalizes \eqref{eq:G_mu_strict_inc}, \eqref{eq:G_mu_at_endpoints} and \eqref{eq:G_mu_limit_z_to_inf} to the $k$-th derivative $G^{(k)}_\mu$ of $G_\mu$.
\begin{lemma}\label{lem:Gmu-kth-deriv}
    Let $\rho$ be a probability measure on $\mathbb{R}$ and $k\ge0$. Then
    \begin{equation}\label{eq:Gmu_k-th_deriv}
        G_\rho^{(k)}(z)
        =
        k!\int\frac{1}{(x-z)^{k+1}}\rho(dx)
        \quad\text{for all}\quad
        z \notin \supp\,\rho.
    \end{equation}
    Furthermore $G_\rho^{(k)}$ is real-valued and real analytic on $(\overline{\supp}\,\rho)^c$, strictly monotone on $(\supp_+ \rho,\infty)$,
    \begin{equation}\label{eq:mu_k-th_deriv_pos_inc}
        \text{ positive and strictly increasing on } (-\infty,\supp_-\rho),
    \end{equation}    
    \begin{equation}\label{eq:G_lim_z_to_-inf}
        \text{ and }
        G^{(k)}_\rho(-\infty)
        \overset{\eqref{eq:f_well_defined_at_endpoints}}{=}
        \lim_{ z \downarrow -\infty} G_\rho^{(k)}(z) = 0.
    \end{equation}

\end{lemma}

\begin{proof}
    The identity \eqref{eq:Gmu_k-th_deriv} follows from the definition \eqref{eq:defSti} of $G_\rho$ and that $\partial^k_z(\l-z)^{-1}=k!(\l-z)^{k+1}$. The rest of the claims follow directly from \eqref{eq:Gmu_k-th_deriv}.
\end{proof}

Thanks to the monotonicity of $G_\rho^{(k)}$
\begin{equation}\label{eq:G_deriv_limits}
    G^{(k)}_\rho( \supp_{-} \rho ) = \lim_{z\uparrow\supp_{-}\rho}G_\rho^{(k)}(z) \in (0,\infty]
    \quad\text{for all}\quad
    k\ge0,
\end{equation}
(recall \eqref{eq:G_mu_at_endpoints}-\eqref{eq:f_well_defined_at_endpoints}). From \eqref{eq:Gmu_k-th_deriv} we see that $G_\rho^{(k)}(\supp_-\rho) < \infty$ iff $\rho(dx)/(x-z)^{k+1}$ is integrable in $(\supp_-\rho,\supp_-\rho+1)$, and that therefore for all $k\ge0$
\begin{equation}\label{eq:G^(k)_mu_deriv_inf_implies_higher_derivs_inf}
    G_\rho^{(k)}(\supp_-\rho) = \infty
    \quad \implies \quad
    G_\rho^{(l)}(\supp_-\rho) = \infty
    \text{ for all }l\ge k.
\end{equation}

Next we prove  that for any compactly supported measure, the associated $G_\rho$ is not analytic at the left boundary of the support. 

\begin{lemma}\label{lem:Gnotanal}
    Let $\rho$ be a non-degenerate probability measure on $\mathbb{R}$ s.t. $\supp_{-} \rho \ne -\infty$. Then $G_\rho|_{(\overline{\supp}\,\rho)^c}$
    cannot be analytically continued to any open interval containing $\supp_-\rho$. 
\end{lemma}

\begin{proof}
    
    Recall \eqref{eq:G_deriv_limits}. If 
    $G_\rho( \supp_{-} \rho ) = \infty$, then clearly $G_\rho$ can not be analytically continued to a neighbourhood containing $\supp_{-}\rho$. Similarly, if there is some $k\ge1$ such that 
    $G^{(k)}_\rho( \supp_{-} \rho ) = \infty$, then $G_\rho$ can not be
    analytically continued to such an interval. Thus we may w.l.o.g. assume
    that
    \begin{equation}
        G^{(k)}_\rho( \supp_{-} \rho ) \in (0,\infty)
        \quad\text{for all}\quad
        k\ge0.
    \end{equation}
    This implies that $G_\rho$ can be extended by continuity from $(-\infty, \supp_{-} \rho)$ to $(-\infty, \supp_{-} \rho]$, and is infinitely continuously differentiable on $(-\infty, \supp_{-} \rho]$.
    Consider now the formal power series
    \begin{equation}\label{eq:power_series}
        \Delta
        \to
        \sum_{k=0}^{\infty} \Delta^k \frac{G_\rho^{(k)}(\supp_{-}\rho)}{k!}.
    \end{equation}
    If $G_\rho$ can be analytically continued to an open
    set containing $\supp_{-}\rho$, then this power series must have
    a finite radius of convergence. Let
    \begin{equation}\label{eq:epsilon(delta)}
        \varepsilon(\delta):=\rho([\supp_{-}\rho,\supp_{-}\rho+\delta])
    \end{equation}
    and note that $\varepsilon(\delta)>0$ for all $\delta>0$, since by definition $\supp_{-}\rho = \inf \supp \rho$.
    For any $\delta>0$ 
    \begin{equation}
        G_\rho^{(k)}(\supp_{-}\rho)
        \overset{\eqref{eq:Gmu_k-th_deriv}}{\ge}
        \int_{\supp_{-}\rho}^{\supp_{-}\rho+\delta}
            \frac{k!}{(x-\supp_{-}\rho)^{k+1}}
        \rho(dx)
        \overset{\eqref{eq:epsilon(delta)}}{\ge}
        \frac{k!}{\delta^{k+1}} \varepsilon(\delta).
    \end{equation}
    Thus for any $\Delta>0$ it holds for $\delta=\Delta/2$ that
    \begin{equation}
        \sum_{k=0}^{\infty}\frac{\Delta^k}{k!}G_\rho^{(k)}(\supp_{-}\rho)
        \ge
        \sum_{k=0}^{\infty}\frac{\Delta^k}{k!}\frac{k!}{\delta^{k+1}}\varepsilon(\delta)
        =
        \frac{\varepsilon(\delta)}{\delta \Delta}\sum_{k=0}^{\infty}2^{k+1}
        =
        \infty.
    \end{equation}
    Therefore \eqref{eq:power_series} is not convergent for any $\Delta>0$.
    Thus $G_\rho$ can not be analytically continued to
    any open interval containing $\supp_{-}\rho$.
\end{proof}

The next lemma proves that the real $R$-transform $R_\rho$ is strictly increasing on the interval $\hat{D}_\rho$ from \eqref{eq:D_hat_def}.

\begin{lemma}\label{lemma:Rmu_increasing}
    For any non-degenerate compactly supported probability measure $\rho$ on $\mathbb{R}$, the real $R$-transform $R_\rho$ 
    satisfies $R_\rho'(t)>0$ for $t \in \hat{D}_\rho$.
\end{lemma}

\begin{proof}
    Recall from Lemma \ref{lem:real_R_trans_extend} that $R_\rho(t) = \hat{R}_\rho(t)$ for $t \in \hat{D}_\rho \setminus \{0\}$.     
    Thus by \eqref{eq:def_real_R_trans_hat}
    \be
        R_\rho'(t)
        =
        -G_\rho^{\inv\prime}(-t)
        +
        t^{-2}
        \quad\text{for all}\quad
        t \in \hat{D}_\rho \setminus \{0\}.
    \ee
    By the inverse function rule
    \begin{equation}\label{eq:Ginv_func_rule}
        G_\rho^{\inv\prime}(g) = \frac{1}{G_\rho'(G_\rho^{\inv}(g))}
        \quad\text{for}\quad
        g \in (G_\rho(\supp_{+}\rho),0) \cup (0,G_\rho(\supp_{-}\rho)).
    \end{equation}
    It follows from the above and \eqref{eq:Gmu_k-th_deriv} that
    \be\label{eq:derR}
        R_\rho'(t)
        =
        -\(
            \int
            \frac{ \rho(d\l) }{ ( \l - G_\rho^{\inv}(-t) )^2 }
        \)^{-1}
        +
        t^{-2}
        \quad\text{for all}\quad
        t \in \hat{D}_\rho \setminus \{0\}.
    \ee    
    The expression \eqref{eq:derR} is positive if and only if
    \be\label{eq:after_reorder}
        \int
        \frac{ \rho(d\l) }{ (\l-G_\rho^{\inv}(-t))^2 }
        -
        t^2>0\,. 
    \ee
    Under the change of variables $z = G_\rho^{\inv}(-t) \iff G_\rho(z) = -t $ for $ z \in (\overline{\supp}\, \rho)^c$ (which is valid by \eqref{eq:Ginv_maps_bijecitvely}), the expression \eqref{eq:after_reorder} turns into
    \be\label{eq:afterchange}
        \int
        \frac{ \rho(d\l) }{ (\l-z)^2 }
        -
        \(
            \int\frac{ \rho(d\l) }{ \l - z }
        \)^2\,.
    \ee
    Finally \eqref{eq:afterchange} is non-negative for any $\rho$, with equality only if $\rho$ is degenerate. This proves that $R_\rho'(t)>0$ for all $t \in \hat{D}_\rho \setminus \{0\}$, which implies the claim.
\end{proof}


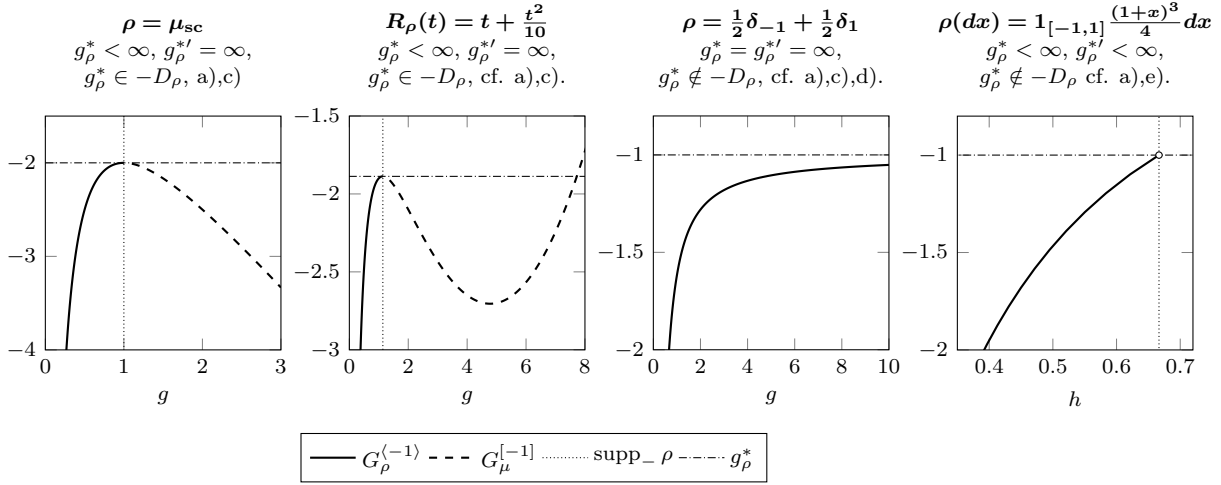
\begin{figure}
    \begin{minipage}{\textwidth}
        \begin{tikzpicture}
            \begin{scope}
                \begin{axis}[
                    legend style={at={(0.5,-0.2)}, anchor=north, legend columns=1, legend cell align=left},
                    alignedplot,fourcolwidth,
                    xlabel={$g$},
                    xmin=0, xmax=3,
                    ymin=-4, ymax=-1.5,
                    title={$\boldsymbol{\rho=\mu_{{\rm sc}}}$\\$g^*_\rho<\infty$, $g^{*\prime}_\rho=\infty$,\\
                    $g^*_\rho \in -D_\rho$, \ref{item:1[lem:Ginv_props]},\ref{item:2[lem:Ginv_props]}},
                    title style={align=center},
                    clip=true
                ]
                
                \addplot[densely dashdotted,black]  coordinates {(0, -2) (3,- 2)};
                
                \addplot[densely dotted, black] coordinates {(1, -5) (1,3)};
                
                \addplot[domain=0.01:1, samples=300,thick, black, dash phase=3pt] ({x}, {-x - 1/x});
                    
                \addplot[domain=1:3,samples=300,dashed, thick, black, dash phase=3pt] ({x}, {-x - 1/x});
                
                \end{axis}
            \end{scope}
            \begin{scope}[xshift=4cm]
                \begin{axis}[ legend style={at={(0.5,-0.2)}, anchor=north, legend columns=1, legend cell align=left},
                    alignedplot,fourcolwidth,
                    xlabel={$g$},
                    xmin=0, xmax=8,
                    ymin=-3, ymax=-1.5,
                    title={$\boldsymbol{R_\rho(t) = t + \frac{t^2}{10}}$\\$g^*_\rho<\infty$, $g^{*\prime}_\rho=\infty$,\\
                    $g^*_\rho \in -D_\rho$, cf. \ref{item:1[lem:Ginv_props]},\ref{item:2[lem:Ginv_props]}.},
                    title style={align=center},
                    clip=true
                ]
                \addplot[densely dashdotted,black]  coordinates {(-1, -1.88723) (10,-1.88723)};
                
                \addplot[densely dotted, black] coordinates {(1.13781, -5) (1.13781,3)};
                
                \addplot[domain=0.01:1.13781, samples=300,thick, black] 
                    ({x}, { -x + x*x/10  - 1/x });
                    
                \addplot[domain=1.13781:10, samples=300,dashed, thick, black, dash phase=3pt] 
                    ({x}, {-x + x*x/10  - 1/x });
                
                \end{axis}
            \end{scope}
            \begin{scope}[xshift=8cm]
                \begin{axis}[ legend style={at={(0.5,-0.2)}, anchor=north, legend columns=1, legend cell align=left},
                    alignedplot,fourcolwidth,
                    xlabel={$g$},
                     xmin=0, xmax=10,
                    ymin=-2, ymax=-0.8,
                    title={$\boldsymbol{\rho=\frac{1}{2}\delta_{-1}+\frac{1}{2}\delta_1}$\\ $g^*_\rho=g^{*\prime}_\rho=\infty$,\\
                    $g^*_\rho \notin -D_\rho$, cf. \ref{item:1[lem:Ginv_props]},\ref{item:2[lem:Ginv_props]},\ref{item:4[lem:Ginv_props]}.},
                    title style={align=center},
                    clip=true
                ]
                    \addplot[densely dashdotted,black] coordinates {(0, -1) (10, -1)};
                    
                    \addplot[domain=0.01:10, samples=300,thick, black]
                         ({x}, {(-1 - sqrt( 1 + 4*x*x ))/(2*x)});
                \end{axis}
            \end{scope}
            \begin{scope}[xshift=12cm]
                \begin{axis}[ legend style={at={(0.5,-0.2)}, anchor=north, legend columns=1, legend cell align=left},
                    alignedplot,fourcolwidth,
                    xlabel={$h$},
                     xmin=0.35, xmax=0.72,
                    ymin=-2, ymax=-0.8,
                    title = {$\boldsymbol{\rho(dx)=1_{[-1,1]}\frac{(1+x)^3}{4}dx}$\\$g^*_\rho<\infty$, $g^{*\prime}_\rho<\infty$,\\ $g^*_\rho \notin -D_\rho$ cf. \ref{item:1[lem:Ginv_props]},\ref{item:5[lem:Ginv_props]}.},
                    title style={align=center},
                    clip=true
                ]
                    \addplot[densely dotted, black] coordinates {(0.666667 , -5) (0.666667, 2)};
                    
                    \addplot[densely dashdotted,black] coordinates {(-1, -1) (5, -1)};
                    
                    \addplot[ domain=-30:-1.001, samples=300, thick, black] 
                    ( {5./3. + (3./2.)*x + (1./2.)*x*x+ (1./4.)*(1 + x)*(1 + x)*(1 + x)*ln((1 -x)/(-1 - x)) },{x});
                    
                    \addplot[mark=*, mark options={fill=white, scale=0.6}] coordinates {({0.666667},{-1})};
                \end{axis}
            \end{scope}
            \begin{scope}[xshift=9.5cm,yshift=-1.1cm]
                \begin{axis}[
                    hide axis,
                    scale only axis,
                    height=0cm,
                    width=0cm,
                    xmin=0, xmax=1,
                    ymin=0, ymax=1,
                    legend columns=4,
                ]
                    \addlegendimage{thick}
                    \addlegendentry{$G_\rho^{\inv}$}
                    \addlegendimage{dashed,thick}
                    \addlegendentry{$G_\mu^{[-1]}$}
                    \addlegendimage{densely dotted,black}
                    \addlegendentry{$\supp_{-}\rho$}
                    \addlegendimage{densely dashdotted,black}
                    \addlegendentry{$g^*_\rho$}
                \end{axis}
            \end{scope}
        \end{tikzpicture}
    \end{minipage}
    \caption{Examples of different possibilities for the ``shapes'' of $G_\rho^{\inv}$ and $G_\rho^{[-1]}$, cf. Lemma \ref{lem:Ginv_props}.}
    \label{fig:Ginv}
\end{figure}

Recalling \eqref{eq:D_mu_plus_def}, note that the previous lemma implies that
\begin{equation}\label{eq:Dhatmu_subset_D+mu}
    \hat{D}_\mu  \subset  D^+_\mu
    \text{ for any non-degenerate compactly supported }
    \mu.    
\end{equation}    

The following lemma is immediate from the definitions \eqref{eq:log-P} of $U_\rho$ and \eqref{eq:defSti} of $G_\rho$.
\begin{lemma}\label{lem:U_rho_deriv_and_asymp}
    For any probability measure $\rho$ on $\mathbb{R}$ s.t. $\supp_{-}\rho \ne -\infty$
    \begin{equation}
        U_\rho'(z)=-G_\rho(z)
        \quad\text{for}\quad z<\supp_{-}\rho,
        \quad\text{and}\quad
        \lim_{z\to-\infty}\{U_\rho(z)-\log|z|\}=0.
    \end{equation}
\end{lemma}


\section{Characterization of left endpoint of support (Proof of Theorem \ref{thm:add_conv_supp_end_point})}
\label{sect:end_point_proof}

In this section we prove Theorem \ref{thm:add_conv_supp_end_point}, our characterization of the left endpoint $z^*_{\mu \boxplus \nu}=\supp_{-}\mu\boxplus\nu$ and ${g^*_{\mu \boxplus \nu}=G_{\mu \boxplus \nu}(z^*_{\mu \boxplus \nu})}$. It follows from mostly elementary properties of the inverse $G_{\mu\boxplus\nu}^{\inv}$. The latter is related to $R_{\mu \boxplus \nu}$ via \eqref{eq:def_real_R_trans_hat}, and thus to $R_\mu$ and $R_\nu$ via \eqref{eq:add_conv_R_trans_real}.

The real $R$-transform $R_\mu$ gives rise to a real analytic extension of $G_\mu^{\inv}$, which we denote by $G_\mu^{[-1]}$ and now define. For any compactly supported probability measure $\rho$ on $\mathbb{R}$, define
\begin{equation}\label{eq:Grho_inv_ext_def}
    G_\rho^{[-1]}:-D_\rho\setminus\{0\}\to\mathbb{R}
    \quad\quad\text{by}\quad\quad
    G_\rho^{[-1]}(u)=R_\rho(-u)-u^{-1}.
\end{equation}
Then by \eqref{eq:def_real_R_trans_hat}
\begin{equation}\label{eq:Ginv_ext_restricted}
    G_\rho^{[-1]}|_{-\hat{D}_\rho\setminus\{0\}}
    =
    G_\rho^{\inv},
\end{equation}
so indeed $G_\rho^{[-1]}$ is a real analytic extension of $G_\rho^{\inv}$ from $-\hat{D}_\rho\setminus\{0\}$
to $-D_\rho\setminus\{0\}$ (recall Lemma \ref{lem:real_R_trans_extend}).

Let us furthermore abbreviate
\begin{equation}\label{eq:G_rho_at_endpoint_abbrev}
    g^*_\rho := G_\rho( \supp_{-} \rho ) \in (0,\infty],
    \quad\quad\text{ and }\quad\quad
    g^{*\prime}_\rho := G_\rho'( \supp_{-} \rho ) \in (0,\infty],
\end{equation}
extending the notation \eqref{eq:zstar_gstar_def}.

The next lemma describes how the ``shape'' of $G_\rho^{\inv}$ and $G_\rho^{[-1]}$ depends on the values of $g^*_\rho$ and $g^{*\prime}_\rho$. The result is best understood geometrically, noting that if $(x,f(x)),x\in[a,b]$ is the graph of a real function, then the graph of its inverse is $(f(x),x),x\in[a,b]$. Thus for instance if $f'(b) = \infty$, it is clear that the inverse function $f^{\inv}(y)$ must have vanishing derivative at the point $y=f(b)$ of its domain. This is the geometric meaning of Lemma \ref{lem:Ginv_props} \ref{item:2[lem:Ginv_props]}. In fact Lemma \ref{lem:Ginv_props} \ref{item:1[lem:Ginv_props]}, \ref{item:2[lem:Ginv_props]}, \ref{item:4[lem:Ginv_props]} all have similarly simple geometric interpretations. See Figure \ref{fig:Ginv} for some illustrative examples.

\begin{lemma}\label{lem:Ginv_props}
    For any compactly supported probability measure $\rho$ on $\mathbb{R}$ the following holds.

    \begin{enumerate}[label=\alph*)]
        \item\label{item:1[lem:Ginv_props]} For all $g \in (0,g^*_\rho)$ it holds that $G_\rho^{\inv\prime}(g) > 0$.
        
        \item\label{item:well_def_at_gstar[lem:Ginv_props]}$G_\rho^{\inv\prime}(g^*_\rho)$ and $G_\rho^{[-1]\prime}(g^*_\rho)$ are well-defined (in the sense of \eqref{eq:f_well_defined_at_endpoints}).
        \item\label{item:2[lem:Ginv_props]} If $g^{*\prime}_\rho=\infty$, then $G_\rho^{\inv\prime}(g^*_\rho) = G_\rho^{[-1]\prime}(g^*_\rho) = 0$ (including if $g^*_\rho=\infty$).
        \item\label{item:4[lem:Ginv_props]}
        If $g^*_\rho = 
        \infty$, then $g^{*\prime}_\rho = \infty$ and 
        $(0,\infty)\cap -D_\rho = (0,\infty)\cap -\hat{D}_\rho = (0,\infty)$.
        \item\label{item:5[lem:Ginv_props]} If $g^{*\prime}_\rho<\infty$, then $g^*_\rho<\infty$ and
        $(0,\infty)\cap -D_\rho = (0,\infty)\cap -\hat{D}_\rho = (0,g^*_\rho)$.

    \end{enumerate}
\end{lemma}

\begin{proof}
    {\bf \ref{item:1[lem:Ginv_props]}} This follows from the inverse function rule \eqref{eq:Ginv_func_rule}, and the case $k=1$ of \eqref{eq:mu_k-th_deriv_pos_inc}.
    \,\newline

    {\bf \ref{item:well_def_at_gstar[lem:Ginv_props]}} If $g^*_\rho \in -D_\rho$ then $G_\rho^{[-1]\prime}(g^*_\rho)$ is well-defined in the usual sense. Otherwise, it follows from  \ref{item:1[lem:Ginv_props]} that both $G_\rho^{[-1]\prime}(g^*_\rho)$ and $G_\rho^{\inv\prime}(g^*_\rho)$ are well-defined in the sense of \eqref{eq:f_well_defined_at_endpoints}.
    \,\newline
    
    {\bf \ref{item:2[lem:Ginv_props]}}
    Note that    \begin{equation}\label{eq:using_inv_func_rule}
        G_\rho^{\inv}(g^*_\rho)
        \overset{\eqref{eq:f_well_defined_at_endpoints}}{=}
        \lim_{g \uparrow g^*_\rho} G^{\inv\prime}(g)
        \overset{\eqref{eq:Ginv_func_rule}}{=}
        \lim_{g \uparrow g_\rho^*}
        \frac{1}{G_\rho'(G_\rho^{\inv}(g))}
        \overset{\eqref{eq:Ginv_limit}}{=}
        \lim_{z \uparrow \supp_{-}\rho}
        \frac{1}{G_\rho'(z)},       
    \end{equation}
    and similarly for $G_\mu^{[-1]}(g^*_\rho)$. Since $G_\rho'(z) \to g^{*\prime}_\rho$ as $z \uparrow \supp_-\rho$, it follows that \eqref{eq:using_inv_func_rule} equals zero if $g^{*\prime}_\rho = \infty$.
    \,\newline

    {\bf \ref{item:4[lem:Ginv_props]}} Assume $g^*_\rho=\infty$. Then $g^{*\prime}_\rho=\infty$ by \eqref{eq:G^(k)_mu_deriv_inf_implies_higher_derivs_inf}. Furthermore $(-\infty,0) \subset \hat{D}_\rho$ by \eqref{eq:D_hat_def}, which implies that also $(-\infty,0) \subset D_\rho$ by \eqref{[lem:real_R_trans_extend]eq:D_mu}, and proves the second claim.
    \,\newline
    
    {\bf \ref{item:5[lem:Ginv_props]}}
    Assume $g^{*\prime}_\rho<\infty$. Then $g^*_\rho<\infty$ by \eqref{eq:G^(k)_mu_deriv_inf_implies_higher_derivs_inf}.
    The second claim follows if we prove that $\sup(-D_\rho) = g^*_\rho$. To this end, assume for contradiction that $\sup(-D_\rho)>g^*_\rho$. Then $G_\rho^{[-1]}$ is an analytic extension of $G_\rho^{\inv}$ to an interval that contains a neighbourhood of $g^*_\rho$. It furthermore satisfies
    \begin{equation}
        G_\rho^{[-1]}(g^*_\rho)
        \overset{\eqref{eq:Ginv_ext_restricted}}{=}
        \lim_{g \uparrow g_\rho^*} G^{\inv}{}(g)
        \overset{\eqref{eq:Ginv_limit}}{=}
        \supp_{-}\rho
        \quad\text{and}\quad
        G_\rho^{[-1]\prime}(g^*_\rho)
        \overset{\eqref{eq:Ginv_ext_restricted}}{=}
        \lim_{g \uparrow g_\rho^*} G^{\inv\prime}(g)
        \overset{\eqref{eq:using_inv_func_rule}}{=}
        \frac{1}{g^{*\prime}_\rho} \ne 0.
    \end{equation}
    Thus, by the analytic inverse function theorem, $G_\rho^{[-1]}$
    has an analytic inverse $H$ in a neighbourhood of $g^*_\rho$
    s.t. $H(\supp_{-}\rho)=g^*_\rho$. Then for a small enough $\varepsilon>0$,
    $H|_{(\supp_{-}\rho-\varepsilon,\supp_{-}\rho+\varepsilon)}$ is an
    analytic extension of $G_\rho|_{(\supp_{-}\rho-\varepsilon,\supp_{-}\rho)}$. But this contradicts Lemma \ref{lem:Gnotanal}. Thus $\sup(-D_\rho)=g^*_\rho$.
\end{proof}

Later when we prove Proposition \ref{prop:Fmunu_props} \ref{item:BES_conds[prop:Fmunu_props]}, we will need the following result about the shape of $G_\mu^{[-1]}$ when $\mu$ has the left edge behaviour of the semicircle law.

\begin{lemma}\label{lem:Ginv_sc_edge}\,
    \begin{enumerate}[label=\alph*)]
        \item \label{item:sc_Ginv_loc_max[lem:Ginv_sc_edge]} Let $\beta>0$ and $\mu = \musc[\b]$. It holds that $g_{\musc[\beta]}^* = \beta$ and $g_{\musc[\beta]}^{*\prime} = \infty$. Furthermore the domain of $G_{\musc[\beta]}^{[-1]}$ is $\mathbb{R} \setminus \{0\}$, and its unique critical points in the domain are $\pm g_{\musc[\beta]}^* = \pm \beta$, for which $G_{\musc[\beta]}^{[-1]\prime\prime}(\pm \beta) = \mp2\beta^{-3}$. In particular $g = \beta$ is a local maximum of $G_{\musc[\beta]}^{[-1]}$.
    
        \item\label{item:sc_edge_behaviour[lem:Ginv_sc_edge]} Let $\rho$ be a probability s.t. $z_\rho^* := \supp_- \rho \ne -\infty$, and assume that there exist constants $c_1,c_2>0$ s.t.
        \begin{equation}\label{eq:edge_cond[lem:Ginv_sc_edge]}
            c_1^{-1}\varepsilon^{3/2}
            \le
            \rho(z_\rho^*,z_\rho^*+\varepsilon)
            \le
            c_1\varepsilon^{3/2}
            \text{ for }
            \varepsilon\in(0,c_2).
        \end{equation}
         Then (i) $g_\rho^*<\infty=g_\rho^{*\prime}$ and (ii) $G_\rho^{[-1]\prime\prime}(g_\rho^*)<0=G_\rho^{[-1]\prime}(g_\rho^*)$. In particular $g_\rho^*$ is a local maximum of $G_\rho^{[-1]}$.
    \end{enumerate}
\end{lemma}

\begin{remark}
    The left edge behaviour assumed in \ref{item:sc_edge_behaviour[lem:Ginv_sc_edge]} is that of the semicircle law. For measures with a density $f(x)$, it corresponds to $f(x) \propto (x - z^*_\rho)^{\frac{1}{2}}$ as $x \downarrow z^*_\rho$. It is a common situation for additive convolutions, see Remark \ref{rem:BES}.
\end{remark}

\begin{proof}
    {\bf \ref{item:sc_Ginv_loc_max[lem:Ginv_sc_edge]}} This follows from explicit computation, since $G_{\musc[\beta]}^{[-1]}(g) = - \beta^2 g - g^{-1}$ by \eqref{eq:Grho_inv_ext_def} and \eqref{eq:R_trans_sc_special_case}.
    \,\newline
    
    {\bf \ref{item:sc_edge_behaviour[lem:Ginv_sc_edge]}} Write $\musc=\musc[1]$ for the semicircle law supported on $[-2,2]$. Because of the decay of its density \eqref{eq:sc_def} as $z \downarrow -2$, $\musc$ satisfies the condition \eqref{eq:edge_cond[lem:Ginv_sc_edge]} for some $c_1,c_2$. Since $\rho$ is also assumed to satisfy it, there exist $c_3,c_4>0$ s.t.
    \begin{equation}
        c_3^{-1}\musc(-2,-2+\varepsilon)
        \le
        \rho(z_\rho^*,z_\rho^*+\varepsilon)
        \le
        c_3\musc(-2,-2+\varepsilon)
        \quad\text{ for }\quad
        \varepsilon \in (0,c_4).
    \end{equation}
    This implies that for any monotone decreasing function $f(x)\ge0$ and $\varepsilon \in (0,c_4)$
    \begin{equation}
        c_3^{-1} \int_{-2}^{-2+\varepsilon} f(x - (-2)) \musc(dx)
        \le
        \int_{z_\rho^*}^{z_\rho^*+\varepsilon} f( x - z^*_\rho ) \rho( dx)
        \le
        c_3 \int_{-2}^{-2+\varepsilon} f(x - (-2)) \musc(dx).
    \end{equation}
    Using this with $f(y) = (y+\delta)^{-k}$ for arbitrary $\delta>0,k\ge1$, and arbitrary $\varepsilon \in (0,c_4)$, we obtain that
    \begin{equation}\label{eq:G^(k)_mu_edge_sandwich}
        c_3^{-1}G_{\musc}^{(k)}( -2 - \delta ) - c_5
        \le
        G_\rho^{(k)}( z_\rho^* - \delta )
        \le
        c_3 G_{\musc}^{(k)}( -2 - \delta ) + c_5
        \quad\text{ for }\quad
        \delta>0,
    \end{equation}
    for a constant $c_5>0$ depending only on $\varepsilon$ and $k$.
    Taking the limit $\delta\downarrow0$ for $k=0,1$ and using that $g_{\musc}^*<\infty=g_{\musc}^{*\prime}$, we deduce that $g_\rho^*<\infty=g_\rho^{*\prime}$, proving (i).    
    
    Next, $G_\rho^{\inv \prime\prime}$ is monotone (recall \eqref{eq:mu_k-th_deriv_pos_inc}), so
    \begin{equation}
        G_\rho^{[-1]\prime\prime}( g^*_\rho )
        =
        \lim_{g \uparrow g^*_\rho}G_\rho^{\inv \prime\prime}(g).
    \end{equation}   
    By the inverse function rule
    \begin{equation}
        G_\rho^{\inv \prime\prime}(g)
        =
        -\frac{G_\rho''(G_\rho^{\inv}(g))}{G_\rho'(G_\rho^{\inv}(g))^{3}}
        \quad
        \text{ for }
        g < g^*_\rho.
    \end{equation}
    Thus
    \begin{equation}\label{eq:G_inv''_at_edge_from_inverse_func_rule}
        G_\rho^{[-1]\prime\prime}(g_\rho^*)
        =
        \lim_{\delta \downarrow 0 }
        \left(
            -\frac{
                G_\rho''(z^*_\rho - \delta)
            }{
                G_\rho'(z^*_\rho - \delta)^{3}
            }
        \right).
    \end{equation}
    Using that $g_{\musc}^{*\prime} = G_{\musc}''(z^*_\rho) = \infty$ (recall \eqref{eq:G_deriv_limits}) and \eqref{eq:G^(k)_mu_edge_sandwich} with $k=1,2$, we deduce that $G_\rho'(z^*_\rho - \delta),G_\rho''(z^*_\rho - \delta) \to \infty$ as $\delta \downarrow 0$. Putting this information back into \eqref{eq:G^(k)_mu_edge_sandwich}, we conclude that $G_\rho''(z^*_\rho - \delta) \ge \frac{1}{2}c_3^{-1} G_{\musc}''(-2-\delta)$ and $G_\rho'(z^*_\rho - \delta) \le 2 c_3 G_{\musc}'(-2 - \delta)$ for $\delta>0$ small enough. Therefore by \eqref{eq:G_inv''_at_edge_from_inverse_func_rule}
    \begin{equation}
        G_\rho^{[-1]\prime\prime}(g_\rho^*)
        \le
        c_6 \limsup_{\delta \downarrow 0 }
        \left(
            -\frac{
                G_{\musc}''(z^*_{\musc} - \delta)
            }{
                G_{\musc}'(z^*_{\musc} - \delta)^{3}
            }
        \right)
    \end{equation}
    for a constant $c_6>0$. The $\limsup$ on the r.h.s equals $G_{\musc}^{[-1]\prime\prime}(g_\rho^*)<0$ (because of the aforementioned inverse function rule), so (ii) follows.
\end{proof}

The next result shows how to deduce $\supp_- \rho$ and $g^*_\rho$ from the ``shape'' of $G_\rho^{[-1]}$. It is a corollary of the previous lemma, and will be used later with $\rho = \mu \boxplus \nu$ to characterize $z^*_{\mu \boxplus \nu}=\supp_-\mu \boxplus \nu$ and $g_{\mu \boxplus \nu}^*$ in terms of $R_\mu$ and $R_\nu$. This will be possible since
\begin{equation}\label{eq:G_inv_ex_and_J_mu_nu}
    G_{\mu\boxplus\nu}^{[-1]}(g)
    \overset{\eqref{eq:Grho_inv_ext_def}}{=}
    R_{\mu\boxplus\nu}(-g)-g^{-1}
    \overset{\eqref{eq:add_conv_R_trans_real}}{=}    
    R_\mu(-g) + R_\nu(-g) - g^{-1}
    \qquad
    \forall g \in (0,\infty) \cap -D_\mu \cap -D_{\nu}.
\end{equation}
Below we will use
\begin{equation}\label{eq:J_mu_nu_def}
    J_{\mu,\nu}
    :
    (0,\infty)
    \cap
    -(D_\mu \cap D_{\nu})
    \to
    \mathbb{R},
    \quad\quad
    J_{\mu,\nu}(g)
    :=
    R_\mu(-g)
    +
    R_\nu(-g)
    -
    g^{-1},
\end{equation}
for which
\begin{equation}\label{eq:Jmu_nu_and_Gplus_inv}
    J_{\mu,\nu}(g) \overset{\eqref{eq:G_inv_ex_and_J_mu_nu}}{=} G_{\mu\boxplus\nu}^{[-1]}(g)
    \qquad
    \forall g \in \dom(J_{\mu,\nu}).
\end{equation}
A subtle technical point is that we cannot exclude the possibility that $J_{\mu,\nu}$ has a smaller domain than $G_\muplusnu^{[-1]}|_{(0,\infty)}$. For this reason Corollary \ref{cor:Ginv_and_endpoint} below defines
\begin{equation}
    \text{ $J_\rho$ as the restriction $G_\rho^{[-1]}|_{(0,\infty)\cap-D}$
    for an arbitrary interval $D$ s.t. $\hat{D}_\mu \subset D\subset D_\mu$ }
\end{equation}
(see \eqref{eq:J_rho} below and recall \eqref{eq:Ginv_ext_restricted}). As the previous lemma, Corollary \ref{cor:Ginv_and_endpoint} is best understood geometrically. Note that each plot in Figure \ref{fig:Ginv} shows the aforementioned $J_\rho$ - the black curve correspond to $J_\rho|_{-\hat{D}_\rho}$, and depending on the choice of $D$ part or all of the dashed curve represent the rest of $J_\rho$. In the plots
\begin{equation}\label{eq:g^*_rho_geometric}
    \begin{array}{c}
        \text{$g^*_\rho$ is either the leftmost critical point of $J_\rho$, if $J_\rho$ has a critical point,}\\
        \text{and otherwise $J_\rho = G_\rho^{[-1]}|_{(0,\infty)} = G_\rho^{\inv}|_{(0,\infty)}$ and $g^*_\rho$ is the right endpoint of }\dom(J_\rho).
    \end{array}
\end{equation}
Corollary \ref{cor:Ginv_and_endpoint} b) below is a compact way to express this characterization of $g^*_\rho$ as a formula, using the set $\mathcal{G}_\rho$ in \eqref{eq:G_rho_at_endpoint_formula}. Note furthermore that in all plots of Figure \ref{fig:Ginv}, the endpoint $\supp_- \rho$ equals $J_\rho(g^*_\rho)$ (which is defined as in \eqref{eq:f_well_defined_at_endpoints} if $g^*_\rho \notin -D_\rho$). This is expressed in Corollary \ref{cor:Ginv_and_endpoint} c) below.

\begin{corollary}\label{cor:Ginv_and_endpoint}
    For any compactly supported probability
    measure $\rho$ on $\mathbb{R}$ and interval $D$ s.t. 
    $\hat{D}_\rho \subset D \subset D_\rho$,
    let 
    \begin{equation}\label{eq:J_rho}
        J_\rho:(0,\infty)\cap-D\to\mathbb{R},
        \quad\quad\quad
        J_\rho(g)
        =
        R_\rho(-g)-g^{-1},     
    \end{equation}
    \begin{equation}\label{eq:G_rho_at_endpoint_formula}
        \mathcal{G}_\rho
        =
        \left\{
            g\in\dom(J_\rho):J_\rho'(r)>0\text{ for all }r\in(0,g)
        \right\}.
    \end{equation}
    Then the following holds.
    
    a) $J_\rho(g)$ is increasing for $g$ close enough to $0$.
    \begin{equation}
        \text{b) }
        g^*_\rho
        =
        \sup\mathcal{G}_\rho
        =
        \underset{g \in \mathcal{G}_{\mu,\nu}} {\argsup}\,J_\rho(g).
        \quad\quad\quad\quad
        \text{c) }
        \supp_{-}\rho
        =        
        J_\rho(g^*_\rho)
        = \sup_{g\in\mathcal{G}_\rho}J_\rho(g).
    \end{equation}
    
\end{corollary}

\begin{proof}
    We have
    \begin{equation}\label{eq:D_Dhat}
        (0,g^*_\rho)
        \overset{(*)}{=}
        (0,\infty) \cap -\hat{D}_\mu
        \overset{\eqref{[lem:real_R_trans_extend]eq:D_mu}}{\subset}
        (0,\infty) \cap -D,
    \end{equation}
    where $(*)$ holds by \eqref{eq:G_rho_at_endpoint_abbrev} and \eqref{eq:D_hat_def}. Also by \eqref{eq:Grho_inv_ext_def}-\eqref{eq:Ginv_ext_restricted} and \eqref{eq:J_rho}
    \begin{equation}\label{eq:J_rho_Ginv_rho}
        J_\rho(g)
        \overset{\text{if }g\in-D}{=}
        G_\rho^{[-1]}(g)
        \overset{\text{if } g\in-\hat{D}_\mu}{=}
        G_\rho^{\inv}(g).
    \end{equation}

    {\bf a)} This follows from \eqref{eq:Ginv_increasing} and \eqref{eq:J_rho_Ginv_rho}.
    \,\newline

    {\bf b)} We have $g^*_\rho \le \sup \mathcal{G}_\rho$ by Lemma \ref{lem:Ginv_props} \ref{item:1[lem:Ginv_props]}, \eqref{eq:G_rho_at_endpoint_formula} and \eqref{eq:J_rho_Ginv_rho}. Also, it follows from \eqref{eq:D_Dhat}-\eqref{eq:J_rho_Ginv_rho} that 
    \begin{equation}\label{eq:J_rho'}
        J_\rho'(g^*_\rho) = G_\rho^{[-1]\prime}(g_\rho^{*}) = G_\rho^{\inv\prime}(g_\rho^{*}),
    \end{equation}
    (where one or both expressions on the r.h.s. may be well-defined only in the sense of \eqref{eq:f_well_defined_at_endpoints}; recall Lemma \ref{lem:Ginv_props} \ref{item:well_def_at_gstar[lem:Ginv_props]}). Either Lemma \ref{lem:Ginv_props} \ref{item:2[lem:Ginv_props]} holds, proving that $\sup \mathcal{G}_\rho \le g^*_\rho$ via \eqref{eq:J_rho'}, or Lemma \ref{lem:Ginv_props} \ref{item:5[lem:Ginv_props]} holds, implying that $\dom(J_\rho) = (0,g^*_\rho)$ and therefore $\sup \mathcal{G}_\rho \le g^*_\rho$. This proves that $g^*_\rho = \sup \mathcal{G}_\rho$.    
    Lastly $\sup\mathcal{G}_\rho = \argsup_{g \in \mathcal{G}_\rho}\,J_\rho(g)$ by definition (recall \eqref{eq:G_rho_at_endpoint_formula} and that $J_\rho$ is continuous).
    \,\newline
    
    {\bf c)} The first identity of c) follows by setting $g = g^*_\rho$ in \eqref{eq:J_rho_Ginv_rho}, and using \eqref{eq:Ginv_limit}. Just as in b), the second identity follows by definition.
\end{proof}

As a special case of the corollary we obtain the following characterization of the endpoint $z^*_{\mu\boxplus\nu} = \supp_{-} \mu \boxplus \nu$ in terms of the $R$-transforms $R_\mu$ and $R_\nu$ - or more precisely in terms of the function $J_{\mu,\nu}$ previously defined in \eqref{eq:J_mu_nu_def}. By \eqref{eq:D_sandwich} we can apply Corollary \ref{cor:Ginv_and_endpoint} with $\rho = \mu \boxplus \nu$ and $D = D_\mu \cap D_{\nu}$. As in \eqref{eq:g^*_rho_geometric}, the resulting geometric characterization is that $g^*_{\muplusnu}$ is the leftmost critical point of $J_{\mu,\nu}$, and if no critical point exists it is simply the right endpoint of $\dom(J_{\mu,\nu})$, in addition to $z^*_{\mu \boxplus \nu} = J_{\mu,\nu}(g^*_{\mu \boxplus \nu})$. In formulas, these statements are expressed as Proposition \ref{prop:add_conv_supp_end_point_Rmu_Rnu} b,c) below, cf. Corollary \ref{cor:Ginv_and_endpoint} b,c).

\begin{proposition}\label{prop:add_conv_supp_end_point_Rmu_Rnu}
    Let $\mu,\nu$ be compactly supported probability
    measures on $\mathbb{R}$. Let    \begin{equation}\label{eq:mathcal_G_mu_nu_def}
        \mathcal{G}_{\mu,\nu}
        :
        =
        \left\{
            g \in \dom(J_{\mu,\nu})
            :
            J_{\mu,\nu}'(r)>0
            \text{ for all }
            r\in(0,g)
        \right\}.
    \end{equation}
    Then the following holds.
    
    a) $J_{\mu,\nu}(g)$ is increasing for $g$ close enough to $0$.
    \begin{equation}
        \text{b) }g^*_{\mu \boxplus \nu}
        =
        \sup \mathcal{G}_{\mu,\nu}
        =
        \underset{g \in \mathcal{G}_{\mu,\nu}} {\argsup}\,J_\rho(g).
        \quad\quad
        \text{c) }
        z^*_{\mu\boxplus\nu}
        = 
        J_{\mu,\nu}( g^*_{\mu \boxplus \nu} )
        =        \sup_{g\in\mathcal{G}_{\mu,\nu}} J_{\mu,\nu}(g).
    \end{equation}

    d) The same holds true if in \eqref{eq:J_mu_nu_def} one replaces 
    $(R_\rho,D_\rho)$ by $(\hat{R}_\rho,\hat{D}_\rho)$, for $\rho = \mu$ and/or $\rho = \nu$.
\end{proposition}

\begin{proof}
    This is the special case $\rho=\mu \boxplus \nu$ and $D = A\cap B$ of Corollary \ref{cor:Ginv_and_endpoint}, for any combination of $A \in \{D_\mu,\hat{D}_\mu\}$ and  $B \in \{ D_{\nu},\hat{D}_{\nu} \}$ (the condition $\hat{D}_{\mu \boxplus \nu} \subset D \subset D_{\mu \boxplus \nu}$ indeed holds for each combination by \eqref{eq:D_sandwich}).

\end{proof}

We are now ready to prove the characterization Theorem \ref{thm:add_conv_supp_end_point} of the left endpoint of the support in terms of the function $F_{\mu,\nu}$ from \eqref{eq:F_func_def}. First note that
\begin{equation}\label{eq:RmuplusRnu}
    R_\mu(-g)
    +
    R_\nu(-g)
    -
    g^{-1}
    \overset{\eqref{eq:Grho_inv_ext_def}}{=}
    R_\mu(-g)
    +
    G_{\nu}^{[-1]}(g)
    \qquad
    \forall g \in -(D_\mu \cup D_\nu),
\end{equation}
giving
\begin{equation}\label{eq:Jmunu_and_R_mu_Ginv_nu}  
    J_{\mu,\nu}(g)
    \overset{\eqref{eq:J_mu_nu_def}}{=}
    R_\mu(-g)
    +
    G_{\nu}^{[-1]}(g)
    \qquad
    \forall g \in \dom(J_{\mu,\nu}).
\end{equation}
We will use this identity to show that $F_{\mu,\nu}(h) = J_{\mu,\nu}(G_{\nu}(h))$ (recall \eqref{eq:F_func_def}). Geometrically, it means that the graph of $F_{\mu,\nu}$ is obtained from that of $J_{\mu,\nu}$ by applying the map $G_\nu^{\inv}$ to the $x$-axis. Since $G_\nu^{\inv\prime}(g)>0$ for all $g$, the qualitative shape of the function stays the same, so that e.g. $g$ is a critical point of $J_\rho$ iff $G_\nu^{\inv}(g)$ is a critical point of $F_{\mu,\nu}$. With this in mind, it easy to see how the geometric picture painted by Theorem \ref{thm:add_conv_supp_end_point} follows from the geometric picture of Proposition \ref{prop:add_conv_supp_end_point_Rmu_Rnu} (recall the discussion before that proposition).

\begin{proof}[Proof of Theorem \ref{thm:add_conv_supp_end_point}]
     Assume that either $D=D_\mu$ or $D=\hat{D}_\mu$, and
    \begin{equation}\label{eq:which_dom_F}
        \dom(F_{\mu,\nu})
        =
        \left\{
            h \in (-\infty,\supp_{-}\nu):
            -G_{\nu}(h) \in D
        \right\}.
    \end{equation}

    {\bf \ref{thm:add_conv_supp_end_point-item:dom_F_interval}} Note that $D$ is an open interval containing $0$ (recall \eqref{eq:D_hat_def} and \eqref{[lem:real_R_trans_extend]eq:D_mu}). Also $G_\nu(h)$ is continuous, monotone and $G_\nu(h)\downarrow0$ as $h \downarrow -\infty$ (recall \eqref{eq:G_mu_strict_inc} and  \eqref{eq:G_mu_limit_z_to_inf}). Combining these facts implies that $\dom(F_{\mu,\nu})$ is an open interval that is infinite to the left, as claimed.
    \,\newline
        
    {\bf \ref{thm:add_conv_supp_end_point-item:F_inc}} By \eqref{eq:F_func_def} we have $F_{\mu,\nu}'(h) = -R_\mu'(-G_{\nu}(h))G_\mu'(h)+1$. Furthermore $G_\nu(h),G'_\nu(h)\downarrow0$ as $h\downarrow-\infty$, and also $R_\mu'(-G_{\nu}(h)) \to R_\mu'(0)$ in the same limit. Thus $F'_{\mu,\nu}(h) \to 1$ as $h \downarrow -\infty$, proving that $F_{\mu,\nu}$ is indeed increasing for small enough $h$.
    \,\newline

    {\bf \ref{thm:add_conv_supp_end_point-item:g_star}} We will use the version of Proposition \ref{prop:add_conv_supp_end_point_Rmu_Rnu} with
    \begin{equation}
        \dom(J_{\mu,\nu}) = (0,\infty) \cap -D \cap - \hat{D}_\nu,
    \end{equation}
    to match \eqref{eq:which_dom_F}.
    Then
    \begin{equation}\label{eq:Gnu_dom_F_to_dom_J}
        G_{\nu}\text{ maps }\dom(F_{\mu,\nu})\text{ to }\dom(J_{\mu,\nu})\text{ bijectively},
    \end{equation}
    and by \eqref{eq:F_func_def} and \eqref{eq:Jmunu_and_R_mu_Ginv_nu} 
    \begin{equation}\label{eq:F_J_G}
        F_{\mu,\nu}(h)
        =
        J_{\mu,\nu}(G_{\nu}(h))
        \quad
        \forall h \in \dom(F_{\mu,\nu}).
    \end{equation}
    Thus $F_{\mu,\nu}'(h) = J_{\mu,\nu}'(G_{\nu}(h))G_{\nu}'(h)$. Since $G_{\nu}'(h)>0$ for all $h < \supp_-\nu$ (recall \eqref{eq:mu_k-th_deriv_pos_inc}), this implies that
    \begin{equation}\label{eq:F_crit_G_crit}
        \sign( F_{\mu,\nu}'(h) )
        =
        \sign( J_{\mu,\nu}'( G_{\nu}(h) ) )
        \quad
        \forall h \in \dom(F_{\mu,\nu}).
    \end{equation}    
    Recalling also \eqref{eq:mathcal_H_mu_nu_def} and \eqref{eq:mathcal_G_mu_nu_def}, this implies that
    \begin{equation}
        G_\nu(\mathcal{H}_{\mu,\nu}) = \mathcal{G}_{\mu,\nu}.
    \end{equation}
    Taking the $\sup$ of both sides and using \eqref{eq:mathcal_H_mu_nu_def} for the l.h.s. and Proposition \ref{prop:add_conv_supp_end_point_Rmu_Rnu} b) for the r.h.s. implies the claim \eqref{eq:gstar[thm:add_conv_supp_end_point]}.
    \,\newline
    
    {\bf \ref{thm:add_conv_supp_end_point-item:z_star}} By Proposition \ref{prop:add_conv_supp_end_point_Rmu_Rnu} c) and \eqref{eq:gstar[thm:add_conv_supp_end_point]} and it follows that $z^*_{\mu \boxplus \nu} = J_{\mu,\nu}(G_\nu(h^*_{\mu,\nu}))$. The first identity of the claim \eqref{eq:z_star[thm:add_conv_supp_end_point]} then follows by \eqref{eq:F_J_G}. The second identity of \eqref{eq:z_star[thm:add_conv_supp_end_point]} is elementary (it follows by the definitions \eqref{eq:mathcal_H_mu_nu_def} and the continuity of $F_{\mu,\nu}$).
    \,\newline

    {\bf \ref{thm:add_conv_supp_end_point-item:G_nu(h_star)_in_closure}} The number $g^*_{\mu \boxplus \nu}$ lies in the closure of $-\hat{D}_{\mu \boxplus \nu}$ by the definition \eqref{eq:D_hat_def}. Thus by \eqref{eq:D_sandwich} it also lies in the closure of $-\hat{D}_\mu$. Therefore the claim follows by \eqref{eq:gstar[thm:add_conv_supp_end_point]}.
    \,\newline
    
    {\bf \ref{thm:add_conv_supp_end_point-item:D_hat}} This is already included in the above proof, since we started by assuming $D=D_\mu$ or $D=\hat{D}_\mu$.
\end{proof}

We finish by proving Proposition \ref{prop:Fmunu_props}.

\begin{proof}[Proof of Proposition \ref{prop:Fmunu_props}]
    {\bf \ref{item:leftmost_crit_point[prop:Fmunu_props]}, \ref{item:if_no_crit_point[prop:Fmunu_props]}} Since $\mathcal{H}_{\mu,\nu}$ is non-empty by Theorem \ref{thm:add_conv_supp_end_point} \ref{thm:add_conv_supp_end_point-item:F_inc}, these follows from the definitions in  \eqref{eq:mathcal_H_mu_nu_def}.
    \,\newline
    
    {\bf \ref{item:if_no_crit_point2[prop:Fmunu_props]}} Assume $F_{\mu,\nu}$ has no critical points. Then $h_{\mu,\nu}^* = \sup\dom(F_{\mu,\nu})$ by the already proved \ref{item:if_no_crit_point[prop:Fmunu_props]}. Assume first that $\sup\dom(F_{\mu,\nu})=\supp_{-}\nu$. Then $h_{\mu,\nu}^* = \supp_{-}\nu$, which in turn implies that $g^*_{\mu \boxplus \nu} = G_\nu(\supp_{-}\nu)$ by \eqref{eq:gstar[thm:add_conv_supp_end_point]}. This proves (i).
    Assume now the converse, that $\sup\dom(F_{\mu,\nu})<\supp_{-}\nu$. Then by the definition \eqref{eq:F_func_def} of $\dom(F_{\mu,\nu})$ and since $D_\mu$ is an interval, it must hold that $G_{\nu}(h_{\mu,\nu}^*) = \sup(-D_\mu)$. By Theorem \ref{thm:add_conv_supp_end_point} \ref{thm:add_conv_supp_end_point-item:G_nu(h_star)_in_closure} the quantity $G_{\nu}(h_{\mu,\nu}^*)$ also lies in the closure of $-\hat{D}_\mu \subset - D_\mu$, which implies that in fact $G_{\nu}(h_{\mu,\nu}^*) = \sup(-\hat{D}_\mu) = \sup(-D_\mu) = G_\mu(\supp_-\mu$). This proves the first identity of (ii). The second identity of (ii) is just a repetition of \eqref{eq:gstar[thm:add_conv_supp_end_point]}.
    \,\newline

    {\bf \ref{item:BES_conds[prop:Fmunu_props]}}
    By Lemma \ref{lem:Ginv_sc_edge} \ref{item:sc_edge_behaviour[lem:Ginv_sc_edge]} the point $g^*_{\mu \boxplus \nu}$ is a local maximum of $G^{[-1]}_{\mu \boxplus \nu}$. By \eqref{eq:Jmu_nu_and_Gplus_inv} it is therefore also a local maximum of $J_{\mu,\nu}$. Using the change of variables \eqref{eq:F_J_G} and recalling \eqref{eq:F_crit_G_crit} proves that $h^*_{\mu,\nu} = G^{\inv}_\nu(g^*_{\mu \boxplus \nu})$ is a local maximum of $F_{\mu,\nu}$.
    \,\newline
    
    {\bf \ref{item:F'_monotone[prop:Fmunu_props]}} Assume $F_{\mu,\nu}$ is strictly concave. Note that $F_{\mu,\nu}(h)$ is increasing for small enough $h$, and $\dom(F_{\mu,\nu})$ is bounded from the right. Strict concavity thus implies that $F_{\mu,\nu}$ has a unique global maximum. By the already proved \ref{item:leftmost_crit_point[prop:Fmunu_props]},\ref{item:if_no_crit_point[prop:Fmunu_props]}, the maximum must be $h^*_{\mu,\nu}$, so the right part of \eqref{eq:end_point_etc_special_cases_sc_MP} follows by \eqref{eq:gstar[thm:add_conv_supp_end_point]}. The left part of \eqref{eq:end_point_etc_special_cases_sc_MP} follows by \eqref{eq:z_star[thm:add_conv_supp_end_point]}.
    \,\newline
    
    {\bf \ref{item:F(h)=z_sol[prop:Fmunu_props]}} By the definition \eqref{eq:mathcal_H_mu_nu_def}, $F_{\mu,\nu}$ is strictly increasing on $(-\infty,h_{\mu,\nu}^*)$. Recall also from \eqref{eq:z_star[thm:add_conv_supp_end_point]} that $F_{\mu,\nu}(h^*_{\mu,\nu}) = z^*_{\mu \boxplus \nu}$. This implies that the equation $F_{\mu,\nu}(h) = z$ has a solution in $(-\infty,h_{\mu,\nu}^*)$ iff $z < z^*_{\mu \boxplus \nu}$.
    
    If $z < z^*_{\mu \boxplus \nu}$, then by definition the unique solution is $h:=h_{\mu,\nu}^<(z)$. Let $g := G_\nu(h_{\mu,\nu}^<(z))$. We have $J_{\mu,\nu}(g) = F_{\mu,\nu}(h)$ by \eqref{eq:F_J_G}. Also $J_{\mu,\nu}(g) = G_{\mu \boxplus \nu}^{[-1]}(g)$ by \eqref{eq:G_inv_ex_and_J_mu_nu}. Furthermore $g = G_\nu( h^<_{\mu,\nu}(z)) < G_\nu( h^*_{\mu,\nu} ) = g_{\mu \boxplus \nu}^*$, by \eqref{eq:gstar[thm:add_conv_supp_end_point]} and since $G_\nu$ is strictly increasing. Thus $G_{\mu \boxplus \nu}^{[-1]}(g) = G_{\mu \boxplus \nu}^{\inv}(g)$ by \eqref{eq:Ginv_ext_restricted}. Combining these identities, we deduce that $G_{\mu \boxplus \nu}^{\inv}(g) = F_{\mu,\nu}(h) = z$. This in turn implies that $G_{\mu \boxplus \nu}(z) = g$, completing the proof.
\end{proof}

\section{Variational formula for logarithmic potential (Proof of Theorem \ref{TH:main})}\label{sect:Proof}

We start by computing the derivative of the function $E_{\mu,\nu,z}$. 

\begin{lemma}\label{lem:E_mu_nu_z_deriv}
    Let $\mu,\nu$ be compactly supported
    probability measures on $\mathbb{R}$. For all $z<z^*_{\mu\boxplus\nu}$,
    the function $E_{\mu,\nu,z}(g)$ from \eqref{eq:E_def} is differentiable in the domain
    $\mathcal{E}_{\mu,\nu,z}$ from \eqref{eq:E_func_dom_no_hat}, and for all $g\in\mathcal{E}_{\mu,\nu,z}$
    \begin{equation}\label{eq:E_deriv}
        E_{\mu,\nu,z}'(g)
        =
        R_\mu'(-g)
        \left(
            g
            -
            \int\frac{1}{\lambda-z+R_\mu(-g)}\nu(d\lambda)        
        \right).
    \end{equation}
\end{lemma}

\begin{proof}
    Fix $\mu,\nu$ and $z<z^*_{\mu\boxplus\nu}$. Note that
    $z-R_\mu(-g)<\supp_{-}\nu$ for $g\in\mathcal{E}_{\mu,\nu,z}$ by the definition \eqref{eq:E_func_dom_no_hat}, so the integral in \eqref{eq:E_def} is well-defined and differentiable as a function of $g$. Differentiating the r.h.s. of \eqref{eq:E_def} we obtain that
    \begin{equation}
            E_{\mu,\nu,z}'(g)
            =
            g R_\mu'(-g)
            -
            R_\mu'(-g)\int\frac{1}{\lambda-z+R_\mu(-g)}\nu(d\lambda)
    \end{equation}
    for all $g\in\mathcal{E}_{\mu,\nu,z}$, which rearranged yields \eqref{eq:E_deriv}.
\end{proof}

The fixed point equation
\begin{equation}\label{eq:fixed_point}
    \int\frac{1}{\lambda-z+R_\mu(-g)}\nu(d\lambda)
    =
    g,
\end{equation}
(cf. \eqref{eq:E_deriv}) characterizes the Stieltjes transform $G_{\mu \boxplus \nu}$ of the additive convolution. In the simplest scenario \eqref{eq:fixed_point} has no solution for $z>z^*_{\mu\boxplus\nu}$, and for $z<z^*_{\mu\boxplus\nu}$ it has a unique solution equalling $g=G_{\mu \boxplus \nu}(z)$. However, there are $\mu,\nu$ s.t. \eqref{eq:fixed_point} has solutions also for $z>z^*_{\mu\boxplus\nu}$, and $\mu,\nu$ s.t. \eqref{eq:fixed_point} has several solutions for $z<z^*_{\mu\boxplus\nu}$ (see Figure \ref{fig:intro_E_mu_nu_z_F_mu_nu_shapes1} b,c,d). If $z<z^*_{\mu\boxplus\nu}$ then $g=G_{\mu \boxplus \nu}(z)$ is the leftmost solution, as proved by Lemma \ref{lem:fp_diff_sign} below.

Because of \eqref{eq:E_deriv}, these solutions are critical points of $E_{\mu,\nu,z}$. If $z<z^*_{\mu\boxplus\nu}$ then $g=G_{\mu \boxplus \nu}(z)$ is a local minimum, which is the global minimum of $E_{\mu,\nu,z}$ in $\mathcal{E}_{\mu,\nu,z} \cap (0,g^*_{\mu\boxplus\nu})$, as proved by Lemma \ref{lem:Gplus_is_min} below. This $g$ may or may not be the global minimum in the full domain $\mathcal{E}_{\mu,\nu,z}$ (in Figure \ref{fig:intro_E_mu_nu_z_F_mu_nu_shapes1} a) it is, but in b,c) it is not).

Note that by the definition \eqref{eq:defSti}
\begin{equation}\label{eq:fixed_point_eq_int_as_G_nu}
    \int\frac{1}{\lambda-z+R_\mu(-g)}\nu(d\lambda) = G_\nu(z-R_\mu(-g))
    \quad\text{ for all }\quad
    g \in \mathcal{E}_{\mu,\nu,z},
\end{equation}
so the fixed point equation \eqref{eq:fixed_point} is equivalent to
\begin{equation}\label{eq:fixed_point_Gnu}
    G_{\nu}(z-R_\mu(-g)) = g.
\end{equation}
We now give the aforementioned Lemma \ref{lem:fp_diff_sign}, which characterizes solutions of the fixed point equation \eqref{eq:fixed_point},\eqref{eq:fixed_point_Gnu}.

\begin{lemma}\label{lem:fp_diff_sign}
    Let $\mu,\nu$ be compactly supported probability
    measures on $\mathbb{R}$ and recall \eqref{eq:G_rho_at_endpoint_abbrev}.
     \begin{enumerate}[label=\alph*)]
     
        \item\label{<lem:G_sign>_a)}
        For any $z\in\mathbb{R}$ and $g \in \mathcal{E}_{\mu,\nu,z}$
        \begin{equation}\label{eq:sign_equality}
            \sign\left( 
                g
                -
                G_{\nu}(z-R_\mu(-g))
            \right)
            =
            \begin{cases}
                -1 & \text{ if } g\le 0,\\
                1 & \text{ if }g \ge g^*_\nu,\\
                \sign\left( F_{\mu,\nu}(G_\nu^{\inv}(g)) - z \right) &  \text{ if }g \in (0,g^*_\nu),\\
                \sign\left( G_{\muplusnu}^{[-1]}(g) - z \right) & \text{ if }g \in (0,g^*_\nu),\\
                \sign\left( g- G_{\muplusnu}(z)  \right) & \text{ if }
                z < z^*_{\mu\boxplus\nu},
                g \in (0,g^*_{\mu\boxplus\nu}).
            \end{cases}
        \end{equation}

        \item\label{item:D_+[lem:fp_diff_sign]} If $z \in \mathbb{R}$ and $g \in \mathcal{E}_{\mu,\nu,z} \cap -D_\mu^+$, then
        \begin{equation}
            \sign(E_{\mu,\nu,z}'(g)) = \eqref{eq:sign_equality}.
        \end{equation}
        
        \item\label{<lem:G_sign>_b)}
        If $z<z^*_{\mu\boxplus\nu}$ then
        \begin{equation}\label{eq:GboxplusinE}
            G_{\mu\boxplus\nu}(z) \in \hat{\mathcal{E}}_{\mu,\nu,z} \cap (0,\infty).
        \end{equation}
        
        \item\label{<lem:G_sign>_c)}
        If $z<z^*_{\mu\boxplus\nu}$ then
        \begin{equation}\label{eq:G+(z)_solved_fixed_point_eq}
            g=G_{\mu\boxplus\nu}(z)
            \quad\text{satisfies}\quad
            \eqref{eq:fixed_point},\eqref{eq:fixed_point_Gnu}.
        \end{equation}
        
    \end{enumerate}
\end{lemma}

\begin{proof}
    
    {\bf \ref{<lem:G_sign>_a)}} For any $z\in\mathbb{R}$ and $g\in\mathcal{E}_{\mu,\nu,z}$ we have $z-R_\mu(-g)<\supp_{-}\nu$ by the definition \eqref{eq:E_func_dom_no_hat} of $\mathcal{E}_{\mu,\nu,z}$. Thus $G_\nu(z-R_\mu(-g))>0$, which proves the the top line of \eqref{eq:sign_equality}. Furthermore, since $G_{\nu}$ is strictly monotone on $(-\infty,\supp_{-}\nu)$ it holds that $G_{\nu}(z-R_\mu(-g)) < G_\nu(\supp_{-}\nu) = g^*_\nu$, so the second line of \eqref{eq:sign_equality} follows.
    
    If on the other hand $g \in (0,g^*_\nu)$, then applying $G_\nu^{\inv}$ ``to both sides'' yields
    \begin{equation}\label{eq:sign_intermediate}
        \sign\left( 
                g
                -
                G_{\nu}(z - R_\mu(-g))
            \right) 
        \overset{\eqref{eq:Ginv_increasing}}{=}
        \sign( G_\nu^{\inv}(g) - z + R_\mu(-g)) ).
    \end{equation}    
    Note that we have assumed $g\in\mathcal{E}_{\mu,\nu,z}$, so that by definition $g\in-D_\mu$. If also $g \in (0,g^*_\nu)$ then $h = G_{\nu}^{\inv}(g)$ is well-defined, and $h < \supp_- \nu$. Combining these we obtain that $h \in \dom(F_{\mu,\nu})$ (recall \eqref{eq:F_func_def}). Furthermore
    \begin{equation}
        F_{\mu,\nu}( G_{\nu}^{\inv}(g) )
        \overset{\eqref{eq:F_func_def}}{=}
        R_\mu(-g) + G_{\nu}^{\inv}(g).
    \end{equation}
    From this and \eqref{eq:sign_intermediate} the third line of \eqref{eq:sign_equality} follows.

    Turning to the fourth line of \eqref{eq:sign_equality}, recall our assumption that $g \in (0,g^*_\nu)$ and consider again the r.h.s. of \eqref{eq:sign_intermediate}. Note that $-D_\mu\cap-\hat{D}_\nu \subset -D_{\mu\boxplus\nu}$ (recall \eqref{eq:D_sandwich}). Also, combining \eqref{eq:G_inv_ex_and_J_mu_nu} and \eqref{eq:RmuplusRnu} and recalling \eqref{eq:Ginv_ext_restricted} 
    we obtain
    \begin{equation}\label{eq:I}
        R_\mu(-g) + G_\nu^{\inv}(g)
        =
        G_{\mu\boxplus\nu}^{[-1]}(g)
        \qquad
        \forall g \in -D_\mu \cap -\hat{D}_{\nu} \setminus \{0\}.
    \end{equation}
    Since $\mathcal{E}_{\mu,\nu,z} \subset -D_\mu$ by \eqref{eq:E_func_dom_no_hat}, this completes the proof the fourth line of \eqref{eq:sign_equality}.
    
    Lastly, assuming that $z<z^*_{\mu\boxplus\nu}$ and $g \in (0, g^*_{\mu\boxplus\nu})$ we obtain
    \begin{equation}\label{eq:II}
        \sign( G_{\mu\boxplus\nu}^{[-1]}(g) - z )
        \overset{g \in (0,g^*_{\mu\boxplus\nu})}{=}
        \sign( G^{\inv}_{\mu \boxplus \nu}(g) - z )
        \overset{\eqref{eq:G_mu_strict_inc}}{=}
        \sign( g - G_{\muplusnu}(z) ),
    \end{equation}    
    where in the last step we applied $G_{\mu \boxplus \nu}$ ``to both sides''. This proves the last line of \eqref{eq:sign_equality}.\\

    {\bf \ref{item:D_+[lem:fp_diff_sign]}} This is a consequence of \eqref{eq:D_mu_plus_def}, \eqref{eq:E_deriv} and \eqref{eq:fixed_point_eq_int_as_G_nu}.\\
    
    {\bf \ref{<lem:G_sign>_b)}}
    Assume that $z<z^*_{\mu\boxplus\nu}$ and let $g(z) := G_{\mu\boxplus\nu}(z)$. Then    
    \begin{equation}
        g(z)
        =
        G_{\mu\boxplus\nu}(z)
        \overset{\eqref{eq:D_hat_def}} \in
        - \hat{D}_{\mu \boxplus \nu}
        \overset{ \eqref{eq:D_sandwich} } \subset 
        -(\hat{D}_\mu \cap \hat{D}_\nu).
    \end{equation}
    Thus $g(z) \in - \hat{D}_\mu$, so recalling \eqref{eq:E_func_dom_Ehat}, we observe that \eqref{eq:GboxplusinE} follows once we verify that
    \begin{equation}\label{eq:suff_to_show}
        z - R_\mu(-g(z)) < \supp_- \nu.
    \end{equation}
    To this end, note that $G_{\mu \boxplus \nu}^{[-1]}(g(z))= G_{\mu \boxplus \nu}^{\inv}(g(z))  = z$, so that by \eqref{eq:I}
    \begin{equation}
        R_\mu(-g(z)) + G_\nu^{\inv}(g(z))
        =
        G_{\muplusnu}^{\inv}(g(z))
        =
        z.
    \end{equation}
    Since $G_\nu^{\inv}(g(z))<\supp_{-}\nu$ (recall \eqref{eq:Ginv_maps_bijecitvely}), this implies \eqref{eq:suff_to_show}, and therefore completes the proof.\\
    
    {\bf \ref{<lem:G_sign>_c)}} Recall that $\hat{\mathcal{E}}_{\mu,\nu,z} \subset \mathcal{E}_{\mu,\nu,z}$. Thus the claim is an immediate consequence of the last line of \ref{<lem:G_sign>_a)} and of \ref{<lem:G_sign>_b)}.
\end{proof}

\begin{remark}
    The characterization Proposition \ref{prop:Emunu_props} \ref{prop:Emunu_props_E_func_crit_point} of all critical points of $E_{\mu,\nu,z}$, including spurious such points (see Remark \ref{rem:intro_spurios_crit_point_charaterization}), is essentially a consequence of the third line of \eqref{eq:sign_equality} (the formal proof is at the end of the section).
    
    From the fourth line of \eqref{eq:sign_equality} one obtains the alternative characterization that $g$ 
    is a critical point of $E_{\mu,\nu,z}$ iff $R_\mu'(-g)=0$ or $G_{\mu \boxplus \nu}^{[-1]}(g)=z$. Spurious critical points of $E_{\mu,\nu,z}$ thus arise either because $R_\mu'$ has a zero (which is then necessarily outside $\hat{D}_\mu$ by Lemma \ref{lemma:Rmu_increasing}), or because the analytic extension $G_{\mu \boxplus \nu}^{[-1]}$ of the inverse $G_{\mu \boxplus \nu}^{\inv}$ takes the value $z$ for some $g > g^*_{\mu \boxplus \nu}$. See Figure \ref{fig:intro_E_mu_nu_z_F_mu_nu_shapes1} b,c,d).
\end{remark}

We now prove that if $z<z^*_{\mu\boxplus\nu}$, then the minimizer in the variational formula in \eqref{eq:main} is $g=G_{\mu\boxplus\nu}(z)$.

\begin{lemma}\label{lem:Gplus_is_min}
    Let $\mu,\nu$ be compactly supported probability measures on $\mathbb{R}$.
    \begin{enumerate}[label=\alph*)]
        \item \label{item:z>=supp-_no_crit_pt[lem:G+_is_min]} If $z \ge z^*_{\mu\boxplus\nu}$, then $E_{\mu,\nu,z}$ is strictly decreasing in $\mathcal{E}_{\mu,\nu,z}\cap(0,g^*_{\mu \boxplus \nu})$.
        
        \item \label{item:z<supp-_one_crit_pt[lem:G+_is_min]} If $z<z^*_{\mu\boxplus\nu}$, then $E_{\mu,\nu,z}$ has exactly one critical point in $\mathcal{E}_{\mu,\nu,z}\cap(0,g^*_{\mu \boxplus \nu})$, which is a non-degenerate local minimum, and the global minimum inside that set. Furthermore, the minimum equals
        \begin{equation}
            g = G_{\muplusnu}(z) \in \hat{\mathcal{E}}_{\mu,\nu,z} \cap (0,g^*_{\mu \boxplus \nu}).
        \end{equation}
        In particular,
        \begin{equation}\label{[lem:G+_is_min]eq:inf}
            \inf_{g \in \hat{\mathcal{E}}_{\mu,\nu,z} \cap (0,g^*_{\mu\boxplus\nu})}
            E_{\mu,\nu,z}(g)
            =
            \inf_{
                g
                \in
                \mathcal{E}_{\mu,\nu,z}            
                \cap(0,g^*_{\mu \boxplus \nu})
            }
            E_{\mu,\nu,z}(g)
            =
            E_{\mu,\nu,z}(G_{\muplusnu}(z)).
        \end{equation}
    \end{enumerate}
\end{lemma}

\begin{proof}
    {\bf \ref{item:z>=supp-_no_crit_pt[lem:G+_is_min]}} Note that $(0,g^*_{\mu \boxplus \nu}) = (0,\infty)\cap -\hat{D}_{\mu \boxplus \nu}$ (recall \eqref{eq:D_hat_def}). Thus by
    \eqref{eq:D_sandwich} and \eqref{eq:Dhatmu_subset_D+mu} we have $(0,g^*_{\mu \boxplus \nu}) \subset -D^+_\mu$. Therefore by Lemma \ref{lem:fp_diff_sign} \ref{<lem:G_sign>_a)} \ref{item:D_+[lem:fp_diff_sign]} and \eqref{eq:Ginv_ext_restricted}
    \begin{equation}
        \sign\left(E_{\mu,\nu,z}'(g)\right)
        =
        \sign\left( G_{\muplusnu}^{\inv}(g) - z \right)
    \end{equation}
    for all $g\in\mathcal{E}_{\mu,\nu,z}\cap(0,g^*_{\mu \boxplus \nu})$. We have $G_{\muplusnu}^{\inv}(g) < z^*_{\mu \boxplus \nu}$ for all $g < g^*_{\mu \boxplus \nu}$, so if $z \ge z^*_{\mu \boxplus \nu}$, then clearly the r.h.s. is negative. This proves \ref{item:z>=supp-_no_crit_pt[lem:G+_is_min]}.\\

    {\bf \ref{item:z<supp-_one_crit_pt[lem:G+_is_min]}}
    If $ z < z^*_{\mu \boxplus \nu}$, then using Lemma \ref{lem:fp_diff_sign}  \ref{<lem:G_sign>_a)} \ref{item:D_+[lem:fp_diff_sign]} again yields 
    \begin{equation}
        \sign\left(E_{\mu,\nu,z}'(g)\right)=\sign(G_{\muplusnu}(z)-g).
    \end{equation}
    This and \eqref{eq:G_mu_strict_inc} imply that for all $g \in \mathcal{E}_{\mu,\nu,z}\cap(0,g^*_{\mu \boxplus \nu})$
    \begin{equation}
        E_{\mu,\nu,z}'(g)\begin{cases}
        <0 & \text{\,for }0<g<G_{\muplusnu}(z),\\
        =0 & \text{ for }g=G_{\muplusnu}(z),\\
        >0 & \text{\,for }G_{\muplusnu}(z)<g.
        \end{cases}
    \end{equation}
    Since $G_{\muplusnu}(z) \in \hat{\mathcal{E}}_{\mu,\nu,z} \cap (0,g^*_{\mu \boxplus \nu})$ by Lemma \ref{lem:fp_diff_sign} \ref{<lem:G_sign>_c)}, this concludes the proof of \ref{item:z<supp-_one_crit_pt[lem:G+_is_min]}.
\end{proof}

The remainder of the proof is concerned with proving that $U_{\mu \boxplus \nu}(z)=E_{\mu,\nu,z}(G_{\muplusnu}(z))$. For this we first note that $U_{\mu \boxplus \nu}$ satisfies
\begin{equation}\label{eq:U+_deriv_and_asymp}
    U_{\muplusnu}'(z)=-G_{\muplusnu}(z)\quad\text{for}\quad z<z^*_{\mu\boxplus\nu},\quad\text{and}\quad\lim_{z\to-\infty}\{U_{\muplusnu}(z)-\log|z|\}=0,
\end{equation}
as a consequence of Lemma \ref{lem:U_rho_deriv_and_asymp} with $\rho = \mu\boxplus\nu$. The next lemma shows that $z\to E_{\mu,\nu,z}(G_{\mu\boxplus\nu}(z))$ has the same behaviour, and the corollary thereafter deduces that $E_{\mu,\nu,z}(G_{\mu\boxplus\nu}(z))=U_{\mu \boxplus \nu}(z)$.

\begin{lemma}\label{lem:G+_z_G+_z_deriv_and_asymp}
    For all $z<z^*_{\mu\boxplus\nu}$
    \begin{equation}\label{eq:E_of_z_G+_z_deriv}
        a)\,\,\frac{d}{dz}\left\{E_{\mu,\nu,z}(G_{\muplusnu}(z))\right\}=-G_{\muplusnu}(z),
        \quad\quad
        b)\,\,\lim_{z\downarrow-\infty}\{E_{\mu,\nu,z}(G_{\muplusnu}(z))-\log|z|\}=0.
    \end{equation}
\end{lemma}

\begin{proof}
    {\bf a)} By the chain rule
    \begin{equation}\label{eq:E_of_z_G+_z_deriv_chain_rule}
        \frac{d}{dz}\left\{E_{\mu,\nu,z}(G_{\muplusnu}(z))\right\}
        =
        \frac{d}{dz}E_{\mu,\nu,z}(g)|_{g=G_{\muplusnu}(z)}        
        +
        G_{\muplusnu}'(z)E_{\mu,\nu,z}'(g)|_{g=G_{\muplusnu}(z)}.
    \end{equation}
    We have already verified in Lemma \ref{lem:Gplus_is_min} that $g=G_{\mu\boxplus\nu}(z)$ is a local minimum of $E_{\mu,\nu,z}$, so
    \begin{equation}\label{eq:E_of_z_G+_z_deriv_in_g}
        E_{\mu,\nu,z}'(g)|_{ g=G_{\muplusnu}(z) } = 0.
    \end{equation}
    Recalling \eqref{eq:defSti} and \eqref{eq:E_def}, note that
    \begin{equation}
        \frac{d}{dz}E_{\mu,\nu,z}(g)
        =
        -G_{\nu}(z-R_\mu(-g)).
    \end{equation}
    Thus by Lemma \ref{lem:fp_diff_sign} \ref{<lem:G_sign>_c)}
    \begin{equation}\label{eq:E_of_z_G+_z_deriv_in_z}
        \frac{d}{dz}E_{\mu,\nu,z}(g)|_{g=G_{\muplusnu}(z)}
        =
        -G_{\muplusnu}(z).
    \end{equation}
    Combining (\ref{eq:E_of_z_G+_z_deriv_chain_rule})-(\ref{eq:E_of_z_G+_z_deriv_in_z})
    proves a).
    
    {\bf b)} Recall again the definition
    \eqref{eq:E_def} of $E_{\mu,\nu,z}$, and note that the parts involving $R_\mu$ satisfy
    \begin{equation}\label{eq:E_func_R_terms}
        \lim_{g \downarrow 0} \int_{0}^{g} s R_\mu'(-s)ds = 0
        \quad\text{and}\quad
        \lim_{g \downarrow 0}R_\mu(-g) = R_\mu(0),
    \end{equation}
    since $R_\mu$ is analytic in an open neighbourhood of $0$. Also $G_{\muplusnu}(z)\downarrow0$ as $z\downarrow-\infty$ (recall \eqref{eq:G_mu_limit_z_to_inf}), so using \eqref{eq:E_func_R_terms} with $g=G_{\muplusnu}(z)$ in \eqref{eq:E_def} proves b).
\end{proof}

\begin{corollary}\label{cor:U_+(z)_and_E_mu_nu(z,G_+(z))_ident}
    For all $z<z^*_{\mu\boxplus\nu}$
    \begin{equation}\label{<lem:U_+(z)_and_E_mu_nu(z,G_+(z))_ident>eq:ident}
        U_{\muplusnu}(z)=E_{\mu,\nu,z}(G_{\muplusnu}(z)).
    \end{equation}
\end{corollary}

\begin{proof}
    By (\ref{eq:U+_deriv_and_asymp}) and Lemma \ref{lem:G+_z_G+_z_deriv_and_asymp},
    the left and right-hand sides of (\ref{<lem:U_+(z)_and_E_mu_nu(z,G_+(z))_ident>eq:ident})
    have identical derivative, and their difference converges to zero
    as $z\downarrow-\infty$. This implies that the two sides are equal, proving
    (\ref{<lem:U_+(z)_and_E_mu_nu(z,G_+(z))_ident>eq:ident}).
\end{proof}

With the above, it is now easy to deduce our main theorem.

\begin{proof}[Proof of Theorem \ref{TH:main}]
    The identity \eqref{eq:main} is a direct consequence of 
    \eqref{[lem:G+_is_min]eq:inf} and \eqref{<lem:U_+(z)_and_E_mu_nu(z,G_+(z))_ident>eq:ident}. The claim about the minimizer follows from Lemma \ref{lem:Gplus_is_min} b).
\end{proof}

In the remainder of the section, we prove Proposition \ref{prop:Emunu_props} and Proposition \ref{prop:Emunu_inf_weaker_cond}, which encapsulate some additional properties of $E_{\mu,\nu,z}$ and the variational formula \eqref{eq:main}.

\begin{proof}[Proof of Proposition \ref{prop:Emunu_props}]
    \,\newline
    {\bf \ref{prop:Emunu_props_E_func_crit_point}} By Lemma \ref{lem:E_mu_nu_z_deriv}, $g \in \mathcal{E}_{\mu,\nu,z}$ is a critical point of $E_{\mu,\nu,z}$ iff $R_\mu'(-g)=0$ or the fixed point equation \eqref{eq:fixed_point_Gnu} is satisfied. It follows from the second and third lines of Lemma \ref{lem:fp_diff_sign} a) that \eqref{eq:fixed_point_Gnu} is satisfied iff $g\in(0,g^*_\nu)$ and $F_{\mu,\nu}(h) = z$ for $h = G_{\nu}^{\inv}(g)$. Since $h = G_{\nu}^{\inv}(g)$ is equivalent to $g = G_{\nu}(h)$, this completes the proof of \ref{prop:Emunu_props_E_func_crit_point}.

    \,\newline
    {\bf \ref{prop:Emunu_props_loc_min_of_E_func}} This is proved by Lemma \ref{lem:Gplus_is_min}.
\end{proof}

\begin{proof}[Proof of Proposition \ref{prop:Emunu_inf_weaker_cond}]
     We claim that for $z < z^*_{\mu \boxplus \nu}$
    \begin{equation}\label{eq:F_first_and_sec_sol}
        F_{\mu,\nu}(h)
        \begin{cases}
            <z & \text{\,if }h\in\left(-\infty,h_{\mu,\nu}^<(z)\right),\\
            >z & \text{\,if }h\in\left(h_{\mu,\nu}^<(z),h^>_{\mu,\nu}(z)\right).
        \end{cases}
    \end{equation}   
    The top line of \eqref{eq:F_first_and_sec_sol} follows since $F_{\mu,\nu}(h)$ is strictly increasing for $h < h^*_{\mu,\nu}$, and $h^<_{\mu,\nu}(z) < h^*_{\mu,\nu}$ is the leftmost solution of $F_{\mu,\nu}(h) = z$. The same facts imply that $F_{\mu,\nu}(h) > z$ for $ h \in (h^<_{\mu,\nu}(z),h^<_{\mu,\nu}(z)+\varepsilon)$ for small enough $\varepsilon>0$. Since $h^>_{\mu,\nu}(z)$ is the second solution of $F_{\mu,\nu}(z)=h$, if it exists, the bottom line of \eqref{eq:F_first_and_sec_sol} follows.

    Recall the third line of Lemma \ref{lem:fp_diff_sign} \ref{<lem:G_sign>_a)}. Recall also that $G_\nu(h^<_{\mu,\nu}(z)) = G_{\mu \boxplus \nu}(z)$, and let $g^>(z): = G_\nu(h^>_{\mu,\nu}(z))$. Using Lemma \ref{lem:fp_diff_sign} \ref{<lem:G_sign>_a)} \ref{item:D_+[lem:fp_diff_sign]} we obtain from \eqref{eq:F_first_and_sec_sol} that
    \begin{equation}\label{eq:E_deriv_sign}
        E_{\mu,\nu,z}'(g)
        \begin{cases}
            <0 & \text{\,for }g \in \left(-\infty,G_{\mu \boxplus \nu}(z)\right) \cap -D^+_\mu,\\
            >0 & \text{\,for }g\in\left(G_{\mu \boxplus \nu}(z),g^>(z)\right) \cap -D^+_\mu.
        \end{cases}
    \end{equation}\\

    {\bf \ref{item:F(h)=z_sec_sol[prop:Emunu_inf_weaker_cond]}} We assume that $z < z^*_{\mu \boxplus \nu}$, that $h^>_{\mu,\nu}(z)$ is a solution of $F_{\mu,\nu}(h) = z$, and that $(0,g^>(z)) \subset -D^+_\mu$. Then $g^>(z)$ is a critical point of $E_{\mu,\nu,z}$ by Proposition \ref{prop:Emunu_props} \ref{prop:Emunu_props_E_func_crit_point}. It follows from \eqref{eq:E_deriv_sign} that $E_{\mu,\nu,z}(g)$ is strictly increasing between $G_{\mu \boxplus \nu}(z)$ and $g^>(z)$. This proves that $g^>(z)$ is a local maximum or inflection point.
    \,\newline

    {\bf \ref{item:var_form_larger_range[prop:Emunu_inf_weaker_cond]}} Assume that $u\le0$ and $v\ge g^*_{\mu \boxplus \nu}$ are s.t. $(u,v) \subset - D_\mu^+$. It follows from \eqref{eq:E_deriv_sign} that $E_{\mu,\nu,z}$ is strictly
    decreasing in $(u,0]$ and strictly increasing in $[g_{\mu\boxplus\nu}^*,v)$. Thus even $(0,g^*_{\mu \boxplus \nu})$ in \eqref{eq:main} is replaced by $(u,v)$, the infima cannot be achieved outside $(0,g^*_{\mu \boxplus \nu})$. This proves the claim.
    \,\newline

    {\bf \ref{item:F_mon_at_most_loc_min_and_loc_max[prop:Emunu_inf_weaker_cond]}} Assume that $F_{\mu,\nu}$ is strictly concave, so that $ h^*_{\mu,\nu}$ is the unique global maximum of $F_{\mu,\nu}$.
 By Lemma \ref{lem:fp_diff_sign} \ref{item:D_+[lem:fp_diff_sign]} and the assumption $\mathcal{E}_{\mu,\nu,z} \subset -D^+_\mu$ we have
    \begin{equation}\label{eq:E_mu_nu_z_sign_F_mu_nu_eq}
        \sign( E_{\mu,\nu,z}'(g) )
        =
        \sign\left( F_{\mu,\nu}(G_\nu^{\inv}(g)) - z \right)
        \quad\text{ for all }\quad
        g \in \mathcal{E}_{\mu,\nu,z}.
    \end{equation}
    If $z > F(h^*_{\mu,\nu}) = z^*_{\mu \boxplus \nu}$, then $F_{\mu,\nu}(h) < z$ for all $h$. By \eqref{eq:E_mu_nu_z_sign_F_mu_nu_eq} this implies that $E_{\mu,\nu,z}'(g)>0$ for all $g \in \mathcal{E}_{\mu,\nu,z}$, which proves the top line of \eqref{eq:special_case_sc_MP_E_crit_points}.

    If $z = z^*_{\mu \boxplus \nu}$, then the equation $F_{\mu,\nu}(h) = z$ has the unique solution $h = h^*_{\mu,\nu} = G_\nu^{\inv}(g^*_{\mu \boxplus \nu})$, and $F_{\mu,\nu}(h) < z$ for all other $h$. By \eqref{eq:E_mu_nu_z_sign_F_mu_nu_eq} this implies that $g = g^*_{\mu \boxplus \nu}$ is an inflection point of $E_{\mu,\nu,z}$, and the unique critical point. Thus we have proved the middle line of \eqref{eq:special_case_sc_MP_E_crit_points}.

    If $z < z^*_{\mu \boxplus \nu}$, then the equation $F_{\mu,\nu}(h) = z$ has at most two solutions, and $F_{\mu,\nu}(h)<z$ for $h>h^>_{\mu,\nu}(z)$. Together with \eqref{eq:F_first_and_sec_sol} and \eqref{eq:E_mu_nu_z_sign_F_mu_nu_eq} this implies the bottom line of \eqref{eq:special_case_sc_MP_E_crit_points}.
\end{proof}

\section{Connecting real $R$-transform to standard theory}\label{sect:R_transform_theory}

In Section \ref{sect:R-trans} we introduced what we call the {\it real $R$-transform}. Lemma \ref{lem:real_R_trans_extend} essentially states that it is a well-defined object, and Lemma \ref{lem:real_R_trans_add_cond} connects it to the free additive convolution.
The standard definition of the $R$-transform in the literature is as a formal complex power series. In this section we will relate this definition to the real $R$-transform, and derive Lemmas \ref{lem:real_R_trans_extend}-\ref{lem:real_R_trans_add_cond} from the theory of $R$-transforms presented in \cite{speicher}. 

\subsection{The $R$-transform as a power series}
The $R$-transform $R_\mu$ of a probability measure $\mu$ on $\mathbb{R}$ is a formal complex power series, which is derived from the formal complex power series of $G_\mu$ via the relation \eqref{eq:R-transf} (see \cite{voiculescu} or  \cite[Chapter 2]{speicher}). We denote the power series by $\overline{R}_\mu$.

\subsection{Free additive convolution in terms of power series}
The {\it free additive convolution} of two probability measures $\mu$, $\nu$ is the unique probability measure whose $R$-transform is the formal power series $\overline{R}_\mu + \overline{R}_\nu$ obtained by adding the $R$-transforms of $\mu$ and $\nu$ (see \cite{voiculescu}, \cite[Theorem 18, Chapter 2]{speicher}). The free additive convolution $\mu\boxplus\nu$ has compact support if $\mu$ and $\nu$ do, and 
\begin{equation}\label{eq:add_conv_supp_ineq}
    \max\(|\supp\mu|,|\supp\nu|\)\leq|\supp(\mu\boxplus\nu)|\leq |\supp\mu|+|\supp\nu|\,
\end{equation}
(see \cite[page 326 and Lemma 3.1]{voiculescu}), where $|\supp\mu|:=\supp_+\mu-\supp_-\mu$.

\subsection{$R$-transforms as complex analytic functions}\label{subsec:R_trans_complex_analytic}
\cite[Chapter 3]{speicher} provides an alternative construction of the $R$-transform by proving that \eqref{eq:R-transf} holds as an identity of complex analytic functions, after restricting $G_\mu$ to an appropriate complex domain. This approach is more convenient for our purposes.

Let $\mu$ be a probability measure on $\mathbb{R}$. For any domain\footnote{Here and elsewhere a complex domain means an open and connected subset of $\mathbb{C}$.} $D\subset\mathbb{C}$, we call $\tilde{R}:D\to\mathbb{C}$
an \emph{$R$-transform function} of $\mu$ if $\tilde{R}$ is complex analytic, $-\tilde{R}(u)-u^{-1}\notin\supp\mu$
for all $u\in D\setminus\{0\}$, and 
\begin{equation}\label{eq:R_trans_func_eq}
    G_\mu( -\tilde{R}(u) - u^{-1} ) = u
    \quad\text{for all}\quad
    u\in D\setminus\{0\}.
\end{equation}
By \cite[Theorem 17 (iii), Chapter 3]{speicher}, there exists
an $R$-transform function whose domain $D$ includes $u=0$ for every $\mu$ with compact support. Furthermore,
by \cite[Theorem 17 (i), Chapter 3]{speicher} there is an $s>0$ such that $G_\mu|_{B(0,s)^c}$
is invertible, where $B(x,r)$ denotes the open ball of radius $r\ge0$ in $\mathbb{C}$, centered at $x\in\mathbb{C}$. 
Denote the inverse by $(G_\mu|_{B(0,s)^c})^{\left\langle -1\right\rangle }$.
For any $R$-transform function $\tilde{R}:D\to\mathbb{C}$ with $0\in D$,
it holds that $\tilde{R}(u)+u^{-1}\in B(0,s)^c$ for $u$ close enough to $0$. Thus there is a neighbourhood $U$ of $0$ on which
\begin{equation}\label{eq:R_trans_Ginv_complex_nbhood}
    \tilde{R}(u)=(G_\mu|_{B(0,s)^c})^{\left\langle -1\right\rangle }(-u)-u^{-1}\quad\text{for}\quad u\in U.
\end{equation}
If $\tilde{R}_i:D_i\to\mathbb{C},i=1,2$ are
two $R$-transform functions of the same $\mu$ and $0\in D_1\cap D_2$, this implies that they agree on a small enough neighbourhood of $0$. Thus they
have the same power series around $u=0$. This uniquely defined
power series is the aforementioned $\overline{R}_\mu$ (see \cite[Theorem 17 (iv), Chapter 3]{speicher} and above (3.1) in the same chapter).

In general there is no uniquely defined complex analytic extension of all the $R$-transform
functions of a measure $\mu$ to a maximal domain in $\mathbb{C}$, so strictly speaking we cannot speak of {\it{the}} $R$-transform function of $\mu$.

\subsection{$R$-transform function on real axis}

Recall from \eqref{eq:def_real_R_trans_hat} the real analytic function
\begin{equation}\label{eq:def_real_R_trans_hat_sec4}
    \hat{R}_\mu(u)=(G_\mu|_{\mathbb{R}\cap(\overline{\supp}\,\mu)^c})^{\left\langle -1\right\rangle }(-u)-u^{-1},
    \quad\quad\quad
    \hat{R}_\mu:\hat{D}_\mu \setminus\{0\} \to \mathbb{R},
\end{equation}
which is defined for any compactly supported $\mu$. Comparing \eqref{eq:R_trans_Ginv_complex_nbhood} and \eqref{eq:def_real_R_trans_hat_sec4}, it is clear that $\hat{R}_\mu$ is related to the $R$-transform functions $\tilde{R}$ of $\mu$.
We now precisely characterize the relationship.
\begin{equation}\label{eq:U_R_def}
    \begin{array}{c}
        \text{For any $R$-transform function $\tilde{R}:D\to\mathbb{C}$ of $\mu$, where $D \ni 0$ is a complex domain,}\\
        \text{define the real interval $U_{\tilde{R}}$ as the connected component of $0$ in $D \cap \mathbb{R}$},
    \end{array}
\end{equation}
(which is non-empty, see the discussion below \eqref{eq:R_trans_func_eq}).
On the real axis there {\it does} exist a unique maximal real analytic extension of $\hat{R}_\mu$ and of any $\tilde{R}|_{U_{\tilde{R}}}$, as we now prove in Lemma \ref{lem:real_R_trans_extend_sec_4} below. Note that Lemma \ref{lem:real_R_trans_extend_sec_4} implies Lemma \ref{lem:real_R_trans_extend}.

\begin{lemma}\label{lem:real_R_trans_extend_sec_4}
    For any compactly supported probability measure $\mu$ on $\mathbb{R}$,
    there exists a unique open interval $D_\mu\subset\mathbb{R}$,
    and a unique real analytic function $R_\mu:D_\mu\to\mathbb{R}$ such that the following holds. For any $R$-transform function $\tilde{R}:D\to\mathbb{C}$ of $\mu$ with $0 \in D$, it holds that $U_{\tilde{R}}\subset D_\mu$ and $R_\mu|_{U_{\tilde{R}}} = \tilde{R}|_{U_{\tilde{R}}}$ (recall \eqref{eq:U_R_def}). Furthermore, $\hat{D}_\mu \subset D_\mu$, and $R_\mu$ is a real analytic extension of $\hat{R}_\mu$.
\end{lemma}

\begin{proof}
Recall that for any compactly supported $\mu$, there is an $s>0$ large enough so that $G_\mu|_{B(0,s)^c}$
is invertible (see Section \ref{subsec:R_trans_complex_analytic}). If $u$ is sufficiently close to zero, then $\hat{R}_\mu(u)+u^{-1}\in B(0,s)^c\cap\mathbb{R}\cap(\overline{\supp}\,\mu)^c$.
For such $u$
\begin{equation}
    G_\mu|_{B(0,s)^c}(-\hat{R}_\mu(u)-u^{-1})
    =
    G_\mu|_{\mathbb{R}\cap(\overline{\supp}\,\mu)^c}(-\hat{R}_\mu(u)-u^{-1})
    \overset{\eqref{eq:def_real_R_trans_hat_sec4}}{=}
    u.
\end{equation}
The invertibility of $G_\mu|_{B(0,s)^c}$ then implies that $(G_\mu|_{B(0,s)^c})^{\left\langle -1\right\rangle }(-u)=(G_\mu|_{\mathbb{R}\cap(\overline{\supp}\,\mu)^c})^{\left\langle -1\right\rangle }(-u).$
Thus by \eqref{eq:R_trans_Ginv_complex_nbhood} and \eqref{eq:def_real_R_trans_hat_sec4}, it holds for any $R$-transform function $\tilde{R}$
that $U_{\tilde{R}}\cap\hat{D}_\mu$ is non-empty and 
\begin{equation}\label{eq:Rhat_and_R_tilde}
    \tilde{R}|_{U_{\tilde{R}}\cap\hat{D}_\mu}=\hat{R}_\mu|_{U_{\tilde{R}}\cap\hat{D}_\mu}.
\end{equation}

For any $R$-transform function $\tilde{R}$, consider the functions $\Re\,\tilde{R}|_{U_{\tilde{R}}}:U_{\tilde{R}}\to\mathbb{R}$
and $\Im\,\tilde{R}|_{U_{\tilde{R}}}:U_{\tilde{R}}\to\mathbb{R}$. Since
$\tilde{R}$ is complex analytic, these functions are real analytic.
It follows from \eqref{eq:Rhat_and_R_tilde} and \eqref{eq:def_real_R_trans_hat_sec4} that $\tilde{R}$ is real-valued in a interval
of $\mathbb{R}$ containing $0$, so real analyticity implies that
$\Im\,\tilde{R}|_{U_{\tilde{R}}}=0$.

Now, a collection of real analytic functions defined on open intervals
containing $0$, and agreeing on the intersection of their domains, have a unique maximal real analytic extension. Thus there
exists a unique $D_\mu\subset\mathbb{R}$ and a unique real analytic
function $R_\mu:D_\mu\to\mathbb{R}$, such that if $f:D\to\mathbb{R}$
is either a real analytic extension of $\tilde{R}|_{U_{\tilde{R}}}$
for an $R$-transform function $\tilde{R}$ of $\mu$, or a real analytic
extension of $\hat{R}_\mu$, then holds that $D\subset D_\mu$
and $f=R_\mu|_{D}$. This completes the proof.
\end{proof}

Since the above lemma implies Lemma \ref{lem:real_R_trans_extend}, it only remains to prove Lemma \ref{lem:real_R_trans_add_cond} about the real $R$-transform and the free additive convolution.

\subsection{Free additive convolution in terms of real $R$-transform}
Next we prove Lemma \ref{lem:add_conv_sec4} about the real $R$-transform and the free additive convolution. Note that Lemma \ref{lem:add_conv_sec4} implies Lemma \ref{lem:real_R_trans_add_cond}.

\begin{lemma}\label{lem:add_conv_sec4}
    If $\mu$ and $\nu$ are probability measures on $\mathbb{R}$ with compact support
    and respective $R$-transforms given by the power series $\overline{R}_\mu,\overline{R}_{\nu}$,
    then there is a unique probability measure $\rho$ with compact support
    whose $R$-transform is the power series $\overline{R}_\mu+\overline{R}_{\nu}$.
    Furthermore the real $R$-transforms of $\mu,\nu,\rho$ satisfy $D_\mu\cap D_{\nu}\subset D_\rho$
    and
    \begin{equation}\label{eq:add_conv_R_trans_real_sec4}
        R_\rho(u)=R_\mu(u)+R_\nu(u)\quad\text{for all}\quad u\in D_\mu\cap D_{\nu}.
    \end{equation}
\end{lemma}

\begin{proof}
    We appeal to \cite[Theorem 28, Chapter 3]{speicher}, which states that the sum of two $R$-transform
    functions of probability measures $\mu,\nu$ with finite variance
    (and possibly non-compact support) is the $R$-transform function corresponding to a unique probability
    measure. The domains of the $R$-transform functions in the statement
    of \cite[Theorem 28]{speicher} are open disks in the half-plane with positive imaginary part which are tangent to $\mathbb{R}$ at $0$ (as can be surmised by inspecting the proof, see e.g. its last two sentences). This is in contrast with the $R$-transform functions that Lemma \ref{lem:real_R_trans_extend_sec_4} 
    relates to the real $R$-transform, whose complex domains include $0$.

    By \cite[Theorem 26, Chapter 3]{speicher}, any probability measure $\mu$ on $\mathbb{R}$ with finite variance $\sigma^2$ has a $R$-transform function with complex domain $B(-i(4\sigma)^{-1},(4\sigma)^{-1})$. Denote it and its domain by $\check{R}:\check{D}_\mu \to \mathbb{C}$.
    By \cite[Theorem 28, Chapter 3]{speicher}, for any $\mu,\nu$ with finite variance there exists a unique probability measure $\rho$ s.t. $\rho$ has an $R$-transform function $\tilde{R}_\rho:D\to\mathbb{C}$ on a domain  $D$ contained in $\check{D}_\mu \cap \check{D}_\nu$, and $\tilde{R}_\rho = \tilde{R}_\mu+\check{R}_{\nu}$ on $D$. Furthermore, $\rho$ has finite variance, $\check{D}_\rho \subset \check{D}_\mu \cap \check{D}_\nu$ (in fact the variance of $\rho$ is the sum of the variances of $\mu$ and $\nu$), and
    \begin{equation}\label{eq:free_add_conv_R_trans_func_fin_var_speicher}
        \check{R}_\rho
        =
        \tilde{R}_\mu+\check{R}_{\nu} \text{ on } \check{D}_\rho.
    \end{equation}

    Since our $\mu,\nu$ also have compact support, so does $\rho$ by \eqref{eq:add_conv_supp_ineq}. Thus all three probabilities also have a $R$-transform functions $\tilde{R}_\mu,\tilde{R}_\nu,\tilde{R}_\rho$ with complex domain including $0$ (see Section \ref{subsec:R_trans_complex_analytic}). Since $\check{D}_\mu,\check{D}_\nu,\check{D}_\rho$ are tangent to $\mathbb{R}$ at $0$, they have a non-empty open intersection with the domains of $\tilde{R}_\mu,\tilde{R}_\nu,\tilde{R}_\rho$, respectively.
    Therefore complex analyticity and \eqref{eq:free_add_conv_R_trans_func_fin_var_speicher} implies that    \begin{equation}\label{eq:free_add_conv_R_trans_func_fin_compact_supp}
        \tilde{R}_\rho=\tilde{R}_\mu+\tilde{R}_{\nu} \text{ in a neighbourhood of } 0.
    \end{equation}
    By the discussion in Section \ref{subsec:R_trans_complex_analytic}, this implies that the $R$-transforms as power series satisfy $\overline{R}_\rho = \overline{R}_\mu + \overline{R}_\nu$.
    Also, $\tilde{R}_\rho|_{U_{\tilde{R}_\rho}}=\tilde{R}_\mu|_{U_{\tilde{R}_\mu}}+\tilde{R}_{\nu}|_{U_{\tilde{R}_{\nu}}}$ in a real open interval containing $0$, which by real analyticity implies \eqref{eq:add_conv_R_trans_real_sec4}.
\end{proof}

This completes the proof of Lemma \ref{lem:real_R_trans_extend} and Lemma \ref{lem:real_R_trans_add_cond}, which are the only facts about the $R$-transform used in the proofs of Theorem \ref{TH:main} and Theorem \ref{thm:add_conv_supp_end_point}.

\section{Stronger invertibility result for Stieltjes transform}\label{sec:stronger_stieltjes_invert}

We have also obtained a result on the invertibility of $G_\mu$ as a holomorphic function with domain $\mathbb{C}\setminus \supp \mu$, which improves on \cite[Theorem 17 (i), (iii), Chapter 3]{speicher}. The latter were used in Section \ref{sect:R_transform_theory}. \cite[Theorem 17 (i)]{speicher} states that
\begin{equation}\label{eq:speicher_thm17(i)}
    \text{
    $G_\mu|_{B(0,4r)^c}$ is injective,
    $\quad$
    where
    $\quad$
    $r=\max(|\supp_-\mu|,|\supp_+\mu|)$,}
\end{equation}
and \cite[Theorem 17 (iii)]{speicher} states that
\begin{equation}
    R_\mu(t) := G_\mu^{\inv}(-t) -t^{-1} \text{ is holomorphic on } B(0,(6r)^{-1}),
\end{equation}
for the same $r$, where $R_\mu(0)$ is defined by continuity. Our stronger version is the following.





\begin{theorem}\label{thm:stieltjes_global_invert}
    Let $\mu$ be a compactly supported probability measure on $\mathbb{R}$.
    \begin{enumerate}[label=\alph*)]
        \item \label{item:G_injective[thm:stieltjes_global_invert]} $G_\mu$ restricted to the open set $B = \{ z \in \mathbb{C}: \Re z\not\in[\supp_-\mu,\supp_+\mu]\}$ is injective.
        \item \label{item:R_holo[thm:stieltjes_global_invert]} Let $D = G_\mu(B)$. The function $R_\mu$ defined by 
        \be\label{eq:defRII}
            R_\mu(t) := (G_\mu|_B)^{\inv}(-t) - t^{-1}
        \ee
        for $t\ne0$ and for $t=0$
        by continuity, is holomorphic in $D \cup \{0\}$, which is an open domain.
    \end{enumerate}
\end{theorem}


The proof of Theorem \ref{thm:stieltjes_global_invert} \ref{item:G_injective[thm:stieltjes_global_invert]} starts by using \eqref{eq:speicher_thm17(i)}, and then employs a global inverse function theorem (Lemma \ref{lemma:italia}) to obtain the stronger version. Specifically, \eqref{eq:speicher_thm17(i)} implies that if a measure $\mu$ has compact support, then $G_\mu$ restricted to the set
\begin{equation}\label{eq:speicher}
        U = \mathbb{C} \setminus \{ z = u + iv \,:\, |u| \le R \,\,\text{ and } |v|\leq R \}
\end{equation}
is bijective, where
\begin{equation}\label{eq:speicher_ray}
    R:= 4\max(|\supp_+\mu|,|\supp_-\mu|,1).
\end{equation}
To extend the domain of injectivity up to the support of $\mu$, we use the following global inverse function theorem, whose proof can be found for instance in \cite[page 593]{FMS}. Note that it only requires a domain which is connected, rather than simply connected as in the Hadamard global inverse function theorem.

\begin{lemma}\label{lemma:italia}
    Let $A$ be an open subset of $\R^d$, and $f\,:A\,\mapsto\R^d$ be a $C^1$ function locally invertible on a connected bounded domain $D\subset A$. Then $f$ restricted to $D$ is (globally) invertible if it is injective on the boundary of $D$. 
\end{lemma}

The next lemma checks local invertibility by studying points at which the derivative vanishes. 

\begin{lemma}\label{lem:G_loc_inv}
    Let $\mu$ be a measure on $\mathbb{R}$ with compact support and $R>0$ as in \eqref{eq:speicher_ray}. All $z \in \mathbb{C}$ such that $G_\mu' (z) =0 $ are contained in     \begin{equation}\label{eq:zeroset}
        \{ z = u + iv \,:\, u \in (\supp_-\mu ,\supp_+\mu+)\,\,\text{and } |v|\leq |\supp\mu|\}.
    \end{equation}
\end{lemma}
\begin{proof}
    Note that 
    \begin{equation}
        G_\mu'(z) = \int\frac{\mu(d\l)}{(\l - z)^2}.
    \end{equation}
    For real $z$ this is clearly positive provided it is well-defined, which is always the case for
    $z \in \R \setminus [\supp_{-} \mu,\supp_{+} \mu]$.
    Turning to complex $z=u+iv,v\ne0$, we compute
    \begin{equation}
        G_\mu'(z)
        =
        \int\frac{\mu(d\l)}{(\l-z)^2}
        =
        \int\frac{(\l-\bar z)^2}{|(\l-z)^2|^2}\mu(d\l)
        =
        \int\frac{(\l-u)^2-v^2}{|\l-z|^4}\mu(d\l)
        +
        2iv\int\frac{\l - u}{|\l-z|^4}\mu(d\l)\,.
    \end{equation}
    Note that if the imaginary part vanishes, i.e. $\int\frac{\l - u}{|\l-z|^4}\mu(d\l) = 0$,
    then it must hold that $u \in (\supp_{-} \mu, \supp_{+} \mu)$.
    Furthermore, for any $u \in [\supp_{-} \mu,\supp_{+} \mu]$ it holds that
    \be
        \int\frac{(\l-u)^2}{((\l-u)^2+v^2)^2}\mu(d\l)
        \le
        |\supp\mu|^2
        \int\frac{\mu(d\l)}{((\l-u)^2+v^2)^2}\,.
    \ee
    Therefore if $\Re(G_\mu'(z)) = 0$, then $|v| \le |\supp \mu|$. This proves the claim.
\end{proof}

We will apply Lemma \ref{lemma:italia} to a domain that excludes a set containing \eqref{eq:zeroset}, namely the set $B_1$ of the next lemma. The lemma shows that $G_\mu$ restricted to its boundary $\partial B_1$ is injective, as required to use Lemma \ref{lemma:italia}.

\begin{lemma}\label{lem:B1}
    Let $\mu$ be compactly supported.
    For all $\varepsilon>0$ small enough and $R'>R \e^{-1}$,  $G_\mu$ is injective when restricted to the set 
    \begin{equation}\label{eq:B1}
        B_1 =
        \{ z = u + iv \,:\, u \in [\supp_{-} \mu - \varepsilon ,\supp_+ \mu + \varepsilon]\,\,\text{ and } |v|\leq R' \}\,.  
%
    \end{equation}
\end{lemma}
\begin{proof}

    Let us take $R'>R$. We decompose the boundary as
    $ \partial {B_1} = S_1 \cup  S_2 \cup S_3 $
    for    \begin{equation}
        \begin{array}{ccl}
        S_1 & = & \{z\in\partial B_1:\Re z\in(\supp_{-}\mu-\varepsilon,\supp_{+}\mu+\varepsilon),|\Im z|=R'\},\\
        S_2 & = & \{z\in\partial B_1:\Re z\in\{\supp_{-}\mu-\varepsilon,\supp_{+}\mu+\varepsilon\},R<|\Im z|\le R'\},\\
        S_3 & = & \{z\in\partial B_1:\Re z\in\{\supp_{-}\mu-\varepsilon,\supp_{+}\mu+\varepsilon\},|\Im z|\le R\}.
        \end{array}
    \end{equation}
   We now verify that the restrictions $G_\mu|_{\partial S_i}$ are injective for $i=1,2,3$, and that
    $G_\mu(S_i) \cap G_\mu(S_j) = \emptyset$ for all $i\ne j$. Together these claims imply that $G_\mu|_{\partial B_1}$ is injective.

    Since $R'>R$, the set $S_1\cup S_2$ is contained in $U$ from \eqref{eq:speicher}. Thus, by \cite[Theorem 17, (i)  Chapter 3]{speicher}, $G_\mu|_{S_1\cup S_2}$ is injective.
    Next, for all $z = u + iv$ with $u\not\in [\supp_{-} \mu, \supp_{+} \mu]$ we have that
    \be\label{eq:segments}
        \Re(G_\mu(u+iv))
        =
        \int\frac{\l - u}{(\l-u)^2+v^2}\mu(d\l)
    \ee
    is positive for $u > \supp_{+} \mu$ and negative for $u < \supp_{-} \mu$. Furthermore for fixed $u$ the real part \eqref{eq:segments} is decreasing for $v \ge 0$, where $\Im\, G_\mu\geq 0$ and it is increasing for $v \le 0$, where $\Im\, G_\mu\leq 0$. This proves that $G_\mu( S_2 ) \cap G_\mu( S_3 )=\emptyset$, and that $G_\mu|_{S_3}$ is injective.
    It only remains to ensure that $G_\mu( S_1 ) \cap G_\mu( S_3 )=\emptyset$. For this we use the bounds
    \begin{equation}\label{eq:G_bounds}
        \frac{
        {\displaystyle 
        \min\left( \inf_{x\in\supp\mu}\left|x-\Re z\right|,|\Im z|\right)}
        }{{\displaystyle \sup_{x\in\supp\mu}\left|x-z\right|^2}}
        \le
        |G_\mu(z)|
        \le
        \frac{1}{\displaystyle{\inf_{x\in\supp\mu}}\left|x-z\right|},
    \end{equation}
    which follow by the definition \eqref{eq:defSti} of $G_\mu$. The above bounds gives $\inf_{z \in S_3} |G_\mu(z)| > \e/R$ and $\sup_{z \in S_1} |G_\mu(z)| < 1/{R'}$. If $R'>R\e^{-1}$, then $G_\mu( \partial S_1) \cap G_\mu( \partial S_3) = \emptyset$, proving the injectivity of $G_\mu|_{\partial B}$. Finally, by the lower bound in \eqref{eq:G_bounds}, we have that $\inf_{z\in \partial B_1}|G_\mu(z)| \geq \e/R'^2\geq \e^3/R^2$.
\end{proof}

We are now ready to combine Lemmas \ref{lemma:italia}, \ref{lem:G_loc_inv} and \ref{lem:B1} to prove Theorem \ref{thm:stieltjes_global_invert}.

\begin{proof}[Proof of Theorem \ref{thm:stieltjes_global_invert}]

    {\bf \ref{item:G_injective[thm:stieltjes_global_invert]}} Let $\varepsilon>0$ be small enough and $R'>0$ large enough, so that Lemma \ref{lem:B1} implies that $G_\mu$ restricted to $\partial B_1$ is injective, while Lemma \ref{lem:G_loc_inv}
     implies that all critical points of $G_\mu$ lie in the interior of $B_1$.
    
     By \eqref{eq:G_bounds} the quantity $|G_\mu(z)|$ is bounded below on $\partial B_1$, and $|G_\mu(z)|\to 0 $ as $|z| \to \infty$. Thus for $R''$ large enough it holds for
     \begin{equation}
            B_2:=\{ z = u + iv \,:\, 
        |u| \le R'',\,\,|v|\leq R''\},
    \end{equation}
    that $\sup_{z \in \partial B_2} |G_\mu(z)| < \inf_{z \in \partial B_1} |G_\mu(z)|$. Thus in particular $G_\mu( \partial B_1) \cap G_\mu( \partial B_2) = \emptyset $. Furthermore if $R''$ is large enough, then $\partial B_2 \subset U$ for $U$ from \eqref{eq:speicher}, so $G_\mu$ restricted to $\partial B_2$ is injective.
    
    Thus $G_\mu(z)$ restricted to $\partial B_i$ is injective for $i=1,2$, and $G_\mu(\partial B_1)\cap G_\mu(\partial B_2) = \emptyset$, implying that $G_\mu$ restricted to $\partial (B_2 \setminus B_1)$ is injective. Furthermore, there are no critical points of $G_\mu$ in $B_2 \setminus B_1$ (since they are all in $B_1$), so $G_\mu$ is locally invertible on $B_2 \setminus B_1$. Therefore Lemma \ref{lemma:italia} implies that $G_\mu$ restricted to $B_2 \setminus B_1$ is injective.
    
    In the above argument we can take $R''$ arbitrarily large, which implies that in fact $G_\mu$
    restricted to $B_1^c$ is injective. Choosing $\varepsilon>0$
    arbitrarily small, and thus $R'$ arbitrarily large, implies that
    $G_\mu$ restricted to $\{ z \in \mathbb{C}: \Re z \notin [-\supp_{-} \mu, \supp_{+} \mu]\}$ is injective. 

    {\bf \ref{item:R_holo[thm:stieltjes_global_invert]}} Part \ref{item:G_injective[thm:stieltjes_global_invert]} implies that $(G_\mu|_B)^{\inv}$ is well-defined and holomorphic in the domain $D$.
    Moreover, we have
    \been{
        \lim_{|z| \to \infty} G_\mu(z)=0\,,
    }
    and clearly $G_\mu(z)\neq 0$ for any $z$ such that $|z|<\infty$, since $G_\mu$ is injective. This implies that $0 \notin D$. Thus the function $R_\mu$ is also holomorphic in $D$. Moreover, it has a removable singularity at the origin as shown in \cite[Theorem 17, Chapter 3]{speicher}. Therefore $R_\mu$ as defined is indeed holomorphic on $D \cup \{0\}$.
\end{proof}

\end{document}